\documentclass[12pt]{amsart}

\usepackage{amsmath,amsthm,amscd,amsfonts,amssymb}

\makeindex

\usepackage[pagebackref,colorlinks=true,linkcolor=blue,citecolor=blue]{hyperref}

\newtheorem{thm}{Theorem}[section]
\newtheorem{cor}[thm]{Corollary}

\newtheorem{lem}[thm]{Lemma}

\newtheorem{prop}[thm]{Proposition}
\theoremstyle{remark}
\newtheorem*{rem}{Remark}

\numberwithin{equation}{section}

\newcommand{\A}{\mathbb{A}}
\newcommand{\GL}{\mathrm{GL}}

\newcommand{\ZZ}{\mathbb{Z}}

\newcommand{\Gal}{\mathrm{Gal}}

\newcommand{\QQ}{\mathbb{Q}}

\newcommand{\lto}{\longrightarrow}

\newcommand{\OO}{\mathcal{O}}
\newcommand{\CC}{\mathbb{C}}
\newcommand{\RR}{\mathbb{R}}

\newcommand{\pp}{\mathfrak{p}}

\newcommand{\End}{\mathrm{End}}

\newcommand{\quash}[1]{}

\theoremstyle{definition}
\newtheorem{defn}[thm]{Definition}

\renewcommand{\bar}{\overline}

\numberwithin{equation}{subsection}

\renewcommand{\hat}{\widehat}

\newcommand{\p}{{\sf p}}

\newcommand{\Res}{\mathrm{Res}}

\begin{document}

\title[Relative trace formulae and unitary groups]{Twisted relative trace formulae\\ with a view towards unitary groups}
\author{Jayce R.~ Getz}
\address{Department of Mathematics\\
Duke University\\
Durham, NC 27708-0320}
\email{jgetz@math.duke.edu}
\author{Eric Wambach}
\address{Department of Mathematics \\ California
Institute of Technology \\ Pasadena, CA
91125}\email{wambach@its.caltech.edu}
\thanks{The first author is thankful for partial support provided by NSF grant DMS 0703537.  Commutative diagrams were
typeset using Paul Taylor's diagrams package.}

\begin{abstract}
We introduce a twisted relative trace formula which simultaneously generalizes the twisted trace formula of Langlands et.~al.~(in the quadratic case)
 and the relative trace formula
of Jacquet and Lai \cite{JacquetLai}.  Certain matching statements
relating this twisted relative
trace formula to a relative trace formula are also proven
(including the relevant fundamental lemma in the ``biquadratic case'').
Using recent work of Jacquet, Lapid and their
collaborators \cite{JacquetKlII} and the Rankin-Selberg integral
representation of the Asai $L$-function (obtained by Flicker using
 the theory of Jacquet, Piatetskii-Shapiro, and
 Shalika \cite{FlickerDist}), we give the following application:
Let $E/F$ be a totally real quadratic extension
 with $\langle \sigma \rangle=\Gal(E/F)$, let $U^{\sigma}$ be a quasi-split unitary group with respect
 to a CM extension $M/F$, and let $U:=\mathrm{Res}_{E/F}U^{\sigma}$.
 Under suitable local hypotheses, we show that a cuspidal
 cohomological automorphic representation $\pi$ of $U$ whose Asai $L$-function has a pole at the edge of the critical strip
 is nearly equivalent to a cuspidal
 cohomological automorphic representation $\pi'$ of $U$ that is
 $U^{\sigma}$-distinguished in the sense that there is a form in the space of $\pi'$ admitting a nonzero period over $U^{\sigma}$.  This provides cohomologically nontrivial cycles of middle dimension on unitary Shimura varieties analogous to those on Hilbert modular surfaces studied by Harder, Langlands, and Rapoport \cite{HLR}.
\end{abstract}

\maketitle

\tableofcontents

\section{Introduction}

\subsection{Distinction} Let $F$ be a number field and let $G$ be a reductive $F$-group with automorphism $\sigma$ over $F$ of order $2$.
Write $G^{\sigma} \leq G$ for the reductive subgroup whose points in an $F$-algebra $R$ are given by
$$
G^{\sigma}(R):=\{g \in G(R): g=g^{\sigma}\}.
$$
Let $\pi$ be a cuspidal unitary automorphic representation of $G(\A_F)$. For smooth $\phi$ in the $\pi$-isotypic subspace of the cuspidal subspace $L^2_{0}(G(F) \backslash {}^1G(\A_F)) \leq L^2(G(F) \backslash {}^1G(\A_F))$, the period integral
\begin{align} \label{period}
\mathcal{P}_{G^{\sigma}}(\phi):&=\int_{G^{\sigma}(F) \backslash (G^{\sigma}(\A_F) \cap {}^1G(\A_F))} \phi(x) dx
\end{align}\index{$\mathcal{P}_{G^{\sigma}}(\phi)$}converges (see \cite[Proposition 1, \S 2]{AGR}).
Here ${}^1G(\A_F) \leq G(\A_F)$ is the Harish-Chandra subgroup
(see \S \ref{HC-subgroup}) and $dx$ is induced by a choice of Haar
measure on $G^{\sigma}(\A_F) \cap {}^1G(\A_F)$.
  If $\mathcal{P}_{G^{\sigma}}(\phi) \neq 0$  for some smooth $\phi$ in the $\pi$-isotypic subspace of $L^2_{0}(G(F) \backslash {}^1G(\A_F))$, we say that $\pi$ is \textbf{$G^{\sigma}$-distinguished}, or simply distinguished.

In \cite{HLR}, Harder, Langlands, and Rapoport  introduced the notion of distinction as a tool for
producing cohomologically nontrivial cycles on Hilbert modular surfaces. Their work was based on the following
observation: If $G(F \otimes_{\QQ}\RR) \cong G^{\sigma}(F \otimes_{\QQ}\RR) \times G^{\sigma}(F \otimes_{\QQ} \RR)$ and if $\pi$ has nonzero cohomology with coefficients in $\CC$ (see Definition \ref{cohom-defn}) and is $G^{\sigma}$-distinguished then
there is a differential form $\omega$ on a locally symmetric space attached to $G$ that is
$\pi^{\infty}$-isotypic under the action of certain Hecke correspondences such that $\omega$
has a nonzero period over a sub-symmetric space defined by $G^{\sigma}$.   Harder, Langlands and
Rapoport investigated the case where $G=\mathrm{Res}_{\QQ(\sqrt{d})/\QQ}\GL_2$ for a
square-free positive integer $d$ and $\sigma$ is the automorphism of $G$ induced by the
nontrivial element of $\Gal(\QQ(\sqrt{d})/\QQ)$.
As a consequence of their investigation they proved the Tate conjecture for the ``non-CM''
part of the cohomology of Hilbert modular surfaces.
The key fact that allowed them to link Galois invariant elements in an \'etale cohomology group to the existence of cycles was a characterization of $\GL_2$-distinguished cuspidal automorphic representations of $\mathrm{Res}_{\QQ(\sqrt{d})/\QQ}\GL_2(\A_{\QQ})$ in terms of a certain invariance property of $\pi$ under $\sigma$.

The characterization of $\GL_{2/\QQ}$-distinguished representations used in \cite{HLR} can be proven using a modification of the Rankin-Selberg method \cite{Asai}. More generally, let $E/F$ be a quadratic extension of number fields.
A variation of Jacquet, Piatetskii-Shapiro and Shalika's interpretation of the Rankin-Selberg method similarly gives some characterization of those automorphic representations of $\mathrm{Res}_{E/F}\GL_n$ distinguished by $\GL_{n/F}$ in terms of Asai $L$-functions \cite{FlickerDist}.  However, in most situations a characterization of the automorphic representations of $G$ distinguished by $G^{\sigma}$ in terms of arithmetic properties of the automorphic representation or analytic properties of $L$-functions attached to the automorphic representation is unknown.

To address these situations, Jacquet has introduced relative trace formulae in order to provide a (mostly conjectural) conceptual understanding of distinction in great generality.  Though the theory still is far from complete, many interesting results are known, even in higher rank.  For example, if $E/F$ is a quadratic extension as above, Jacquet, Lapid, and their collaborators have provided a characterization of those representations of $\mathrm{Res}_{E/F}\GL_n(\A_F)$ that are distinguished by a  quasi-split unitary group in $n$-variables attached to $E/F$ \cite{JacquetKlII} \cite{Jacquetqs}.

\subsection{A result on unitary groups} Since the locally symmetric spaces attached to $\GL_n$ are never hermitian for $n>2$, the higher rank results of Jacquet and Lapid mentioned above cannot be directly applied to the study of the Tate conjecture for Shimura varieties as in \cite{HLR}.  Despite this, in this paper we provide one means of relating distinction of automorphic representations of $\GL_{n}$ to distinction of automorphic representations on unitary groups (including those defining hermitian locally symmetric spaces).  We now state an application.  Let $E/F$ be a quadratic extension of totally real number fields, let $\langle \sigma \rangle=\Gal(E/F)$, and  let $M/F$ be a CM extension.  Let $U^{\sigma}$ be a quasi-split unitary group in $n$ variables attached to the extension $M/F$ and let $U:=\mathrm{Res}_{E/F}U^{\sigma}$.  We have the following result:

\begin{thm} \label{intro-thm2} Let $\pi$ be a cuspidal automorphic representation of $U^{\sigma}(\A_E)=U(\A_F)$.
Suppose that $\pi$ satisfies the following assumptions:
\begin{enumerate}
\item There is a finite-dimensional representation $V$ of $U_{F_{\infty}}$ such that $\pi$ has
nonzero cohomology with coefficients in $V$.
\item There is a finite place $v_1$ of $F$ totally split in $ME/F$ such that $\pi_{v_1}$ is supercuspidal.
\item There is a finite place $v_2 \neq v_1$ of $F$ totally split in $ME/F$ such that $\pi_{v_2}$ is in the discrete series.
\item For all places $v$ of $F$ such that $ME/F$ is ramified and $M/F$, $E/F$ are both nonsplit at $v$ the weak base change $\Pi$ of $\pi$ to $U(\A_{ME}) \cong \GL_{n}(\A_{ME})$ has the property that $\Pi_v$ is relatively $\tau$-regular.
\end{enumerate}
The representation $\pi$ admits a weak base change $\Pi$ to $\GL_n(\A_{ME})$. If the partial Asai $L$-function $L^S(s,\Pi;r)$ has a pole at $s=1$ then some cuspidal automorphic representation
$\pi'$ of $U(\A_F)$ nearly equivalent to $\pi$ is $U^{\sigma}$-distinguished.
 Moreover, we can take $\pi'$ to have nonzero cohomology
with coefficients in $V$.
\end{thm}
\noindent Here we say two automorphic representations are nearly equivalent if their local factors are isomorphic at almost all places.
Moreover, $S$ is a finite set of places of $M$ including all infinite places, all places where $ME/M$ is ramified, and all places below places of $ME$ where $\Pi$ is ramified, and
$$
r:{}^L\mathrm{Res}_{ME/M}\GL_n \lto \GL_{n^2}(\CC)
$$
is the twisted tensor representation attached to the extension $ME/M$ (see \cite{FlickerDist} or
\cite[\S 6]{Ramakrishnan} for the definition of this representation).  We are also using the fact that $L^S(s,\Pi;r)$ has a meromorphic continuation to a closed right half plane containing $s=1$ (see \cite{FlickerDist} and \cite{FlickerZin}).

Theorem \ref{intro-thm2} is proved as Corollary \ref{main-cor} below.  For the definition of a relatively $\tau$-regular representation, see \S \ref{sec-sph} below. For other unexplained notation and terminology, see \S \ref{ssec-apps}.
We indicate the role of some of the assumptions in the theorem:

\noindent \emph{Remarks.}

\begin{enumerate}
\item
The assumption that $\pi$ is cohomological allows us to use the proof of \cite[Theorem 3.1.4]{HarLab} to conclude that the weak base change $\Pi$ exists.

\item The assumptions at the places $v_1$ and $v_2$ are made for two reasons.  First, they allow us to apply known results on base change liftings from \cite{HarLab}.  Second,
they allow us to use the simple form of the relative
trace formula proved in \cite{Hahn}.  If a relative version of the
Arthur-Selberg trace formula were known, together with a theory of relative endoscopy,
  then it would probably be possible to remove the
  restrictions on $v_1$ and $v_2$.

\item Assumption (4) allows us to use test functions supported on the relatively $\tau$-regular semisimple set at places where we have not proven satisfactory matching statements for arbitrary functions.  The definition of a relatively $\tau$-regular admissible representation is given in \S \ref{sec-sph} below.

\item The assumption that $U^{\sigma}$ is quasi-split
is relaxed somewhat in Corollary \ref{main-cor} below.

\end{enumerate}

Given all of these simplifying assumptions, one might be surprised that this paper is still rather long.  As an explanation, we point out that to prove our theorem we have to provide, from scratch, analogues of
    many standard constructions in the theory of the usual trace formula (for example, we need analogues of many of the results of \cite{KottRatConj}, \cite{KottStCusp},  \cite{KottEllSing}, \cite{KottTama}).  Unfortunately, this takes some space.  We have tried to streamline the presentation by not working in the greatest possible generality.

\subsection{Outline of a proof}

Let $\pi$ be a cuspidal automorphic representation of $U(\A_F)$ satisfying the hypotheses of Theorem \ref{intro-thm2} and
admitting a weak base change $\Pi$ to
$\mathrm{Res}_{ME/F}\GL_{n}(\A_{F})$ (see \S \ref{defn-wbc} for our
conventions regarding weak base change).  Let $\tau$ be the automorphism of $G:=\mathrm{Res}_{ME/F}\GL_n$ fixing $U$, and let $\theta=\sigma \circ \tau$.  We prove Theorem \ref{intro-thm2} in two steps.  First, we prove that if $\Pi$ is
distinguished by $\mathrm{Res}_{M/F}\GL_{n}$ and the quasi-split unitary group $G^{\theta}$ attached to $EM/(EM)^{\langle \theta\rangle}$ then
some automorphic representation
 $\pi'$ nearly equivalent to $\pi$ is $U^{\sigma}$ distinguished.
We prove this by exhibiting identities of the form
\begin{align*}
\sum_{\pi} \sum_{\phi \in \mathcal{B}(\pi)} \mathcal{P}_{U^{\sigma}}(\pi(\Phi^1)\phi)\overline{\mathcal{P}_{U^{\sigma}}(\phi)}=2\sum_{\Pi}\sum_{\phi_0 \in \mathcal{B}(\Pi)}\mathcal{P}_{G^{\sigma}}(\Pi(f^1)\phi_0)\overline{\mathcal{P}_{G^{\theta}}(\phi_0)}
\end{align*}
 for a sufficiently large set of functions $f \in C_c^{\infty}(G(\A_F))$ and $\Phi \in C_c^{\infty}(U(\A_F))$ that match in the sense of Definition \ref{defn-match}.  Here the sum on the left is over cuspidal automorphic representations of $U(\A_F)$, the sum on the right is over cuspidal automorphic representations of $G(\A_F)$, and $\mathcal{B}(\pi)$ is an orthonormal basis of the $\pi$-isotypic subspace of $L^2(U(F) \backslash U(\A_F))$ consisting of smooth vectors; $\mathcal{B}(\Pi)$ is defined similarly.  We refer the reader to Proposition \ref{prop-spect-compar} for more details.

 Second, the fact that $\Pi$ is a base change implies that $\Pi^{\tau} \cong \Pi$.  By work of
Flicker and his collaborators, our assumption on the pole of the Asai $L$-function implies that $\Pi$ is distinguished by
$\mathrm{Res}_{M/F}\GL_{n}$ and moreover that $\Pi^{\vee} \cong \Pi^{\sigma}$.  This implies that $\Pi^{\theta \vee} \cong \Pi$, and hence by work of Jacquet,
 Lapid, and their collaborators this in turn implies that $\Pi$ is distinguished by $G^{\theta}$. This second step is contained in \S \ref{ssec-apps} and we will say no more about
it in the introduction.

The proof of the first step is the heart of this paper and is based on a particular case of a general principle which we now explain.
Let $H$ be a connected reductive group, let $\langle \tau \rangle=\Gal(M/F)$, and let
$$
G:=\mathrm{Res}_{M/F}H.
$$
Let $\sigma$ be an automorphism of $H$ of order $2$ and let $H^{\sigma} \leq H$ and $G^{\sigma} \leq G$ be the subgroups fixed by $\sigma$.  Let $\theta=\sigma \circ \tau$, viewed as an automorphism of $G$, and let $G^{\theta} \leq G$ be the subgroup fixed by $\theta$. Let $\pi$ be a cuspidal automorphic representation of $H(\A_F)$.  Assume that $\pi$ admits a weak base change $\Pi$ to $G(\A_F)$.
In favorable circumstances the machinery developed in this paper together with suitable fundamental lemmas should imply that
if $\Pi$ is $G^{\sigma}$-distinguished and $G^{\theta}$-distinguished then a cuspidal automorphic representation $\pi'$ nearly equivalent to $\pi$ is $H^{\sigma}$-distinguished.

We mention two large obstacles to turning the general principle
\begin{align*}
\Pi \textrm{ distinguished by }G^{\sigma} \textrm{ and } G^{\theta}
 \Rightarrow
 \exists \textrm{ }H^{\sigma}\textrm{-distinguished }\pi' \textrm{ nearly equivalent to }\pi
\end{align*}
into a theorem in any given case.  First is the current lack of
a relative analogue of the Arthur-Selberg trace formula or a topological relative trace formula.
Second is the current lack of a theory of relative endoscopy and the fundamental lemmas that should come along with the theory.  If we had the tools of a relative Arthur-Selberg trace formula and the theory of relative endoscopy in hand, one could probably remove the annoying local restrictions in Theorem \ref{intro-thm2} and prove much more general versions of the principle indicated above.  However, we caution that if one considers involutions $\sigma$ that are not Galois involutions as we are primarily concerned with in this paper, there seem to be new phenomena lurking that we cannot explain here.  These new phenomena may make the principle false as it is stated,
  though we believe something close to it is true.

We end this outline with some speculation regarding the undefined term ``relative endoscopy.''  The theory of (twisted) endoscopy, when complete, should yield an understanding of automorphic representations of classical groups in terms of those on general linear groups via a twisted trace formula (see \cite[\S 30]{ArthurIntro}, for example).  We conjecture that a similar theory will relate distinguished representations on a classical group to representations on $\GL_n$ that are distinguished with respect to two subgroups.  We believe that Theorem \ref{intro-thm2} provides an interesting implication of these relations in the special case where endoscopy reduces to base change.

 \begin{rem}  We note that our approach to
 the relative trace formula is modeled on that exposed in \cite{JacquetLai} and
 the introduction to \cite{JLR}.  More precisely, let $E/F$ be a quadratic extension, suppose
 $H^{\sigma}=\GL_2$, $H=\mathrm{Res}_{E/F}\GL_2$, and let $\sigma$ be induced by the nontrivial automorphism
 $\sigma$ of $\Gal(E/F)$.  Then the relatively elliptic part of the trace formula developed in \cite{JacquetLai}
 is roughly the ``$\tau=1$'' case of our theory.  If we instead let $E=F \oplus F/F$ be the ``split''
 quadratic extension and again let $\sigma$ be the automorphism of $H$ induced by the nontrivial $F$-automorphism
 of $E/F$, then our trace formula comparing distinction on $G:=\mathrm{Res}_{M/F}H$ to distinction on $H$
 is roughly equivalent to the ``elliptic part'' of the trace formula comparison used to establish quadratic base
 change in \cite{LanglandsBC} (though, of course, Langlands treats general cyclic extensions).  This justifies our claim
 in the abstract that our formulae simultaneously generalize the relative trace formula and the twisted trace formula (in the quadratic case).
 See Proposition \ref{prop-spect-compar} for a precise statement regarding trace formula comparisons.
\end{rem}

We now give a synopsis of the various sections.  We fix some notation in the next section.  In
\S \ref{ssec-basic-notat},
 following \cite{JacquetLai} we
develop a notion of relative classes and relative $\tau$-classes generalizing conjugacy classes and twisted conjugacy
classes, respectively.
A norm map is also defined and studied, and in \S \ref{ssec-norm-exist}
we prove that it has the properties one would expect,
 at least in the ``biquadratic unitary'' case (see \S \ref{ssec-basic-notat}).

Section \ref{sec-matching} studies the corresponding notion of matching of stable local relative orbital integrals.  We also prove the fundamental lemma for unit elements in the ``biquadratic case'' and the fundamental lemma for
Hecke functions when various objects split (e.g. $M/F$). In \S \ref{sec-match-ram} we prove that one can find some function in $C_c^{\infty}(H(F_v))$ matching a given function in $C^{\infty}_c(G(F_v))$ for nonarchimedian $v$.  We define the notion of a relatively $\tau$-regular admissible representation in \S \ref{sec-sph}; we require this notion due to our incomplete knowledge of matching functions.

Section \ref{sec-prestab} concerns globalization of stable relative orbital integrals.  Here we follow Labesse's formulation
\cite{Lab} of
the work of Kottwitz and Shelstad (following Langlands, see \cite{KS}, \cite{LanglStab}) on the usual trace formula.  In \S \ref{sec-group-ell} we discuss the geometric expansion of the stable twisted relative trace formula.
In \S \ref{sec-rtf} we use the main theorem of \cite{Hahn} together with all of the previous work to
give a spectral expansion of the twisted relative trace
formula on $G$ and the relative trace formula on $H$, at least in the case relevant for the proof of Theorem
\ref{intro-thm2}.  As noted above,
\S \ref{ssec-apps} contains the final argument proving Theorem \ref{intro-thm2}.

\section{Selected notation and conventions}

\label{sec-notation}

\subsection{Algebraic groups}
\label{notatAlggroups} Let $G$ be an algebraic group over a field
$F$. We write $Z_G$ for the center of $G$ and $G^{\circ}$ for the
identity component of $G$ in the Zariski topology.  If $\gamma \in
G(F)$ and $G' \leq G$ is a subgroup we write $C_{\gamma,G'}$ for the
centralizer in $G'$ of $\gamma$.  If the element $\gamma$ is semisimple then we say it is
\textbf{elliptic} if $C_{\gamma,G}/Z_G$ is anisotropic.  If $\tau$ is an automorphism of $G$ and
$\gamma \in G(R)$ for some commutative $F$-algebra $R$ we write
$$
\gamma^{-\tau}:=(\gamma^{-1})^{\tau}.
$$
We write $C_{\gamma,G}^{\tau}$ for the $\tau$-centralizer of $\gamma$ in $G$; it is the reductive $F$-group
whose points in an $F$-algebra $R$ are given by
$$
C_{\gamma,G}^{\tau}(R):=\{g \in G(R):g^{-1}\gamma g^{\tau}=\gamma\}.
$$
\index{$C_{\gamma,G}^{\tau}$}A torus $T \leq G$ is said to be \textbf{$\tau$-split} if for any
commutative $F$-algebra $R$ we have $g^{\tau}=g^{-1}$ for all $g
\in T(R)$.  If $\theta$ is another automorphism of $G$, we say that $T$ is $(\tau,\theta)$-split if it is both $\tau$ and $\theta$-split.

\subsection{Ad\`eles}
The ad\`eles of a number field $F$ are denoted by $\A_F$. For a
set of places $S$ of $F$ we write $\A_{F,S}:=\A_F \cap \prod_{v \in
S}F_v$ and $\A^S_F:=\A_F \cap \prod_{v \not \in S}F_v$.  If $S$ is finite we often write $F_S:=\A_{F,S}$.  The set of
infinite places of $F$ will be denoted by $\infty$. Thus
$\A_{\QQ,\infty}=\RR$ and $\A_{\QQ}^{\infty}:=\prod_{\substack{p \in
\ZZ_{>0}
\\p \textrm{ prime}}}\QQ_p$. For an affine $F$-variety $G$ and a
subset $W \leq G(\A_F)$ the notation $W_{S}$ (resp. $W^S$) will
denote the projection of $W$ to $G(\A_{F,S})$ (resp. $G(\A^S_F)$).
If $W$ is replaced by an element of $G(\A_F)$, or if $G$ is an
algebraic group and $W$ is replaced by a character of $G(\A_F)$ or a
Haar measure on $G(\A_F)$, the same notation will be in force; e.g. if
$\gamma \in G(\A_F)$ then $\gamma_v$ is the projection of $\gamma$ to $G(F_v)$.

\subsection{Harish-Chandra subgroups} \label{HC-subgroup}
Let $G$ be a connected reductive group over a number field $F$.  We write $A_G \leq Z_G(F \otimes_{\QQ} \RR)$ \index{$A_G$} for the connected component of the real points of the largest $\QQ$-split torus in the center of $\mathrm{Res}_{F/\QQ}G$.  Here when we say ``connected component'' we mean in the real topology.  Write
$X$ for the group of $\QQ$-rational characters of $G$.  There is a morphism
\begin{align*}
HC_G:G(\A_F) \lto \mathrm{Hom}(X,\RR)
\end{align*}
defined by
\begin{align}
\langle HC_G(x),\chi\rangle =|\log(x^{\chi})|
\end{align}
for $x \in G(\A_F)$ and $\chi \in X$.  We write
\begin{align}
^{1}G(\A_F):=\ker(HC_G)
\end{align}
\index{$^{1}G(\A_F)$}and refer to it as the Harish-Chandra subgroup of $G(\A_F)$.  Note that $G(F) \leq {}^1G(\A_F)$ and
$G(\A_F)$ is the direct product of $A_G$ and $^{1}G(\A_F)$.

\section{Twisted relative classes and norm maps}
\label{sec-norm-maps}
In this section we recall the notion of relative classes in algebraic groups
with involution and extend this notion to the ``$\tau$-twisted'' case.  In \S \ref{ssec-match}
we
define stable relative classes and a relative version of the usual norm map,
which, in the original setting
of the trace formula, relates $\tau$-conjugacy
classes and conjugacy classes.  In \S \ref{ssec-norm-exist} we use
work of Borovoi to show that the norm map has particularly nice properties in the
``biquadratic unitary case.''

\subsection{Twisted relative classes}
\label{ssec-basic-notat}

Let $F$ be a number field and let $M/F$ be either a quadratic or a trivial
extension.  Let $\tau$ be the generator of $\Gal(M/F)$.
Fix, once and for all, an algebraic closure $\bar{M}$ of $M$.
We usually regard $\bar{M}$ as an algebraic closure of $F$
as well, and denote it by $\bar{F}$.  For all places $v$ of $F$,
choose once and for all an algebraic closure $\bar{F}_v$ of
the completion $F_v$ of $F$ and a compatible system of
embeddings $\bar{F} \hookrightarrow \bar{F}_v$.

Let $H$ \index{$H$} be a reductive $F$-group with (semisimple) automorphism $\sigma$\index{$\sigma$}
of order two (defined over $F$).  We let $H^{\sigma}$ \index{$H^{\sigma}$}be the reductive $F$-group whose points in
an $F$-algebra $R$ are given by
$$
H^{\sigma}(R):=\{g \in H(R):g^{\sigma}=g\}.
$$ \index{$H^{\sigma}$}We assume that
\begin{itemize}
\item $H^{\sigma}$ is connected.
\end{itemize}
This will always be the case if $H$ is semisimple and simply connected \cite[\S 8]{SteinbMem}.
Write
\begin{align}
\nonumber G:&=\mathrm{Res}_{M/F}H\\
\nonumber G^{\sigma}:&=\mathrm{Res}_{M/F}H^{\sigma}.
\end{align} \index{$G$} \index{$G^{\sigma}$}The automorphism $\sigma$ of $H$ induces an automorphism of $G$ which we will also denote by $\sigma$. Set $\theta:=\sigma \circ \tau$ \index{$\theta$}(where $\sigma$ and $\tau$ are viewed as automorphisms of $G$), and write $G^{\theta} \leq G$ for the subgroup whose points in an $F$-algebra $R$ are given by
$$
G^{\theta}(R):=\{g \in G(R):g=g^{\theta}\}.
$$\index{$G^{\theta}$}
We note that $\theta$, $\sigma$, and $\tau$ all commute.
\begin{rem} Let $E/F$ be a quadratic extension of number fields with $ME/F$ biquadratic.
The primary case of interest for us in this paper is the case where  $H=\mathrm{Res}_{E/F}H^{\sigma}$ and $\sigma$ is the automorphism of $H$ defined by the generator of $\Gal(E/F)$ (which we will also denote by $\sigma$).  We refer to this case as the \textbf{biquadratic case}.
\end{rem}

We have left actions of $H^{\sigma} \times H^{\sigma}$ on $H$ and $G^{\sigma} \times G^{\theta}$ on $G$ given by
\begin{align} \label{actionsgalore}
\nonumber (H^{\sigma} \times H^{\sigma})(R) \times H(R) &\lto H(R)\\
(g_1,g_2,g) &\longmapsto (g_1gg_2^{-1})\\
\nonumber(G^{\sigma} \times G^{\theta})(R) \times G(R) &\lto G(R)\\
\nonumber(g_1,g_2,g) &\longmapsto  (g_1gg_2^{-1})
\end{align}
for $F$-algebras $R$.
Our goal in this section is to compare the
sets of double cosets $H^{\sigma}(R) \backslash H(R) /H^{\sigma}(R)$ and
$G^{\sigma}(R) \backslash G(R) /G^{\theta}(R)$.

For this purpose, following \cite[Lemma 2.4]{Rich} we introduce
the subschemes $Q:=Q_H \subset H$\index{$Q$} and $S:=S_{G} \subset G$. \index{$S$}
They are defined to be the scheme theoretic images of
the morphisms given on points in an $F$-algebra $R$ by
{\allowdisplaybreaks \begin{align} \label{nat-maps}
\nonumber B_{\sigma}:H(R) &\lto H(R)\\
g & \longmapsto gg^{-\sigma}\\
\nonumber B_{\theta}:G(R) & \lto G(R)\\
\nonumber g & \longmapsto gg^{-\theta},
\end{align}}
\index{$B_{\sigma}$} \index{$B_{\theta}$}respectively; they are closed affine $F$-subschemes of $H$ and $G$,
respectively (see \cite[\S 2.1]{HelmWang}).
Notice that $B_{\sigma}(g_1gg_2^{-1})=g_1B_{\sigma}(g)g_1^{-1}$
and $B_{\theta}(g_1gg_2^{-1})=g_1B_{\theta}(g)g_1^{-\tau}$ for $(g_1,g_2) \in H^{\sigma}(R)^2$ (resp.~$(g_1,g_2) \in G^{\sigma} \times G^{\theta}(R)$).
There are injections
{ \allowdisplaybreaks\begin{align}
\nonumber H^{\sigma}(R) \backslash H(R) /H^{\sigma}(R) &\lto H^{\sigma}(R) \backslash Q(R)\\
g & \longmapsto g g^{-\sigma}\\
\nonumber G^{\sigma}(R) \backslash G(R)/G^{\theta}(R) &\lto G^{\sigma}(R) \backslash S(R)\\
\nonumber g & \longmapsto gg^{-\theta}
\end{align}}
where $H^{\sigma}$ acts by conjugation on $Q$ and $G^{\sigma}$ acts by $\tau$-conjugation on $S$.

For any $F$-algebra $k$ and any $\gamma \in H(k)$ (resp. $\delta \in G(k)$) write
$H_{\gamma}$ (resp. $G_{\delta}$)
for the $k$-group whose points in a commutative $k$-algebra $R$
are given by
{\allowdisplaybreaks \begin{align}
H_{\gamma}(R):&=\{(g_1,g_2)\in H^{\sigma}(R) \times H^{\sigma}(R):g_1^{-1}\gamma
g_2=\gamma\}\\
\nonumber G_{\delta}(R):&=\{(g_1,g_2) \in G^{\sigma}(R) \times G^{\theta}(R)
:g_1^{-1}\delta g_2=\delta\}.
\end{align}}\index{$H_{\gamma}$}\index{$G_{\delta}$}Write $C_{\gamma,H^{\sigma}}$ (resp. $C_{\delta,G^{\sigma}}^{\tau}$)
for the centralizer of $\gamma$ in $H^{\sigma}$ (resp. $\tau$-centralizer of $\delta$ in
$G^{\sigma}$).

\begin{lem} \label{lem-2cent} There are isomorphisms
$H_{\gamma} \tilde{\to} C_{\gamma\gamma^{-\sigma},H^{\sigma}}$ and
$G_{\delta} \tilde{\to}
C_{\delta\delta^{-\theta},G^{\sigma}}^{\tau}$
given on a $k$-algebra $R$ by
\begin{align} \label{2cent}
\nonumber H_{\gamma}(R) &\tilde{\lto} C_{\gamma\gamma^{-\sigma},H^{\sigma}}(R)\\
(g_1,g_2) &\longmapsto g_1\\
\nonumber G_{\delta}(R) & \tilde{\lto} C_{\delta
\delta^{-\theta},G^{\sigma}}^{\tau}(R)\\
\nonumber (g_1,g_2) &\longmapsto g_1.
\end{align}
\end{lem}
\begin{proof}
Injectivity is clear.  For surjectivity, note that if $g_1^{-1}\gamma \gamma^{-\sigma}g_1 =\gamma \gamma^{-\sigma}$ for $g_1 \in H^{\sigma}(R)$ then
$$
\gamma^{-1}g_1^{-1}\gamma=\gamma^{-\sigma} g_1^{-1}\gamma^{\sigma}=(\gamma^{-1}g_1^{-1}\gamma)^{\sigma},
$$
so $(g_1,(\gamma^{-1}g_1^{-1}\gamma)^{-1}) \in H_{\gamma}(R)$.  Similarly, if $g_1^{-1}\delta \delta^{-\theta}g_1^{\tau}=\delta \delta^{-\theta}$ for $g_1 \in G^{\sigma}(R)$, then
$$
\delta^{-1}g_1^{-1}\delta=\delta^{-\theta}g_1^{-\tau}\delta^{\theta}=(\delta^{-1} g_1^{-\sigma}\delta)^{\theta}=(\delta^{-1} g_1^{-1} \delta)^{\theta}
$$
so $(g_1,(\delta^{-1}g_1^{-1}\delta)^{-1}) \in G_{\delta}(R)$.
\end{proof}

\begin{lem} \label{lem-connect} Let $k/F$ be a field, let $\gamma \in H(k)$
and $\delta \in G(k)$.  If $\gamma\gamma^{-\sigma}$ is semisimple
(resp.~$\delta\delta^{-\theta}$ is  $\tau$-semisimple) then $C_{\gamma \gamma^{-\sigma},H^{\sigma}}$
(resp.~$C_{\delta \delta^{-\theta},G^{\sigma}}^{\tau}$) is reductive.
\end{lem}
\begin{proof} See \cite[\S 2]{Hahn}.
\end{proof}

We also often use the following lemma:

\begin{lem} \label{lem-biquad-cent} Let $k/F$ be a field, let $\gamma \in H(k)$ and $\delta \in G(k)$.  Assume we are in the biquadratic case and that $(H^{\sigma})^{\mathrm{der}}$ is simply connected.
If $\gamma \gamma^{-\sigma}$ is semisimple (resp.~$\delta\delta^{-\theta}$ is $\tau$-semisimple) then $C_{\gamma\gamma^{-\sigma},H^{\sigma}}$ (resp.~$C_{\delta\delta^{-\theta},G^{\sigma}}^{\tau}$) is connected.
\end{lem}
\begin{proof} Recall that for any semisimple $x \in H^{\sigma}(k)$ the group $C_{x,H^{\sigma}}$ is connected by a theorem of Steinberg (\cite{SteinbMem}, \cite[\S 3]{KottRatConj}).
Upon passing to the algebraic closure, it is easy to see that this fact implies the lemma.
\end{proof}

For an $F$-algebra $R$
write
\begin{align}
\Gamma_r(R):&=H^{\sigma}(R) \backslash H(R) /H^{\sigma}(R)\\
\nonumber \Gamma_{r\tau}(R):&=G^{\sigma}(R) \backslash G(R) /G^{\theta}(R).
\end{align}\index{$\Gamma_r$}\index{$\Gamma_{r\tau}$}We refer to the elements of $\Gamma_r(R)$ (resp. $\Gamma_{r\tau}(R)$) as relative classes (resp. relative $\tau$-classes).
If two elements $\gamma,\gamma' \in H(R)$ (resp. $\delta,\delta' \in G(R)$) map to the same
class in $\Gamma_r(R)$ (resp. $\Gamma_{r\tau}(R)$) we say that they are in the same relative class
(resp. relative $\tau$-class).  In this paper we will loosely refer to any construction or object involving
relative $\tau$-classes as the twisted case.

The
following definitions are adaptations of those appearing in
\cite{JacquetLai} and \cite{FlickerRTF}:
\begin{defn}  Let $k/F$ be a field.
An element $\gamma \in H(k)$ is \textbf{relatively semisimple} (resp.~relatively elliptic, relatively regular) if $\gamma \gamma^{-\sigma}$  is semisimple
(resp.~elliptic, regular) as an element of $H(k)$.
\end{defn}

\begin{defn} Let $k/F$ be a field.  An element $\delta \in G(k)$ is
\textbf{relatively $\tau$-semisimple} (resp.~relatively $\tau$-elliptic, relatively $\tau$-regular) if
$\delta\delta^{-\theta}$ is $\tau$-semisimple (resp.~$\tau$-elliptic, $\tau$-regular) in the usual sense.
\end{defn}

For brevity, we say that $\delta$ is relatively $\tau$-regular semisimple if it is
both relatively $\tau$-regular and relatively $\tau$-semisimple, with similar
conventions regarding other combinations of regular, semisimple, and elliptic.
It is easy to check that $\gamma \in H(R)$ is relatively semisimple
(resp. elliptic, regular) if and only if every element in $H^{\sigma}(R)\gamma H^{\sigma}(R)$
is relatively semisimple (resp. elliptic, regular).  A similar statement holds in the twisted case. For a field $k/F$  we denote by
\begin{align}
\Gamma_r^{ss}(k):&=\{H^{\sigma}(k)\gamma H^{\sigma}(k) \in \Gamma_r(k):\gamma \textrm{ is
relatively semisimple}\}\\
\nonumber \Gamma_{r\tau}^{ss}(k):&=\{G^{\sigma}(k)\delta G^{\theta}(k) \in \Gamma_{r\tau}(k):\delta \textrm{ is relatively }\tau\textrm{-semisimple}\}.
\end{align}\index{$\Gamma_r^{ss}$} \index{$\Gamma_{r\tau}^{ss}$}

Let $Q^{ss} \subset Q$ be the subscheme whose points in
an $F$-algebra $R$ are semisimple elements of $H(R)$.

\begin{lem} \label{lem-sscl1} Assume that we are in the biquadratic case and $H^{\mathrm{der}}$ is simply connected.
For $k$ either $F$ or $F_v$ for some place $v$ of $F$ the natural morphism \eqref{nat-maps} induce
a bijection
$$
\Gamma_r^{ss}(k) \leftrightarrow H^{\sigma}(k) \backslash Q^{ss}(k).
$$
\end{lem}
\begin{proof}
We claim that every element of $Q^{ss}(k)$ is of the form $\gamma \gamma^{-\sigma}$
for some $\gamma
\in H(k)$.  To prove the claim, first let $\alpha \in Q^{ss}(k)$ and consider $C_{\alpha,H}(k)$.
Since $\alpha$ is semisimple and $H^{\mathrm{der}}$ is simply connected, $C_{\alpha,H}$ is a
(connected) reductive group (\cite{SteinbMem}, \cite[\S 3]{KottRatConj}). Using the fact that $\alpha=\alpha^{-\sigma}$,
it is trivial to check that $C_{\alpha,H}$ is $\sigma$-invariant, and it follows that
$C_{\alpha,H}=\mathrm{Res}_{E \otimes_F k/k}I$ for some reductive subgroup $I \leq H^{\sigma}$.  Let $T' \leq
I$ be a maximal torus and let $T=\mathrm{Res}_{E \otimes_F k/k}T'$.  Then $\alpha$ is in the
centralizer of $T$ in $C_{\alpha,H}$, and it follows that $\alpha \in T(k)$
\cite[\S IV.11.12]{LAGBorel}.

Denoting by $\langle \sigma \rangle$ the subgroup of the group of automorphisms
of $T(k)$ generated by $\sigma$, we have
$$
{}_{N_{\langle \sigma \rangle}}T(k)/I_{\langle \sigma \rangle}T(k)
=\hat{H}^{-1}(\langle \sigma \rangle,T(k)) \cong H^1(\langle \sigma \rangle, T(k))\cong1
$$
(see \cite[\S VIII.1 and \S VIII.4]{SerreLFs} for notation and the existence of the
first isomorphism).  Here we have used Shapiro's lemma for the last isomorphism.
By definition of ${}_{N_{\langle \sigma \rangle}}T(k)/I_{\langle \sigma \rangle}T(k)$,
it follows that $\alpha=\gamma \gamma^{-\sigma}$ for some $\gamma \in T(k)$.  This completes the proof
of our claim.  We leave to the reader the straightforward task of deducing the lemma from the claim.
\end{proof}

\subsection{Norms of stable classes: Definitions}
\label{ssec-match} In this subsection we define stable classes and a notion of norm.
We assume throughout the subsection that $k$ is either $F$ or $F_v$ for some place $v$ of $F$.
 In this section and for the remainder of the paper, we will always make the following simplifying assumption:
\begin{itemize}
\item For all relatively semisimple $\gamma \in H(\bar{F})$ the group $H_{\gamma}$ is connected.
\end{itemize}
By passing to an algebraic closure, it is easy to see that this assumption implies that $G_{\delta}$ is connected for all relatively $\tau$-semisimple $\delta \in G(\bar{F})$.  For a condition sufficient to ensure the assumption above holds, see Lemma \ref{lem-biquad-cent}.

\begin{defn} \label{defn-src}   Two relatively semisimple elements
$\gamma,\gamma_0 \in H(k)$
are in the same \textbf{stable relative class} if and only if
$x\gamma\gamma^{-\sigma}x^{-1}=\gamma_0\gamma_0^{-\sigma}$ for some $x \in H^{\sigma}(\bar{k})$.

Similarly, two elements $\delta, \delta_0 \in G(k)$ are in the same
\textbf{stable relative $\tau$-class} if and only if $y \delta\delta^{-\theta} y^{-\tau}=\delta_0\delta_0^{-\theta}$ for some $y \in G^{\sigma}(\bar{k})$.
\end{defn}

For a pair of reductive $k$-groups $I,H$,
define
$$
\mathcal{D}(I,H;k):=\mathrm{ker}\left[H^1(k,I) \to H^1(k,H)\right].
$$\index{$\mathcal{D}(I,H;k)$}
We note that the set of relative classes in the stable relative class of $\gamma_0 \in H(k)$ is in bijection with
\begin{align}
 \mathcal{D}(C_{\gamma_0\gamma_0^{-\sigma},H^{\sigma}},H^{\sigma};k)
\end{align}
and the set of relative $\tau$-classes in the stable relative $\tau$-class of $\delta_0 \in G(k)$ is in bijection with
\begin{align} \label{stableclassbij}
\mathcal{D}(C^{\tau}_{\delta_0\delta_0^{-\theta},G^{\sigma}},G^{\sigma};k).
\end{align}
In particular, if $k=F_v$ for some place $v$ of $F$, then both of these sets of stable classes are finite
(compare \cite[\S 1.8, \S 2.3]{Lab}).

Later, when defining orbital integrals, the following observation will be useful.  Suppose that $\gamma,\gamma_0$ are in the same stable relative class; choose $x$ as in Definition \ref{defn-src}.  Then we have an inner twist
\begin{align}
\label{src-twist} H_{\gamma}(\bar{k}) \tilde{\lto} C_{\gamma\gamma^{-\sigma},H^{\sigma}}(\bar{k}) &\tilde{\lto}C_{\gamma_0\gamma_0^{-\sigma},H^{\sigma}}(\bar{k}) \tilde{\lto} H_{\gamma_0}(\bar{k})\\
\nonumber g & \longmapsto xgx^{-1}
\end{align}
where the first and last isomorphism are those of \eqref{2cent} (compare \cite[Lemma 3.2]{KottRatConj}).  Similarly, if $\delta,\delta_0$ are in the same stable relative $\tau$-class, choose $y$ as in Definition \ref{defn-src}.  Then there is an inner twisting
\begin{align}
\label{strc-twist}
G_{\delta}(\bar{k}) \tilde{\lto} C_{\delta \delta^{-\theta},G^{\sigma}}^{\tau}(\bar{k}) & \tilde{\lto} C_{\delta_0 \delta_0^{-\theta},G^{\sigma}}^{\tau}(\bar{k}) \tilde{\lto} G_{\delta_0}(\bar{k})\\
\nonumber g &\longmapsto ygy^{-1}.
\end{align}

We have the following definition:

\begin{defn} \label{defn-norm} Suppose that $\gamma \in H(k)$ is relatively semisimple.
We say that $\gamma$ is a \textbf{norm} of $\delta \in G(k)$ if there is an $y \in G^{\sigma}(\bar{k})$ such that
\begin{align*}
y\delta\delta^{-\theta}(\delta\delta^{-\theta})^{\tau} y^{-1}=
\gamma\gamma^{-\sigma}.
\end{align*}

\end{defn}

The following lemma together with Proposition \ref{matching-prop} below can be used to show that norms exist in certain cases:

\begin{lem} \label{lem-geom-norm}   Suppose that $\delta \in G(k)$
is relatively $\tau$-semisimple.
  Let $O(\bar{k}) \subset Q(\bar{k})$ denote the subset of elements that are $G^{\sigma}(\bar{k})$-conjugate
to $\delta \delta^{-\theta}(\delta \delta^{-\theta})^{-\sigma}=\delta \delta^{-\theta}(\delta \delta^{-\theta})^{\tau}$.  Then $O(\bar{k})$ is fixed by $\Gal(\bar{k}/k)$ and is the set of
$\bar{k}$-points of a (nonempty) homogeneous space $O$ for $H^{\sigma}_{k}$.  Any $\gamma \in H(k)$ such that $\gamma \gamma^{-\sigma} \in O(k)$ is a norm of $\delta$.
\end{lem}

\begin{proof}
Let $\widetilde{Q}:=\mathrm{Res}_{M/F} Q$; it is the scheme theoretic image of the map $G \lto G$ defined by
\begin{align*}
G(R) &\lto G(R)\\
g &\longmapsto gg^{-\sigma}
\end{align*}
for $F$-algebras $R$.   The group scheme $G^{\sigma}$ acts on $\widetilde{Q}$ by conjugation.
Note that $\delta \delta^{-\theta}(\delta\delta^{-\theta})^{\tau} \in \widetilde{Q}(k)$.  We claim that a $G^{\sigma}(\bar{k})$-conjugate $\bar{\gamma}$ of $\delta \delta^{-\theta}(\delta \delta^{-\theta})^{\tau}$ is contained in $Q(\bar{k})$.
Assuming the claim for a moment, it is easy to check that
$$
O(\bar{k}):=\{h^{-1}\bar{\gamma}h:h \in H^{\sigma}(\bar{k})\}=\{g^{-1} \delta \delta^{-\theta}(\delta \delta^{-\theta})^{\tau}g :g \in G^{\sigma}(\bar{k})\} \cap Q(\bar{k}).
$$
Since the set on the right is fixed by $\Gal(\bar{k}/k)$, the set on the left is as well.  It follows that $O(\bar{k})$ is the set of $\bar{k}$-points of a homogeneous space $O$ for $H^{\sigma}_k$.

Choose an isomorphism
\begin{align} \label{isom15}
G_{\bar{k}} \tilde{\lto} H_{\bar{k}} \times H_{\bar{k}}
\end{align}
 equivariant with respect to $\sigma$ and intertwining $\tau$ with $(x,y) \mapsto (y,x)$.  Let $(\delta_1,\delta_2)$ be the image of $\delta \delta^{-\theta}$ under this isomorphism.  Then $\delta_2=\delta_1^{-\sigma}$.  Replacing $\delta\delta^{-\theta}$ by $g^{-1}\delta \delta^{-\theta}g^{\tau}$ for some $g \in G^{\sigma}(\bar{k})$ does not affect the $G^{\sigma}(\bar{k})$-conjugacy class of $\delta \delta^{-\theta}(\delta \delta^{-\theta})^{\tau}$.  Translating this statement to $H_{\bar{k}} \times H_{\bar{k}}$ using \eqref{isom15}, we see that to prove the claim it suffices to exhibit $h_1,h_2 \in H^{\sigma}(\bar{k})$
such that $h_1^{-1}\delta_1h_2=h_2^{-1}\delta_2h_1$.  Since $h_2^{-1} \delta_2h_1=(h_1^{-1}\delta_1h_2)^{-\sigma}$, this is equivalent to the statement that $\delta_1$ is in the relative class of an element $t$ satisfying $t^{-\sigma}=t$.  Because $\delta \delta^{-\theta}(\delta \delta^{-\theta})^{\tau}$ is semisimple by assumption, $\delta_1$ is relatively semisimple.  Thus $\delta_1$ is in the relative class of an element $t$ of a $\sigma$-split torus by
\cite[Theorem 7.5]{Rich}.  This element satisfies $t=t^{-\sigma}$.
\end{proof}

Suppose that $\gamma \in H(k)$ is a norm of $\delta \in G(k)$. We then have an inner twist
\begin{align} \label{norm-twist}
G_{\delta}(\bar{k}) \tilde{\lto}C_{\delta\delta^{-\theta},G^{\sigma}}^{\tau}(\bar{k})
\tilde{\lto} C_{\delta\delta^{-\theta}(\delta\delta^{-\theta})^{\tau},H^{\sigma}}(\bar{k})&\tilde{\lto}C_{\gamma\gamma^{-\sigma},H^{\sigma}}(\bar{k}) \tilde{\lto}H_{\gamma}(\bar{k})\\
\nonumber g &\longmapsto y g y^{-1}
\end{align}
for $y$ as in Definition \ref{defn-norm}.  Here the first and last
isomorphism are those of \eqref{2cent}, and the second is the restriction to $C^{\tau}_{\delta \delta^{-\theta},G^{\sigma}}$
of the projection of $G^{\sigma}(\bar{k}) \cong H^{\sigma}(\bar{k}) \times H^{\sigma}(\bar{k})$ onto the factor indexed by the identity.

\subsection{Existence of norms}

\label{ssec-norm-exist}

If we were in the standard trace formula setting, now would be the point where we
would use the Kottwitz-Steinberg theorem, which states that a conjugacy class defined over $F$ in a quasi-split reductive group with simply
connected derived group contains an element \cite{KottRatConj}, and therefore (in this setting) norms exist.  Unfortunately, no general analogue of this theorem is known in the relative setting.  In fact, preliminary investigations suggest that it is not true.

In this subsection we provide a conditional relative analogue of the Kottwitz-Steinberg theorem using work of Borovoi.
We say that $H$ is a unitary group if there is an $F$-algebra $D$ and an involution $\dagger:D \to D$ of the second kind such that $H(R)$ is equal to
$$
\{ g\in (D \otimes_F R)^{\times}: gg^{\dagger}=1\}
$$
for $F$-algebras $R$.  We assume in addition that either $D$ is simple or the center of $D$ is $F \oplus F$ and $D=D_0 \times D_0$ for some simple algebra $D_0$.  Throughout this subsection we always assume that $H$ is a unitary group and that we are in the biquadratic case.  Thus $(H^{\sigma})^{\mathrm{der}}$ is simply connected.  Moreover since we are in the biquadratic case we have $Z_{H^{\sigma}} \leq Z_{H}$, so the $H^{\sigma}$-orbit of any $\alpha \in Q(\bar{F})$ under conjugation is equal to its orbit under the subgroup $(H^{\sigma})^{\mathrm{der}}$.  This fact is used in the proof of the following proposition:

\begin{prop} \label{matching-prop} If $\delta \in G(F)=\mathrm{Res}_{M/F}H(F)$ is relatively $\tau$-regular
semisimple, the element $\delta_{v_0}$ is relatively $\tau$-elliptic for some place $v_0$, and for all $v|\infty$ there is a $\gamma \in H(F_v)$ that is a norm of $\delta \in G(F_v)$,
then there is a $\gamma \in H(F)$ such that $\gamma$ is a norm of $\delta$.
\end{prop}

Before proving Proposition \ref{matching-prop} we will state and prove a converse to it:

\begin{lem} \label{matching-lem}
Let $k$ be either $F$ or $F_v$ for some place $v$ of $F$.  If $\gamma \in
H(k)$ is relatively semisimple, then it is a norm of some $\delta \in
G(k)$.  In fact, $\gamma\gamma^{-\sigma}=\delta\delta^{-\theta}(\delta\delta^{-\theta})^{\tau}$
for some $\delta \in
G(k)$.
\end{lem}

\begin{proof}
By the proof of Lemma \ref{lem-sscl1}, for any relatively semisimple $\gamma \in H(k)$ we can choose a $\sigma$-stable torus
 $T \leq H_k$ such that $\gamma \gamma^{-\sigma}$ is in the image of the map
 \begin{align*}
 T(k) \lto Q(k)\\
 g \longmapsto gg^{-\sigma}.
 \end{align*}
Let $\widetilde{T}:=\mathrm{Res}_{M \otimes k/k}T \leq G_k$.  Then we have two composite maps
$$
\begin{CD}
\psi_1:\widetilde{T}(k) @>{x \mapsto x x^{\tau}}>>T(k) @>{y \mapsto y y^{-\sigma}}>> T \cap Q(k)\\
\psi_2:\widetilde{T}(k) @>{x \mapsto xx^{-\theta}}>> \widetilde{T} \cap S(k) @>{y \mapsto y y^{\tau}}>> T \cap  Q(k).
\end{CD}
$$
It suffices to show that $\psi_2$ is surjective.  Since all of the groups in the
definition of the $\psi_i$
are commutative, $\psi_1=\psi_2$, so it suffices to show that $\psi_1$ is
surjective.  The second map in the definition of $\psi_1$ is surjective by the argument in
the proof of Lemma \ref{lem-sscl1}.  The first map is
surjective by \cite[Proposition 3.11.1(a)]{Rog}.
\end{proof}

We prove Proposition \ref{matching-prop} using a
cohomological obstruction to the existence of points in a homogeneous space developed in \cite{BorHasse}.
In order to use Borovoi's results, we must develop a little notation.
Set
\begin{align} \label{ker}
\mathrm{ker}^j(F, \cdot):=\mathrm{ker}(H^j(F,\cdot) \to \prod_vH^j(F_v,\cdot))
\end{align}
where the product is over all places of $v$.
Let $H$ be an $F$-group and let $X$ be a
homogeneous space for $H$.  Suppose for simplicity that the
stabilizer $H_{\bar{x}}$ of a point $\bar{x} \in X(\bar{F})$ is connected and reductive.
The stabilizer $H_{\bar{x}}$ is not in general the base change to $\bar{F}$ of an algebraic
group over $F$.  However, as explained in \cite[\S 1.7]{BorHasse}, the maximal toric quotient
$H_{\bar{x}}/H_{\bar{x}}^{\mathrm{der}}$ has a canonical $F$-form.   We denote this $F$-form by $C^m_{\bar{x},H}$.\index{$C^m_{\bar{x},H}$}
The obstruction $\mathrm{Ob}(H,X)$ \index{$\mathrm{Ob}(H,X)$}lies in $\mathrm{ker}^2(F,C^m_{\bar{x},H})$ \cite[\S 1.5]{BorHasse}.

We now prove Proposition \ref{matching-prop}:

\begin{proof}
Let $\alpha=\delta\delta^{-\theta}(\delta\delta^{-\theta})^{\tau}$.  By Lemma \ref{lem-geom-norm} the set of elements of $Q(\bar{F})$ that are $G^{\sigma}(\bar{F})$-conjugate to $\alpha$ defines a homogeneous space $O(\alpha)$ for $H^{\sigma}$ and hence also for $(H^{\sigma})^{\mathrm{der}}$ since $Z_{H^{\sigma}} \leq Z_H$.
By \cite[Theorem 1.6]{BorHasse} it suffices to show that
$$
\mathrm{Ob}((H^{\sigma})^{\mathrm{der}},O(\alpha)) \in \mathrm{ker}^2(F,C_{\alpha,(H^{\sigma})^{\mathrm{der}}}^m)
$$
is trivial; to do this we will show that $(C_{\alpha,(H^{\sigma})^{\mathrm{der}}}^m)_{F_{v_0}}$ is anisotropic and hence
$\mathrm{ker}^2(F,C_{\alpha,(H^{\sigma})^{\mathrm{der}}}^m)=0$ by \cite[Lemme 1.9]{Sansuc}.
One checks via \cite[\S 1.7]{BorHasse} that there is a natural injection
$$
C_{\alpha,(H^{\sigma})^{\mathrm{der}}}^m \lto C_{\alpha,H^{\sigma}}^m
$$
and $Z_{H^{\sigma}}C_{\alpha,(H^{\sigma})^{\mathrm{der}}}^m = C_{\alpha,H^{\sigma}}^m$.
Hence to prove that the torus $(C_{\alpha,(H^{\sigma})^{\mathrm{der}}}^m)_{F_{v_0}}$ is anisotropic it suffices to show that the torus $(C_{\alpha,H^{\sigma}}^m/Z_{H^{\sigma}})_{F_{v_0}}$ is anisotropic.

The $H(\bar{F})$-conjugacy class of $\alpha$ is fixed by $\Gal(\bar{F}/F)$.
  Denote this conjugacy class by $O_{H}$; we view $O_{H}$ as a homogeneous space for $H$. We have
$(C_{\alpha,H^{\sigma}}^m)_{\bar{F}}=C_{\alpha,H^{\sigma}}$ and $(C_{\alpha,H}^m)_{\bar{F}}=C_{\alpha,H}$ since $\alpha$ is regular semisimple
and $H^{\mathrm{der}}$ is simply connected.  We therefore have a natural $\sigma$-equivariant embedding
  \begin{align} \label{kform1}
  C_{\alpha,H^{\sigma}} \to (C_{\alpha,H}^m)_{\bar{F}}
  \end{align}
  where we define $C_{\alpha,H}^m$ with
  respect to the homogeneous space $O_{H}$ for $H$.
  Thus \eqref{kform1} defines an $F$-form of
  $C_{\alpha,H^{\sigma}}$.
By assumption, $\alpha^{\xi}$ is $H^{\sigma}(\bar{F})$-conjugate
    to $\alpha$ for all $\xi \in \Gal(\bar{F}/F)$.
    Using this fact, one checks that the $F$-form of
    $C_{\alpha,H^{\sigma}}$
    defined by \eqref{kform1} is isomorphic to the
    $F$-torus $C_{\alpha,H^{\sigma}}^m$ defined with respect to
    $O(\alpha)$ (see \cite[\S 1.7]{BorHasse}).  Thus, by \eqref{kform1}, we have that
  $$
  C_{\alpha,H^{\sigma}}^m(R)=\{g \in C_{\alpha,H}^m(R): g^{\sigma}=g\}
  $$
  for commutative $F$-algebras $R$.
By the theorem of Kottwitz-Steinberg \cite[Theorem 4.1]{KottRatConj} the $H(\bar{F})$-conjugacy class of $\alpha$ contains an element $\alpha' \in H(F)$.  We
therefore have a natural isomorphism of $F$-groups $C_{\alpha,H}^m \cong C_{\alpha',H}^m \cong
C_{\alpha',H}$
(see \cite[\S 1.7]{BorHasse}).  Since we assumed that $\delta_{v_0}$ was relatively $\tau$-elliptic, we have that $(C_{\alpha',H}/Z_{H})_{F_{v_0}}$ is anisotropic, and hence the same is true of $(C_{\alpha,H^{\sigma}}^m/Z_{H^{\sigma}})_{F_{v_0}}$, completing the proof of the proposition.

\end{proof}

\section{Matching of functions} \label{sec-matching}

In \S \ref{ssec-match} we defined a notion of norm for relative classes.  In this section we define
a corresponding notion of matching functions.  This necessitates the introduction
of local orbital integrals and their stable analogues (these are defined in \S \ref{ssec-loi} and \S \ref{ssec-st-loi}, respectively).
We then define what we mean by matching of functions in \S \ref{ssec-matching}.  As in the usual trace formula, it is useful to
have at our disposal local orbital integrals twisted by a character of a certain cohomology group; these relative $\kappa$-orbital integrals
are defined in \S \ref{ssec-lroi}.  They play a role in the prestabilization of (the regular elliptic part of) the twisted relative trace formula
carried out in \S \ref{sec-prestab} and \S \ref{sec-group-ell} below.  In \S \ref{ssec-fl-first-case} and \S \ref{ssec-split-match} we prove various matching statements.

\subsection{Local orbital integrals: Definitions} \label{ssec-loi}
Let $v$ be a place of $F$.  Let
$\Phi \in C^{\infty}_c(
H(F_{v}))$ and let $\gamma \in H(F_v)$ be a relatively semisimple element.   The (local) {\bf relative
orbital integral} is defined by
\begin{align}\label{loi}
RO_{\gamma}(\Phi):=RO_{\gamma}(\Phi,dt_{\gamma}):&=\iint_{H_{\gamma}(F_v) \backslash
H^{\sigma}(F_{v})^2}\Phi(h_1^{-1} \gamma h_2)
\frac{dh_1dh_2}{dt_{\gamma}}
\end{align}\index{$RO_{\gamma}(\Phi)$}where $dh_i$ and $dt_{\gamma}$ are Haar measures on $H^{\sigma}(F_v)$ and $H_{\gamma}(F_v)$, respectively.

Similarly, let $f \in C^{\infty}_c(G(F_v))$ and let $\delta \in G(F_v)$ be a
relatively $\tau$-semisimple element.  The (local) {\bf twisted relative orbital integral} is defined by
$$
TRO_{\delta}(f):=TRO_{\delta}(f,dt_{\delta}):=
\iint_{G_{\delta}( F_v) \backslash G^{\sigma} \times G^{\theta}(F_v)}
f(g_1^{-1} \delta g_2)  \frac{dg_1dg_2}{dt_{\delta}},
$$ \index{$TRO_{\delta}(f)$}where $dg_i$, and $dt_{\delta}$ are Haar measures on $G^{\sigma}( F_v)$, $G^{\theta}(F_v)$ and $G_{\delta}(F_v)$, respectively.
These relative orbital integrals and twisted relative orbital integrals are all absolutely convergent \cite{Hahn}.

\subsection{Stable local orbital integrals} \label{ssec-st-loi}
We assume the notation of the previous subsection.
For any relatively semisimple $\gamma_0 \in H(F_v)$ we define the
 \textbf{stable relative orbital integral} by
\begin{align}
\label{lsoi} SRO_{\gamma_0}(\Phi):=\sum_{\gamma \sim
\gamma_0}
e(H_{\gamma})RO_{\gamma}(\Phi,dt_{\gamma}),
\end{align}\index{$SRO_{\gamma_0}(f)$}where $e(H_{\gamma})$ denotes the Kottwitz sign as in \cite{KottSign} and the sum is over a set of representatives for the set of relative classes in the stable relative class of $\gamma_0$.

Similarly, for a relatively $\tau$-semisimple $\delta_0 \in G(F_v)$ we define the \textbf{stable twisted
relative orbital integral} by
\begin{align}
\label{lstoi}
STRO_{\delta_0}(f):=
\sum_{\delta \sim \delta_0}e(G_{\delta})TRO_{\delta}(f,dt_{\delta}),
\end{align} \index{$STRO_{\delta_0}(f)$}where the sum is over a set of representatives for the set of relative $\tau$-classes in the stable relative
$\tau$-class of $\delta_0$.  In both cases, we assume that the measures $dt_{\gamma}$ (resp.~$dt_{\delta}$)
are compatible with respect to the inner twists given by \eqref{src-twist} and \eqref{strc-twist}, respectively.
For the definition of compatible, see, e.g., \cite[p.~631]{KottTama}.  We note that the sums in \eqref{lsoi}
and \eqref{lstoi} are finite since their cardinality is equal to the
cardinality of the finite groups $\mathcal{D}(H_{\gamma},H^{\sigma} \times H^{\sigma};F_v)$ and
$\mathcal{D}(G_{\delta},G^{\sigma} \times G^{\theta};F_v)$,
respectively (see \S \ref{ssec-match}).

\subsection{Matching of functions: Definition} \label{ssec-matching}

\begin{defn} \label{defn-match} Two functions $\Phi \in C_c^{\infty}(H(F_v))$ and $f \in C_c^{\infty}(G(F_v))$
\textbf{match on the relatively regular set} (resp.~\textbf{match}) if
$$
SRO_{\gamma}(\Phi)=STRO_{\delta}(f)
$$
whenever a relatively regular semisimple (resp.~relatively semisimple)
$\gamma \in H(F_v)$ is a norm of $\delta \in G(F_v)$,
$$
SRO_{\gamma}(\Phi)=0
$$
whenever $\gamma$ is relatively regular semisimple (resp.~relatively semisimple) and not a norm of a $\delta \in G(F_v)$, and
$$
STRO_{\delta}(f)=0
$$
whenever $\delta$ is relatively $\tau$-regular semisimple (resp. relatively $\tau$-semisimple)
 and does not admit a norm $\gamma \in H(F_v)$.
\end{defn}
In the definition, we assume that the implicit Haar measures are compatible with respect to the inner twist of \eqref{norm-twist}

\begin{rem} The inner twists given by \eqref{src-twist}, \eqref{strc-twist},
and \eqref{norm-twist} are not canonical, since they depend on choices of
elements in $H^{\sigma}(\bar{k})$ and $G^{\sigma}(\bar{k})$.  However, if $H_1$ and $H_2$  are
inner forms of the same connected reductive $F$-group, then the notion of a measure on
$H_1$ being compatible with a measure on $H_2$ is independent of the particular inner
twist realizing the two groups as inner forms of each other (see \cite[p.~631]{KottTama}).
Thus the ambiguity in \eqref{src-twist}, \eqref{strc-twist}, and \eqref{norm-twist} is irrelevant for our purposes.
\end{rem}

\subsection{Local relative $\kappa$-orbital integrals} \label{ssec-lroi}

In this subsection we define relative versions of the $\kappa$-orbital integrals that occur in
Langland's and Kottwitz's prestabilization of the usual trace
formula \cite{KottEllSing}.   We  use Labesse's reformulation; see
\cite{Lab} for notation and generalities regarding abelianized cohomology.

Let $v$ be a place of $F$ and let $I \leq H$ be a pair of connected reductive $F_v$-groups.  Letting a superscript ``$D$''\index{$D$} denote the Pontryagin dual, we set notation for the finite abelian
groups
\begin{align} \label{localKgroup}
\mathfrak{K}(I,H;F_v):&=H^0_{\mathrm{ab}}(F_v,I \backslash H)^D.
\end{align}\index{$\mathfrak{K}(I,H;F_v)$}Let $\gamma \in H(F_v)$ be
relatively semisimple and let $\delta \in G(F_v)$
be relatively $\tau$-semisimple.
For  $\kappa_{Hv} \in \mathfrak{K}(H_{\gamma},H^{\sigma} \times H^{\sigma};F_v)$ and
 $\Phi_v \in C_c^{\infty}(H(F_v))$
 (resp. $\kappa_v \in \mathfrak{K}(G_{\delta},G^{\sigma} \times
 G^{\theta};F_v)$
and $f_v \in C_c^{\infty}(G(F_v))$) we set
\begin{align}\label{kappa-roi}
RO^{\kappa_{Hv}}_{\gamma}(\Phi_v)&=
\int_{H^0(F_v,H_{\gamma}\backslash H^{\sigma}\times H^{\sigma})}
\langle \kappa_{Hv}, (\dot{h_1},\dot{h_2}^{-1})\rangle e((H_{h_1^{-1}\gamma h_2})_
{F_v})\Phi_v(h_1^{-1}\gamma h_2) d\dot{h}_1d\dot{h}_2,\\
\nonumber TRO^{\kappa_v}_{\delta}(f_v)&=
\int_{H^0(F_v,G_{\delta} \backslash G^{\sigma} \times G^{\theta})}
\langle \kappa_v,(\dot{g}_1,\dot{g}_2^{-1}) \rangle e((G_{g_1^{-1}\delta g_2})_{F_v})
f_v(g_1^{-1} \delta g_2) d\dot{g}_1 d \dot{g}_2.
\end{align}\index{$RO^{\kappa_{Hv}}_{\gamma}(\Phi_v)$}\index{$TRO^{\kappa_v}_{\delta}(f_v)$}Here we use a system of compatible measures as in \S \ref{ssec-st-loi}
to define the measures $d\dot{h}_1d\dot{h}_2$ on
$H^0(F_v, H_{\gamma} \backslash H^{\sigma} \times H^{\sigma})$ and $d\dot{g}_1d\dot{g}_2$ on
$H^0(F_v,G_{\delta} \backslash G^{\sigma} \times G^{\theta})$,
respectively (compare \cite[p.~42, 68]{Lab}).  Of course, the orbital
integrals depend on this choice, but we will not encode it into our notation.   In defining \eqref{kappa-roi}, we choose for each class $(\dot{h_1},\dot{h_2}) \in H^0(F_v,H_{\gamma}\backslash H^{\sigma} \times H^{\sigma})$ an element $(h_1,h_2) \in H^{\sigma} \times H^{\sigma}(\bar{F}_v)$ whose image in $H^0(F_v,H_{\gamma}\backslash H^{\sigma} \times H^{\sigma})$ is $(\dot{h_1},\dot{h_2})$.  Similarly, we choose for each class $(\dot{g}_1,\dot{g}_2) \in H^0(F_v,G_{\delta}\backslash G^{\sigma} \times G^{\theta})$ an element $(g_1,g_2) \in G^{\sigma} \times G^{\theta}(\bar{F}_v)$ whose image in $H^0(F_v,G_{\delta}\backslash G^{\sigma} \times G^{\theta})$ is $(\dot{g}_1,\dot{g}_2)$.
The integrals do not depend on these choices.
The symbol $\langle \kappa_{Hv},(\dot{h_1},\dot{h_2}^{-1}) \rangle$ is the value
of $\kappa_v$ on the image of $(\dot{h_1},\dot{h_2}^{-1})$ under the
abelianization map, and similarly for $\langle \kappa_{v},(\dot{g}_1,\dot{g}_2^{-1})\rangle$.
In the case that $\kappa_{Hv}$ (resp. $\kappa_v$) is trivial,
we have
\begin{align*}
RO^{\kappa_{Hv}}_{\gamma}(\Phi_v)&=SRO_{\gamma}(\Phi_v)\\
TRO^{\kappa_v}_{\delta}(f_v)&=STRO_{\delta}(f_v).
\end{align*}


\subsection{Proof of the fundamental lemma: First case}

\label{ssec-fl-first-case}

Let $v$ be a place of $F$.  For this subsection we assume that there is an isomorphism
\begin{align} \label{H-spl-isom}
G_{F_v} \cong H_{F_v} \times H_{F_v}
\end{align}
equivariant with respect to $\sigma$ and intertwining $\tau$ with $(x,y) \mapsto (y,x)$.  For example, $M/F$ could be split at $v$ or (in the biquadratic case)
$ME/F$ could be split at $v$ with $E/F$ and $M/F$ unramified and $E/F$ nonsplit at $v$. In the latter case we have an isomorphism
\begin{align} \label{ek-spl-isom}
E \otimes_F M \otimes_F F_v \cong E_v \oplus E_{v}
\end{align}
equivariant with respect to the natural action of $\sigma$ on both sides and intertwining $\tau$ on the left with
$(x,y) \mapsto (y,x)$ on the right.

\begin{rem} By class field theory, at almost every place either $E/F$ is split or $E/F$ and $M/F$ are unramified with $E/F$ nonsplit at $v$.
\end{rem}

\begin{defn}
An element $\gamma \in H(F_v)$ is a \textbf{split norm} of an element $\delta=(\delta_1,\delta_2) \in G(F_v)$
if $\gamma$ is in the relative class of $\delta_1\delta_2^{-\sigma}$.
\end{defn}

The basic properties of this notion are summarized in the following lemma:

\begin{lem} \label{lem-sn}
If $\gamma$ is a split norm of $\delta$, then every
element of the relative class of $\gamma$ is a split norm of every element of the relative $\tau$-class of $\delta$.
If $\gamma$ and $\gamma'$ are split norms of $\delta$, then $\gamma$ and $\gamma'$ are in the same relative class.
If $\delta$ and $\delta'$ both have split norm $\gamma$, then $\delta$ and $\delta'$ are in the same relative $\tau$-class.
Every  $\gamma \in H(F_v)$ is a split norm and every $\delta \in G(F_v)$ has a split
norm.  If $\gamma$ is relatively semisimple and $\gamma$ is a split norm of $\delta$, then it is a norm of $\delta$.
\end{lem}
\begin{proof} The first three assertions are easy to verify.  For the fourth, note that $(\gamma,1) \in G(F_v)$ has
split norm $\gamma$ for any $\gamma \in H(F_v)$.  Conversely, $\delta=(\delta_1,\delta_2)$ has split norm $\delta_1\delta_2^{-\sigma}$.

We now prove the final statement of the lemma.  Suppose $\gamma$ is a split norm of $\delta$.  Write
$$
\delta\delta^{-\theta}=(\delta_1\delta_2^{-\sigma},\delta_2\delta_1^{-\sigma})=
(x_{\delta},x_{\delta}^{-\sigma})
$$
 for some $x_{\delta} \in H(F_v)$.  At the expense of replacing $\gamma$ with something in its relative class, we can and do assume that $\gamma=x_{\delta}$.  Thus it suffices to show that we can choose $g \in H^{\sigma}(\bar{F}_v)$ such that $x_{\delta}x_{\delta}^{-\sigma}=g^{-1}x_{\delta}^{-\sigma}x_{\delta}g$, for then
$$
(1,g)^{-1}\delta \delta^{-\theta}(\delta \delta^{-\theta})^{\tau}(1,g)=(1,g)^{-1}(x_{\delta}x_{\delta}^{-\sigma},x_{\delta}^{-\sigma}x_{\delta})(1,g)=(x_{\delta}x_{\delta}^{-\sigma},x_{\delta}x_{\delta}^{-\sigma}),
$$
which implies $\gamma=x_{\delta}$ is a norm of $\delta$.

Since $\gamma$ is relatively semisimple, there is a $\sigma$-split torus $T_{\sigma} \leq H_{\bar{F}_v}$ such that
$x_{\delta}x_{\delta}^{-\sigma} \in T_{\sigma}(\bar{F}_v)$  \cite[Theorem 7.5]{Rich}.  Since the
map $x \mapsto xx^{-\sigma}$ coincides with $x \mapsto x^2$ on $T_{\sigma}$, we can and do choose $x' \in
T_{\sigma}(\bar{F}_v)$ such that $x'x'^{-\sigma}=x_{\delta}x_{\delta}^{-\sigma}$.
The fiber of the map $H \to Q$ of \eqref{nat-maps} over any point in $Q(\bar{F}_v)$
 is a $H^{\sigma}(\bar{F})$-torsor, so we have
$x_{\delta}g'=x' \in T_{\sigma}(\bar{F})$ for some $g' \in H^{\sigma}(\bar{F}_v)$.
  Since $T_{\sigma}$ is $\sigma$-stable, we have that
$(x_{\delta}g')^{-\sigma}$ commutes with $x_{\delta}g'$.  In other words,
$$
x_{\delta}x_{\delta}^{-\sigma}=(x_{\delta}g)(x_{\delta}g)^{-\sigma}=(x_{\delta}g)^{-\sigma}(x_{\delta}g)=
g^{-1}x_{\delta}^{-\sigma}x_{\delta}g.
$$
\end{proof}

If $\delta =(\delta_1,\delta_2) \in G(F_v)$ then there is an isomorphism $G_{(\delta_1,\delta_2)} \tilde{\to}H_{\delta_1\delta_2^{-\sigma}}$ given on $F_v$-algebras $R$ by
\begin{align} \label{Mspl-cent}
G_{(\delta_1,\delta_2)}(R) &\lto H_{\delta_1\delta_2^{-\sigma}}(R)\\
\nonumber (g_1,g_2,g_3,g_3^{\sigma}) &\longmapsto (g_1,g_2)
\end{align}
and thus isomorphisms $G_{\delta} \tilde{\lto} H_{\gamma}$ for any split norm $\gamma$ of $\delta$.
For the purposes of the following definition and lemma, fix a compact open subgroup $K_H \leq H(F_v)$ that is $\sigma$-stable and let
$K =K_H \times K_H \leq G(F_v)\cong H(F_v) \times H(F_v)$.

\begin{defn} Two functions $\Phi \in C_c^{\infty}(H(F_v))$
and $f \in C_c^{\infty}(G(F_v))$ \textbf{split match} if
$$
RO_{\gamma}(\Phi,dt_{\gamma})=TRO_{\delta}(f,dt_{\delta})
$$
whenever $\gamma \in H(F_v)$ is a split norm of $\delta \in G(F_v)$.  Here the Haar measures $dt_{\gamma}$ and $dt_{\delta}$ are normalized so that the isomorphism  $G_{\delta}(F_v) \tilde{\lto} H_{\gamma}(F_v)$ induced by \eqref{Mspl-cent} is measure-preserving and the Haar measure on $H(F_v)$ (resp.~$G(F_v)$) gives $K_H$ (resp.~$K$) volume $1$.
\end{defn}

For $\Phi_i \in C_c^{\infty}(H(F_v))$ ($i \in \{1,2\}$) we write $\Phi_1 \times \Phi_2 \in
C_c^{\infty}(G(F_v))=C_c^{\infty}(H \times H(F_v))$ for the product function using \eqref{H-spl-isom}.  We
have the following lemma:
\begin{lem} \label{lem-M-split-match} Let $\Phi_1,\Phi_2 \in C_c^{\infty}(H(F_v)//K_H)$
and write $f=\Phi_1 \times \Phi_2 \in C_c^{\infty}(G(F_v)//K)$ for
the product function.  Write $\Phi_2^{-\sigma}(g):=\Phi_2(g^{-\sigma})$.
 The functions $f$ and $\Phi_1*\Phi_2^{-\sigma}$ split match, where the $*$ denotes
 convolution in $C_c^{\infty}(H(F_v)//K_H)$.
\end{lem}

\begin{proof} Let $\delta=(\delta_1,\delta_2) \in G(F_v)$. We have
\begin{align}
\nonumber &\int_{G_{\delta}(F_v) \backslash G^{\sigma} \times G^{\theta}(F_v)} f(g^{-1} \delta h) dgdh\\
\label{int-comp}&=\int_{H_{\gamma}(F_v) \backslash H^{\sigma} \times H^{\sigma}(F_v)} \left( \int_{H(F_v)} \Phi_1(g_1^{-1}\delta_1 h_1)\Phi_2(g_2^{-1}\delta_2h_1^{\sigma})dh_1\right)dg_1dg_2\\
\nonumber&=\int_{H_{\gamma}(F_v) \backslash H^{\sigma} \times H^{\sigma}(F_v)} \Phi_1*\Phi_2^{-\sigma}(g_1^{-1}\delta_1 \delta_2^{-\sigma}g_2)dg_1dg_2.
\end{align}
Here we have used \eqref{H-spl-isom} to write
$g=(g_1,g_2) \in G^{\sigma}(F_v) \cong H^{\sigma}(F_v) \times H^{\sigma}(F_v)$ and
$h=(h_1,h_1^{\sigma}) \in G^{\theta}(F_v)$ and we have used our assumption that the isomorphism
$G_{\delta}(F_v) \tilde{\lto} H_{\gamma}(F_v)$ given by \eqref{Mspl-cent} is measure preserving.  With
this in mind,
the lemma follows from \eqref{int-comp} and Lemma \ref{lem-sn}.
\end{proof}

The inclusion $H \to G$ induces an $L$-map ${}^LH \to {}^LG$.
Assume for the moment that $H_{F_v}$ and $G_{F_v}$ are both quasi-split, let $K \leq G(F_v)$ and $K_H \leq H(F_v)$ be hyperspecial subgroups and let
$$
b:C_c^{\infty}(G(F_v)//K) \lto C_c^{\infty}(H(F_{v})//K_H)
$$\index{$b$}be the base change homomorphism induced by the $L$-map above.  One can check that with respect to the isomorphism \eqref{H-spl-isom} this homomorphism is given by
$$
b(f_1 \times f_2) =f_1*f_2.
$$
Thus in view of Lemma \ref{lem-sn} and the remark at the end of \S \ref{ssec-matching} we have the following corollary:

\begin{cor} \label{cor-H-split-matching} Let $K_H \leq H(F_v)$ be a $\sigma$-stable
hyperspecial subgroup and let $K=K_H \times K_H \leq G(F_v) \cong H(F_v) \times H(F_v)$.  Then if $f_1 \in C_c^{\infty}(H(F_v)//K_H)$ the functions $f_1 \times \mathrm{ch}_{K_H}$ and $b(f_1 \times \mathrm{ch}_{K_H})=f_1$  match. \qed
\end{cor}

\subsection{Matching at the $E$-split places} \label{ssec-split-match}
Assume for the remainder of this section that we are in the biquadratic case.
Let $v$ be a place of $F$.  In this subsection we will always assume that $E/F$ splits at $v$.
Thus there is an isomorphism
\begin{align} \label{e-spl-isom}
E \otimes_F F_v \cong F_v \oplus F_v
\end{align}
 intertwining $\sigma$ with $(x,y) \mapsto (y,x)$ and an isomorphism
\begin{align} \label{e-spl-isom2}
H_{F_v} \cong H^{\sigma}_{F_v} \times H^{\sigma}_{F_v}
\end{align}
intertwining $\sigma$ with $(x,y) \mapsto (y,x)$.
Using this one easily deduces the following lemma:

\begin{lem} \label{e-spl-lem} Let $\gamma \in H(F_v)$ and $\delta \in G(F_v)$ be relatively semisimple and relatively $\tau$-semisimple, respectively.  Use
\eqref{e-spl-isom} to write
$\gamma\gamma^{-\sigma}=(y_{\gamma},y_{\gamma}^{-1})$ and $\delta\delta^{-\theta}=(x_{\delta},x_{\delta}^{-\tau})$.  We have that $\gamma$ is a norm of $\delta$ if and
only if $x_{\delta}x_{\delta}^{-\tau}$ is $G^{\sigma}(\bar{F}_v)$-conjugate to $y_{\gamma}$. \qed
\end{lem}

If $\Phi_i \in C_c^{\infty}(H^{\sigma}(F_v))$ (resp.
$f_i \in C_c^{\infty}(G^{\sigma}(F_v))$) for $i \in \{1,2\}$, we let
$\Phi_1 \times \Phi_2 \in C_c^{\infty}(H(F_v))$ (resp. $f_1 \times f_2 \in  C_c^{\infty}(G(F_v))$)
be the product functions defined using the isomorphism \eqref{e-spl-isom2}.  For the definition of the $\mathfrak{K}_1$ groups in the following proposition we
refer the reader to \S \ref{ssec-ch-groups}.

\begin{prop} \label{prop-E-split-match} Suppose that $\gamma \in H(F_v)$ is relatively semisimple and
$\delta \in G(F_v)$ is relatively $\tau$-semisimple.  Let $\Phi_i$ and $f_i$ be as above.  Finally
let $\kappa_H \in \mathfrak{K}(H_{\gamma},H^{\sigma} \times H^{\sigma};F_v)_1$ and $\kappa \in \mathfrak{K}(G_{\delta}, G^{\sigma} \times G^{\theta};F_v)_1$.  We then
have
\begin{align*}
RO^{\kappa_H}_{\gamma}(\Phi_1 \times \Phi_2)=\mathcal{O}^{\kappa_H}_{y_{\gamma}}(\Phi_1 * \Phi_2^{-1})\\
TRO^{\kappa}_{\delta}(f_1 \times f_2)=\mathcal{O}^{\kappa}_{x_{\delta} \rtimes \tau}(f_1 *f_2^{-\tau})
\end{align*}
where $\Phi_2^{-1}(h):=\Phi_2(h^{-1})$, $f_2^{-\tau}(h):=f_2(h^{-\tau})$, $\gamma\gamma^{-\sigma}=(y_{\gamma},y_{\gamma}^{-1})$, and $\delta\delta^{-\theta}=(x_{\delta},x_{\delta}^{-\tau})$.
\end{prop}
\noindent In the proposition, we use notation as in \cite[\S 1.5]{HarLab} for the $\kappa$-orbital integrals
$\OO^{\kappa_H}_{y_{\gamma}}$ and $\OO^{\kappa}_{x_{\delta} \rtimes \tau}$, and the $\mathfrak{K}_1$-groups are defined as in \cite[\S 1.8]{Lab} (see also \S \ref{ssec-ch-groups}).  Strictly speaking, the $\kappa_H$ in $\OO^{\kappa_H}_{y_{\gamma}}$ (resp. $\kappa$ in $\OO_{x_{\delta} \rtimes \tau}^{\kappa}$) should actually be the image of $\kappa_H$ (resp. $\kappa$) under
\begin{align} \label{kappamaps}
\mathfrak{K}(H_{\gamma},H^{\sigma} \times H^{\sigma};F_v)_1 &\cong \mathfrak{K}(C_{\gamma \gamma^{-\sigma},H^{\sigma}},H^{\sigma} \times H^{\sigma};F_v)_1 \cong
\mathfrak{K}(C_{y_{\gamma},H^{\sigma}},H^{\sigma};F_v)_1\\
\nonumber \mathfrak{K}(G_{\delta},G^{\sigma} \times G^{\theta};F_v)_1 &\cong \mathfrak{K}(C^{\tau}_{\delta \delta^{-\theta},G^{\sigma}},G^{\sigma} \times G^{\theta};F_v)_1 \cong
\mathfrak{K}(C^{\tau}_{x_{\delta},G^{\sigma}},G^{\sigma};F_v)_1
\end{align}
where the left isomorphisms are induced by Lemma \ref{lem-2cent} and the right isomorphisms are due to the isomorphisms $C_{\gamma \gamma^{-\sigma},H^{\sigma}} \tilde{\to}C_{y_{\gamma},H^{\sigma}}$ and
$C^{\tau}_{\delta \delta^{-\theta},G^{\sigma}}\tilde{\to}C^{\tau}_{x_{\delta},G^{\sigma}}$
induced by the projection of $H \cong H^{\sigma} \times H^{\sigma}$ (resp. $G \cong G^{\sigma} \times G^{\sigma}$) onto the first factor.  However, we won't burden the
notation by indicating this.  Also, to make the $\kappa_H$-orbital integrals well-defined we need to specify choices of Haar measures on $H^{\sigma}$, $G^{\sigma}$, and various centralizers, but, again, we will
not incorporate this into the notation.

\begin{proof}
The statement involving relative orbital integrals can be recovered from the statement involving twisted relative orbital integrals by taking $\tau$ to be trivial.  Therefore, it suffices to prove the statement regarding twisted relative orbital integrals.

Note that there is an isomorphism of
$F_{v}$-group schemes $G_{\delta} \to C^{\tau}_{x_{\delta},G^{\sigma}}$ given
by
\begin{align}\label{third}
G_{\delta}(R) & \longrightarrow
C^{\tau}_{x_{\delta}, G^{\sigma}}(R) \\
((u,u), (h, h^{\tau})) & \longmapsto  u \nonumber
\end{align}
for $F_v$-algebras $R$ (compare Lemma \ref{lem-2cent}).
Moreover, for a given $\delta = (\delta_{1}, \delta_{2}) \in
G(F_{v})$ there is an isomorphism of affine $F_v$-schemes
$G^{\sigma} \times G^{\theta} \to (G^{\sigma})^2$ given by
\begin{align}\label{fourth}
G^{\sigma} \times G^{\theta}(R) & \longrightarrow  G^{\sigma}
\times G^{\sigma}(R) \\
((u,u), (h, h^{\tau})) & \longmapsto
((u,u), (u^{-1} \delta_{2} h^{\tau}, u^{-1} \delta_{2} h^{\tau}))
\nonumber
\end{align}
for $F_v$-algebras $R$.
This maps $G_{\delta}$ onto $C^{\tau}_{x_{\delta},
G^{\sigma}} \times \{ (\delta_{2}, \delta_{2})\},$ and thus a domain of
integration for $H^{0}(F_{v}, G_{\delta} \backslash
G^{\sigma} \times G^{\theta})$ onto $H^{0}(F_{v},
C^{\tau}_{x_{\delta, G^{\sigma}}} \backslash G^{\sigma}) \times
G^{\sigma}(F_{v}).$   This observation is used in one change of variables in the following computation:
\begin{align} \label{123}
&TRO^{\kappa}_{\delta}(f_1 \times f_2)\\
& = \nonumber
\int\limits_{H^{0}(F_{v}, G_{\delta} \backslash
G^{\sigma} \times G^{\theta})}
\langle \kappa, (\dot{g}_{1}, \dot{g}_{2}) \rangle
e(G_{g_{1}^{-1} \delta g_{2}})
f_{1} \times f_{2}(g_{1}^{-1} \delta
g_{2}^{\tau}) d \dot{g}_{1} d\dot{g}_{2} \\ \nonumber
& =
\int\limits_{H^{0}(F_{v}, C^{\tau}_{x_{\delta}, G^{\sigma}} \backslash
G^{\sigma})}
\int\limits_{G^{\sigma} (F_{v})}
\langle \kappa, ((\dot{h}_1,\dot{h}_1),(\dot{\delta}_2^{-\tau} \dot{h}_1^{\tau}\dot{h}_2^{\tau},\dot{\delta}_2^{-1} \dot{h}_1\dot{h}_2)) \rangle
e(C^{\tau}_{h_{1}^{-1} x_{\delta} h^{\tau}_{1}, G^{\sigma}})
f_{1}(h_{1}^{-1} \delta_{1} \delta_{2}^{- \tau} h_{1}^{\tau} h_{2})
f_{2}(h_{2}^{\tau}) dh_{2} d\dot{h}_{1}.
\end{align}
In the last equality we used the fact that there is an isomorphism of $F_v$-groups
$G_{g_{1}^{-1} \delta g_{2}} \cong C^{\tau}_{g_{1}^{-1} x_{\delta} g^{\tau}_{1}, G^{\sigma}}$ for $(g_1,g_2) \in G^{\sigma} \times G^{\theta}(F_v)$ (compare \eqref{third}).  If we temporarily denote the image of $\kappa$ under \eqref{kappamaps} by $\overline{\kappa},$ then
\begin{align} \label{kappa-obs}
\langle
\kappa, ((\dot{h}_1,\dot{h}_1),(\dot{\delta}_2^{-\tau} \dot{h}_1^{\tau}\dot{h}_2^{\tau},\dot{\delta}_2^{-1} \dot{h}_1\dot{h}_2)) \rangle  =
\langle \overline{\kappa}, \dot{h}_{1} \rangle
\end{align}
 holds. From now on, we will omit the bar and just write $\kappa$.  With \eqref{kappa-obs} in mind,
 the expression in \eqref{123} is equal to
\begin{align*}
\int\limits_{H^{0}(F_{v}, C^{\tau}_{x_{\delta, G^{\sigma}}} \backslash
G^{\sigma})}&
\int\limits_{G^{\sigma} (F_{v})}
\langle \kappa, \dot{h}_{1} \rangle
e((C^{\tau}_{h_{1}^{-1} x_{\delta} h^{\tau}_{1}, G^{\sigma}})_{F_{v}})
f_{1}(h_{1}^{-1} \delta_{1} \delta_{2}^{- \tau} h_{1}^{\tau} h_{2})
f_{2}(h_{2}^{\tau}) dh_{2} d\dot{h}_{1} \\
& =
\mathcal{O}^{\kappa}_{x_{\delta} \rtimes \tau}(f_1 *f_2^{-\tau}).
\end{align*}

\end{proof}

Assume now that $v$ is nonarchimedian.  If $G^{\sigma}_{F_v}$ and $H^{\sigma}_{F_v}$ are both unramified, let $K' \leq G^{\sigma}(F_v)$ and $K_H' \leq H^{\sigma}(F_v)$ be hyperspecial subgroups and let
$$
b:C_c^{\infty}(G^{\sigma}(F_v)//K') \lto C_c^{\infty}(H^{\sigma}(F_{v})//K_H')
$$\index{$b$}be the homomorphism induced by the base change homomorphism ${}^LH^{\sigma} \to {}^L G^{\sigma}$.

\begin{cor} \label{fl-e-spl}  If $v$ is nonarchimedian, then for every $\Phi \in C_c^{\infty}(H(F_v))$ there is a matching
$f \in C_c^{\infty}(G(F_v))$ and
conversely.  If $M \otimes_F F_v/F_v$ is unramified, $G^{\sigma}_{F_v}$ and $H^{\sigma}_{F_v}$ are unramified with hyperspecial subgroups $K'$, $K_H'$ as above, and $\Phi_1 \in C_c^{\infty}(G^{\sigma}( F_v)//K')$ then
$\Phi_1 \times \mathrm{ch}_{K'}$ matches $b(\Phi_1) \times \mathrm{ch}_{K'_H}$.
\end{cor}

\begin{proof}This follows immediately from Proposition \ref{prop-E-split-match}, Lemma \ref{e-spl-lem}, \cite[Th\'eor\`eme 3.3.1]{Lab} and
\cite[Proposition 3.7.3]{Lab}.
\end{proof}

\section{Matching at the ramified places}
\label{sec-match-ram}
  We assume the notation of \S \ref{sec-matching} regarding the definition of stable twisted relative orbital integrals and matching of
  functions.  In \S \ref{sec-matching} we provided a supply of matching functions for places of $F$ where various data were unramified.  We now
  prove a weaker matching statement for the ramified places.    It is contained in the following theorem, the main theorem of this section:

\begin{thm} \label{thm-weak-match}   Let $v$ be a place of $F$.
If $f \in C_c^{\infty}(G(F_v))$ is supported in a sufficiently small neighborhood of a relatively $\tau$-regular semisimple $\delta \in G(F_v)$ admitting a norm $\gamma \in H(F_v)$
then there is a $\Phi \in C_c^{\infty}(H(F_v))$
 that matches $f$ on the relatively regular set.
\end{thm}
\noindent The proof will occupy the remainder of the section.  Our approach is an adaptation of that in \cite[\S 3.1]{Lab}.  We also found the work of Rader-Rallis \cite{RadRal} and Hakim \cite{HakimSCHP} useful.

\subsection{Local constancy of the relative orbital integrals} \label{ssec-loc-const}
Let $v$ be a place of $F$ and let $\delta \in G(F_v)$ be a relatively $\tau$-regular semisimple element.
In this subsection we show that the relative orbital integrals of functions with support in
a small neighborhood of $\delta$ can be viewed as a locally constant functions on a torus related to $\delta$.
To ease notation, throughout this section we abbreviate $H=H_{F_v}$, $S=S_{F_v}$, etc.

 We begin with the following lemma:
\begin{lem} \label{lem-tdelta} If
$\gamma \in H(F_v)$ is relatively regular semisimple, then the maximal $\sigma$-split subtorus of
$C_{\gamma\gamma^{-\sigma},H}$ is a maximal $\sigma$-split torus of $H$.  If
 $\delta$ is relatively $\tau$-regular semisimple then the maximal $\sigma$-split subtorus of $C_{\delta\delta^{-\theta},G}^{\tau}$ has the same dimension
 as a maximal $\sigma$-split torus of $H$.
\end{lem}
\noindent We write $T_{\gamma}$ \index{$T_{\gamma}$} for the maximal $\sigma$-split subtorus of $C_{\gamma\gamma^{-\sigma},H}$
and $\widetilde{T}_{\delta}$ \index{$\widetilde{T}_{\delta}$} for the maximal $\sigma$-split subtorus of $C_{\delta\delta^{-\theta},G}^{\tau}$.
\begin{proof}
The element $\gamma\gamma^{-\sigma}$ is contained in a maximal $\sigma$-split torus $T_{\sigma}$ of $H_{\bar{F}_v}$ \cite[Theorem 7.5]{Rich} which in turn is contained in a maximal $\sigma$-stable torus $T$
\cite[Proposition 1.4]{Helmtori}.  Moreover, $T_{\sigma}$ is the unique maximal $\sigma$-split torus of $T$ (this follows from \cite[Theorem 7.5]{Rich}).
 Since $\gamma\gamma^{-\sigma}$ is regular semisimple, it is contained in a unique maximal torus, so $T=(C_{\gamma\gamma^{-\sigma},H})_{\bar{F}_v}$,
 and it follows that $T_{\sigma}=T_{\gamma \bar{F}_v}$.  Here we are implicitly using the fact that $T_{\sigma}$ is in fact defined over $F_v$, being the $\sigma$-split component of the $F_v$-torus $C_{\gamma\gamma^{-\sigma},H}$.  Thus $T_{\gamma}$ is a maximal $\sigma$-split torus.

We now prove the second claim.  Choose an isomorphism $G_{\bar{F}_v} \cong H_{\bar{F}_v} \times H_{\bar{F}_v}$ equivariant with respect to $\sigma$ and intertwining $\tau$ with $(x,y) \mapsto (y,x)$.  Using this isomorphism, write $\delta \delta^{-\theta}=(x_{\delta},x_{\delta}^{-\sigma})$ for some $x_{\delta} \in H(\bar{F}_v)$.  Our assumption that $\delta$ is relatively $\tau$-regular semisimple implies that $x_{\delta}$ is relatively semisimple.  By the argument above, we see that there is a (unique) maximal $\sigma$-split subtorus of $C_{x_{\delta}x_{\delta}^{-\sigma},H}$ which is moreover a maximal $\sigma$-split subtorus of $H$.  On the other hand, there is a $\sigma$-equivariant isomorphism
$$
(C_{\delta \delta^{-\theta},G}^{\tau})_{\bar{F}_v} \tilde{\lto} C_{x_{\delta}x_{\delta}^{-\sigma},H}
$$
induced by the projection of $G_{\bar{F}_v} \cong H_{\bar{F}_v} \times H_{\bar{F}_v}$ onto the first factor.  It follows that the maximal $\sigma$-split subtorus of $(C^{\tau}_{\delta \delta^{-\theta},H})_{\bar{F}_v}$ has the same dimension as a maximal $\sigma$-split torus of $H_{\bar{F}_v}$.  The maximal $\sigma$-split torus of $(C^{\tau}_{\delta \delta^{-\theta},G})_{\bar{F}_v}$ is in fact defined over $F_v$ (compare \cite[\S 1.3]{Helmtori}), and the lemma follows.
\end{proof}

Our first step to proving Theorem \ref{thm-weak-match} is the following proposition:

\begin{prop} \label{prop-loc-const} Suppose that $\delta_0 \in G(F_v)$ is relatively $\tau$-regular semisimple.  Let
$\mathcal{V}$ be a neighborhood of $1$ in $\widetilde{T}_{\delta_0}(F_v)$ and
$\mathcal{W}$ be a neighborhood of $\delta_0$ in $G(F_v)$.
\begin{enumerate}
\item Suppose that $f \in C_c^{\infty}(G(F_v))$ has support in $\mathcal{W}$.  If $\mathcal{V}$ and $\mathcal{W}$ are sufficiently small, then
there is a function $\psi \in C_c^{\infty}(\widetilde{T}_{\delta_0}(F_v))$ with support in $\mathcal{V}$ such that if $\delta \in G(F_v)$ is $\tau$-regular semisimple and its $G^{\sigma} \times G^{\theta}(F_v)$-orbit meets $\mathcal{W}$, then
$$
\delta=g^{-1}_1t\delta_0g_2
$$
with $t \in \widetilde{T}_{\delta_0}(F_v)$, $(g_1,g_2) \in G^{\sigma} \times G^{\theta}(F_v)$, and
$$
TRO_{\delta}(f)=\psi(t).
$$
Moreover, if $t \in \widetilde{T}_{\delta_0}(F_v)$ and $t\delta_0$ is in the stable relative $\tau$-class of $\delta$, then
$$
STRO_{\delta}(f)=\psi(t).
$$

\item Conversely, if $\psi \in C_c^{\infty}(\widetilde{T}_{\delta_0}(F_v))$ has support in $\mathcal{V}$ and the
neighborhoods $\mathcal{V}$ and $\mathcal{W}$ are sufficiently small, then there is an
$f \in C_c^{\infty}(G(F_v))$ satisfying the identities of (1).
\end{enumerate}
\end{prop}

\noindent In \cite[\S 3.1]{Lab}, Labesse proves the analogue of Proposition \ref{prop-loc-const} in the context of the usual
trace formula.  Our proof follows his closely.
We require some preparatory lemmas:

\begin{lem} \label{lem-nb1} Assume $\delta_0$ is relatively $\tau$-regular semisimple.
There is an analytic subvariety $Y \subset G^{\sigma} \times G^{\theta}(F_v)$
that is symmetric (i.e. $(x,y) \in Y$ if and only if $(x^{-1},y^{-1}) \in Y$) and a neighborhood $\mathcal{V}$ of $1$ in $\widetilde{T}_{\delta_0}(F_v)$ such that
the following are true:
\begin{itemize}
\item[(i)] The map
\begin{align*}
Y \times \mathcal{V} &\lto G(F_v)\\
(x,y,t)&\longmapsto x^{-1}t\delta_0y
\end{align*}
is a diffeomorphism from $Y \times \mathcal{V}$ to a neighborhood $\mathcal{W}=\mathcal{W}(Y,\mathcal{V})$ of $\delta_0 \in G(F_v)$.
\item[(ii)]
 If $(x,y) \in Y$, then $x^{-1}t\delta_0y \in \mathcal{W}$ and $t \in \widetilde{T}_{\delta_0}(F_v)$ imply $t \in \mathcal{V}$.  Thus
 $$
 \widetilde{T}_{\delta_0}(F_v) \delta_0 \cap \mathcal{W} =\mathcal{V}\delta_0.
 $$
\end{itemize}

\end{lem}

\begin{proof}
Choose a complement $\mathfrak{n}$ of
$\mathfrak{c}=\mathrm{Lie}\textrm{ }C_{\delta_0\delta_0^{-\theta},G^{\sigma}}^{\tau}(F_v)$ in
$\mathfrak{g}^{\sigma}:=\mathrm{Lie}\textrm{ }G^{\sigma}(F_v)$:
$$
\mathfrak{g}^{\sigma}=\mathfrak{n} \oplus \mathfrak{c}.
$$
Let $\mathfrak{g}^{\theta}:=\mathrm{Lie}\textrm{ }G^{\theta}(F_v)$, let
$\OO$ be a neighborhood of $0$ in $\mathfrak{n} \oplus \mathfrak{g}^{\theta}$, and let
$Y=\mathrm{Exp}\textrm{ } \OO$ be its image under the exponential
map.  Thus $Y \subset G^{\sigma} \times G^{\theta}(F_v)$ is an analytic subvariety.
We claim that the analytic morphism
\begin{align} \label{loc-isom}
Y  \times \widetilde{T}_{\delta_0}(F_v)  &\lto G(F_v)\\
\nonumber (g_1,g_2,t) &\longmapsto (g_1^{-1}t\delta_0g_2)
\end{align}
is locally an isomorphism at $(1,1,1)$.

In order to show this, we first explain how the work of Rader and Rallis implies that the analytic morphism
\begin{align} \label{submers}
G^{\sigma} \times G^{\theta}(F_v) \times \widetilde{T}_{\delta_0}(F_v) &\lto G(F_v)
\\ \nonumber (g_1,g_2,t) &\longmapsto g_1^{-1}t\delta_0g_2
\end{align}
is a submersion whose fibers are of dimension $\dim_{F_v}C_{\delta_0\delta_0^{-\theta},G^{\sigma}}^{\tau}$.
If $\tau$ is trivial, then this follows from \cite[Theorem 3.4(1)]{RadRal}.  Assume that $\tau$ is not trivial.
Note that \eqref{submers} fits into a commutative diagram
\begin{align} \label{submers-diag}
\begin{CD}
G^{\sigma}  \times G^{\theta}(F_v) \times \widetilde{T}_{\delta_0}(F_v) @>{\hskip1.1in}>> G(F_v)\\
@VVV @VV{g \mapsto gg^{-\theta}}V\\
G^{\sigma}(F_v) \times \widetilde{T}_{\delta_0}(F_v) @>{(g_1,t) \mapsto g_1^{-1}t^2\delta_0\delta_0^{-\theta}g_1^{\tau}}>>  S(F_v)
\end{CD}
\end{align}
where the left arrow is the canonical projection and the top is \eqref{submers}.  The vertical arrows are induced by smooth surjective maps of affine varieties both of relative dimension $\dim_{F_v} G^{\theta}$.  Thus if we show that the bottom map is a submersion at (1,1), it follows that the top is a submersion at $(1,1,1)$.  Since the map $t \mapsto t^2$ on $\widetilde{T}_{\delta_0}$ is an isogeny, in order to show that the bottom map of \eqref{submers-diag} is a submersion at $(1,1)$
it suffices to show that the map
\begin{align} \label{submers2}
G^{\sigma}(F_v) \times \widetilde{T}_{\delta_0}(F_v) &\lto S(F_v)\\
\nonumber (g_1,t) &\longmapsto g_1^{-1}t\delta_0\delta_0^{-\theta}g_1^{\tau}
\end{align}
is a submersion at $(1,1)$.

  To see this, let $M_v$ be $F_v$ if $v$ splits in $M/F$ and $M \otimes_F F_v$ otherwise.  To prove that
\eqref{submers2} is a submersion whose fibers are of dimension $\mathrm{dim}_{F_v}C^{\tau}_{\delta_0\delta^{-\theta}_0,G^{\sigma}}$, it suffices to show that
\begin{align} \label{submers3}
G^{\sigma}(M_v) \times \widetilde{T}_{\delta_0}(M_v) &\lto S(M_v)
\\ \nonumber (g_1,t) &\longmapsto g_1^{-1}t\delta_0\delta_0^{-\theta}g_1^{\tau}
\end{align}
is a submersion at $(1,1)$ of relative dimension $\mathrm{dim}_{F_v}C_{\delta_0\delta_0^{-\theta},G^{\sigma}}^{\tau}$.  Choose an isomorphism $G_{M_v} \cong H_{M_v} \times H_{M_v}$ equivariant with respect to $\sigma$ and intertwining $\tau$ with $(x,y) \mapsto (y,x)$.  Using this isomorphism, write $\delta_0\delta_0^{-\theta}=(x_0,x_0^{-\sigma})$ for some $x_0 \in H(M_v)$ and let $T_{x_0}$ be the largest $\sigma$-split torus in $C_{x_0x_0^{-\sigma},H}$.
 There is a commutative diagram of analytic morphisms
\begin{align} \label{submers-diag2}
\begin{CD}
G^{\sigma}(M_v) \times \widetilde{T}_{\delta_0}(M_v) @>{\hskip1.0in}>> S(M_v)\\
@V{(u_1,u_2,t_1,t_2) \mapsto (u_1,u_2,t_1)}VV  @VV{(u_1,u_2) \mapsto u_1}V\\
H^{\sigma}(M_v)  \times H^{\sigma}(M_v)\times T_{x_{0}}(M_v) @>{(u_1,u_2,t) \mapsto u_1^{-1}tx_0u_2}>> H(M_v)
\end{CD}
\end{align}
where the top arrow is \eqref{submers3} and the vertical arrows are analytic isomorphisms induced by the isomorphism $G(M_v) \cong H(M_v) \times H(M_v)$.  Using Lemma \ref{lem-tdelta} and \cite[Theorem 3.4(1)]{RadRal} (and the reference therein for the real case), we see that the bottom map
is a submersion at $(1,1,1)$ of relative dimension $\mathrm{dim}_{F_v} C_{x_0x_0^{-\sigma},H^{\sigma}}$.
It follows that the top vertical map of \eqref{submers-diag2} is a submersion of relative dimension $\dim_{F_v}C_{x_0x_0^{-\sigma},H}$.  In view of the $\sigma$-equivariant isomorphism
$C_{\delta_0\delta_0^{-\theta},G}^{\tau} \tilde{\lto} C_{x_0x_0^{-\sigma},H}$ induced by the projection of $G_{M_v} \cong H_{M_v} \times H_{M_v}$ onto the first factor, this together with our comments above implies that \eqref{submers} is a submersion of relative dimension $\dim_{F_v}C_{\delta_0\delta_0^{-\theta},G^{\sigma}}^{\tau}$, as claimed.

The analytic subvariety
$$
Y \times \widetilde{T}_{\delta_0}(F_v) \subset
G^{\sigma} \times G^{\theta}(F_v) \times \widetilde{T}_{\delta_0}(F_v)
$$
is transverse to the subvariety
$$
X:=\{(g_1,g_2,1):(g_1,g_2) \in G_{\delta_0}(F_v)\} \subset
G^{\sigma} \times G^{\theta}(F_v) \times \widetilde{T}_{\delta_0}(F_v)
$$
at $(1,1,1)$.
The submersion \eqref{submers} maps $X$ identically to $\delta_0$. Our claim that \eqref{loc-isom} is a local
isomorphism follows.  Thus we can and do choose $\OO$ and $\mathcal{V}$ small enough so that the map
$$
Y \times \mathcal{V} \lto G(F_v)
$$
induced by restricting \eqref{loc-isom} is an isomorphism onto an open neighborhood $\mathcal{W}(Y,\mathcal{V})$ of
$\delta_0$.  This proves (i).

Fix a neighborhood $\mathcal{V}_1$ of $1$ in $\widetilde{T}_{\delta_0}(F_v)$ and an analytic subvariety $Y_1 \subset
G^{\sigma}  \times G^{\theta}(F_v)$ small enough that (i) holds.
For an analytic subvariety $Y \subset Y_1$ and a neighborhood
$\mathcal{V} \subset \mathcal{V}_1$ of $1$, let
$$
\mathcal{W}:=\mathcal{W}(Y,\mathcal{V})
$$ be the image of $Y \times
\mathcal{V}$ under \eqref{loc-isom}.  Since
$\widetilde{T}_{\delta_0}(F_v)$ is a closed subgroup of $G(F_v)$, we can choose $Y$ and $\mathcal{V}$
small enough so that
$$
\widetilde{T}_{\delta_0}(F_v) \delta_0 \cap g_1^{-1}\mathcal{W}g_2 \subset \mathcal{V}_1\delta_0
$$
for each $(g_1,g_2) \in Y$. \quash{Choose an open set $V \subset G(F_v)$ such that $V \cap \widetilde{T}_{\delta_0}(F_v)\delta_0=\mathcal{V}_1\delta_0$.
 Just let $Y \times Y \times \mathcal{V}$ be the inverse image of the open set $V$ under the map $(g_1,g_2,g_3,g_4,t) \mapsto g_1^{-1}g_3^{-1}t\delta_0g_4g_2$} Thus $g_1t\delta_0g_2^{-1} \in \mathcal{W}$ with $t \in \widetilde{T}_{\delta_0}(F_v)$ and $(g_1,g_2) \in Y$
implies that $t \in \mathcal{V}_1$.  By the proof of part (i), any $x \in \mathcal{W}(Y_1,\mathcal{V}_1)$ can be
written in a unique manner as $x=g_1^{-1}t\delta_0g_2$ with $(g_1,g_2) \in Y_1$ and
$t \in \mathcal{V}_1$.
We conclude that if $g_1^{-1}t\delta_0g_2 \in \mathcal{W}(Y,\mathcal{V})$ with $t \in
\widetilde{T}_{\delta_0}(F_v)$ and $(g_1,g_2) \in Y$ then $t \in \mathcal{V}$.
\end{proof}

\begin{lem} \label{lem-compact}
Suppose that $\mathcal{W} \subset G(F_v)$ is a relatively compact
neighborhood of $\delta_0$ and $\mathcal{V}$ is a neighborhood of $1$ in
$\widetilde{T}_{\delta_0}(F_v)$.  If $\mathcal{V}$ is sufficiently small, then there is a
compact subset $\Omega \subset G^{\sigma} \times G^{\theta}(F_v)$ such that if
$$
g_1^{-1}t\delta_0 g_2 \in \mathcal{W}
$$
with $t \in \mathcal{V}$ then $(g_1,g_2) \in G_{\delta_0}(F_v)\Omega$.
\end{lem}

\begin{proof} Consider the continuous map
\begin{align*}
A:G^{\sigma} \times G^{\theta}(F_v) \times \widetilde{T}_{\delta_0}(F_v)
&\lto G(F_v)\\
(g_1,g_2,t) &\longmapsto g_1^{-1}t\delta_0g_2
\end{align*}
and let
$$
P:G^{\sigma} \times G^{\theta}(F_v) \times \widetilde{T}_{\delta_0}(F_v)
\lto G^{\sigma} \times G^{\theta}(F_v)
$$
be the canonical projection.  The lemma is equivalent to the statement that if $\mathcal{V}$
and $\mathcal{W}$ are sufficiently small (with $\mathcal{W}$ relatively compact)
then there exists a compact set
$\Omega \subset G^{\sigma} \times G^{\theta}(F_v)$ such that $P \circ
A^{-1}\mathcal{W} \subset G_{\delta_0}(F_v) \Omega$.  This latter
statement is a consequence of \cite[Proposition 2.5]{RadRal} (and the reference therein in the real case); in loc.~cit.~one takes $M=G^{\sigma} \times
G^{\theta}(F_v)$ and $H=G_{\delta_0}(F_v)$.
\end{proof}

\begin{lem} \label{lem-314}Let $\mathcal{V}$ be an open neighborhood of $1$ in
$\widetilde{T}_{\delta_0}(F_v)$ and suppose $m,m' \in \mathcal{V}$.  If
$\mathcal{V}$ is sufficiently small and there is a $(g_1,g_2) \in
G^{\sigma} \times G^{\theta}(F_v)$ such that $g_1^{-1} m\delta_0
g_2=m'\delta_0$, then $(g_1,g_2) \in G_{\delta_0}(F_v)$.
\end{lem}

\begin{proof}
The proof of  \cite[Lemme 3.1.4]{Lab} easily adapts to our situation to prove the lemma.
We only note that the analogue of \cite[Lemme 3.1.2]{Lab} is trivial in our situation because
$\delta_0$ is relatively $\tau$-regular semisimple, and the analogues of \cite[Lemme 3.1.1 and Lemme 3.1.3]{Lab}
 are given by Lemma \ref{lem-nb1} and Lemma \ref{lem-compact}, respectively.
\end{proof}

\begin{lem} \label{lem-stable}
Under the assumptions of Lemma \ref{lem-314}, if $m'\delta_0$ and $m\delta_0$ are
in the same stable relative $\tau$-class, then
$$
\bar{g}_1^{-1}m\bar{g}_1=m'
$$
for some $\bar{g}_1\in C_{\delta_0\delta_0^{-\theta}}^{\tau}(\bar{F}_v)$.  Hence $m=m'$.
\end{lem}

\begin{proof}
Let $F_1/F_v$ be a finite field extension such that $C_{\delta_0\delta_0^{-\theta}F_1}^{\tau}$ is split.  Then if $m,m'$ have the property that there exists
$(\bar{g}_1,\bar{g}_2) \in G^{\sigma} \times G^{\theta}(\bar{F}_v)$
satisfying
 $\bar{g}_1^{-1}m\delta_0\bar{g}_2=m'\delta_0$,
we can and do assume that $(\bar{g}_1,\bar{g}_2) \in G^{\sigma} \times G^{\theta}(F_1)$ (compare \eqref{stableclassbij}).
Applying Lemma \ref{lem-314} ``over $F_1$'' we have that
$(\bar{g}_1,\bar{g}_2) \in
G_{\delta_0}(F_1)$.  Since $C_{\delta_0\delta_0^{-\theta}}^{\tau}$ is a torus, $\bar{g}_1$ commutes with $m$, which proves the last statement.
\end{proof}

\begin{lem} \label{lem-last} For $Y,\mathcal{V}$, and $\mathcal{W}(Y,\mathcal{V})$
as in Lemma \ref{lem-nb1}, suppose $Y \subset \mathcal{U}$ where $\mathcal{U}$ is an
open neighborhood of $1$ in $G^{\sigma} \times G^{\theta}(F_v)$.  If
$\mathcal{U}$ and $\mathcal{V}$ are sufficiently small, then each
$$
g_1^{-1}t\delta_0 g_2 \in \mathcal{W}(Y,\mathcal{V})
$$
with $(g_1,g_2) \in G^{\sigma} \times G^{\theta}(F_v)$ and $t \in
\mathcal{V}$ satisfies $(g_1,g_2) \in G_{\delta_0}(F_v)Y$.

\end{lem}

\begin{proof} By our hypothesis and Lemma \ref{lem-nb1}, we have
$$
g_1^{-1}t\delta_0g_2=y_1^{-1}t'\delta_0y_2
$$
for some $(y_1,y_2) \in Y$ and $t' \in \mathcal{V}$.  By Lemma \ref{lem-314}, if $\mathcal{U}$ and $\mathcal{V}$ are sufficiently small we have that
$(g_1y_1^{-1},g_2y_2^{-1}) \in G_{\delta_0}(F_v)$.
\end{proof}

With these lemmas in place, we can now prove Proposition \ref{prop-loc-const}:

\begin{proof}[Proof of Proposition \ref{prop-loc-const}]
Choose a neighborhood $\mathcal{V}$ of $1$ in $\widetilde{T}_{\delta_0}(F_v)$ and an analytic subvariety $Y \subset G^{\sigma} \times G^{\theta}(F_v)$ satisfying the conclusion of Lemma \ref{lem-nb1}.
Let $dg_1,dg_2$, and $dt_{\delta_0}$ be Haar measures on $G^{\sigma}(F_v)$, $G^{\theta}(F_v)$, and $G_{\delta_0}(F_v)$, respectively.  In this proof we use these measures to define $TRO_{\delta_0}(f)=TRO_{\delta_0}(f,dt_{\delta_0})$.  There is a natural map
$$
Y \lto G_{\delta_0}(F_v) \backslash G^{\sigma} \times G^{\theta}(F_v);
$$
let $d\mu$ be the measure on $Y$ that is the pullback of the measure $\frac{dg_1 dg_{2}}{dt_{\delta_0}}$ with respect to this map.  Suppose that $f \in C_c^{\infty}(G(F_v))$ has support in $\mathcal{W}:=\mathcal{W}(Y,\mathcal{V})$ and that $Y$ and $\mathcal{V}$ are sufficiently small.  For each $t \in \widetilde{T}_{\delta_0}(F_v)$, define
$$
\psi(t):=\int_Y f(g_1^{-1}t\delta_0 g_2)d\mu(g_1,g_2).
$$
By Lemma \ref{lem-nb1} the function $\psi$ on $\widetilde{T}_{\delta_0}(F_v)$ is smooth and compactly supported (with support in $\mathcal{V}$).  Since
$
G_{\delta_0}=G_{\delta}
$
we have
$$
TRO_{\delta}(f)=\int_{G_{\delta_0}(F_v) \backslash G^{\sigma}\times G^{\theta}(F_v)}  f(g_1^{-1}t\delta_0 g_2)  \frac{dg_1 dg_2}{dt_{\delta_0}}
$$
for some $t \in \mathcal{V}$.  In view of Lemma \ref{lem-last} this implies
\begin{align*}
TRO_{\delta}(f)=\int_{Y}  f(g_1^{-1}t \delta_0 g_2)  d\mu(g_1,g_2)
\end{align*}
and thus
$$
TRO_{\delta}(f)= \psi(t).
$$
The assertion involving stable twisted relative orbital integrals follows from Lemma
\ref{lem-stable}, and this completes the proof of (1).

For the proof of (2), suppose that we are given $\psi$ with support in $\mathcal{V}$.  Each $\delta \in \mathcal{W}$ can be written in a unique fashion as
$$
\delta=g_1^{-1}t \delta_0g_2
$$
with $(g_1,g_2) \in Y$ and $t \in \mathcal{V}$.  Define
$$
(J^{\beta}\psi)(\delta)=\beta(g_1,g_2) \psi(t)
$$
where $\beta \in C_c^{\infty}(Y)$ is chosen so that
$$
\int_Y \beta(g_1,g_2)d\mu(g_1,g_2)=1.
$$
Setting $f=J^{\beta}\psi$, statement (2) follows from the observation that
$$
\int_Y J^{\beta}\psi=\psi.
$$

\end{proof}

\subsection{Proof of Theorem \ref{thm-weak-match}}
For this entire subsection we assume the hypotheses of Theorem \ref{thm-weak-match}.  We assume $f$ is supported in a neighborhood $\mathcal{W}=\mathcal{W}(Y,\mathcal{V})$ of a
relatively $\tau$-regular semisimple element $\delta_0$, with $Y$ and $\mathcal{V}$  as in Lemma \ref{lem-nb1}.  With notation and assumptions as in Proposition \ref{prop-loc-const}, we have that
$$
STRO_{\delta}(f)=0
$$
if the stable class of $\delta$ does not meet $\widetilde{T}_{\delta_0}(F_v)\delta_0$, and
$$
STRO_{\delta}(f)=\psi(t)
$$
if $\delta$ is in the stable class of $t\delta_0$ with $t \in \widetilde{T}_{\delta_0}(F_v)$.

In a sufficiently small neighborhood of the identity in $\widetilde{T}_{\delta_0}(F_v)$, the
 map $t \mapsto t^2$ is an isomorphism that preserves conjugacy classes.
Thus we can and do choose a function $\psi_1 \in
C_{c}^{\infty}(\widetilde{T}_{\delta_0}(F_v))$ with support in a small neighborhood
$\mathcal{V}_1$ of the identity such that
$$
\psi(t)=\psi_1(t^2).
$$
By shrinking $\mathcal{V}$ if necessary, we can make $\mathcal{V}_1$ as small as we wish.

Assume that $\delta_0$ has norm
$\gamma_0 \in H(F_v)$.  Thus $C_{\delta_0\delta_0^{-\theta},G}^{\tau}$ and $C_{\gamma_0\gamma_0^{-\sigma},H}$ are inner twists of each other, and the inner twist can be defined by an element of $H^{\sigma}(\bar{F}_v)$ (compare \eqref{norm-twist}).
Since $\delta_0$ is relatively $\tau$-regular,
$C_{\delta_0\delta_0^{-\theta}}^{\tau}$ is a torus, and thus the inner twist induces a $\sigma$-equivariant isomorphism
\begin{align} \label{B23-map}
B:C_{\delta_0\delta_0^{-\theta},G}^{\tau} \lto C_{\gamma_0\gamma_0^{-\sigma},H}.
\end{align}
Note that $\widetilde{T}_{\delta_0}$ is a maximal $\sigma$-split torus of $C_{\delta_0\delta_0^{-\theta},G^{\sigma}}^{\tau}$.
 Since $T_{\gamma_0}$ and $\widetilde{T}_{\delta_0}$ have the same dimension by
 Lemma \ref{lem-tdelta}, the isomorphism \eqref{B23-map} induces another isomorphism
$$
B:\widetilde{T}_{\delta_0} \lto T_{\gamma_0}
$$
such that for $F_v$-algebras $R$ one has $B(g_1^{-1}tg_1)=B(g_1)^{-1}B(t)B(g_1)$ for $g_1 \in C_{\delta_0\delta_0^{-\theta},G^{\sigma}}^{\tau}(R)$ and $t \in \widetilde{T}_{\delta_0}(R)$.
Set $\psi_2:=\psi_1\circ B^{-1} \in C_c^{\infty}(T_{\gamma_0}(F_v))$.  Then $\psi_2$ has
support in $B(\mathcal{V}_1)$, and
\begin{align} \label{sums-eq}
\psi_1(t^2)=\psi_2(B(t^2)).
\end{align}
We note that $B(t^2)\gamma_0$ is a norm of $t\delta_0$ (this is the reason for employing the
squaring map above).

Invoking Proposition \ref{prop-loc-const}(2),
we
shrink
 $\mathcal{V}$ and $\mathcal{V}_1$ if necessary and choose a function $\Phi \in
C_c^{\infty}(H(F_v))$ such that
$$
SRO_{\gamma}(\Phi)=0
$$
if the stable class of $\gamma$ does not meet $T_{\gamma_0}(F_v)\gamma_0$, and
$$
SRO_{\gamma}(\Phi)=\psi_2(t_2)
$$
if $t_2\gamma_0$ and $\gamma$ are in the same stable class.  This implies the assertion
 of the theorem. \qed

\section{Spherical characters} \label{sec-sph}
In this section we introduce the notion of a relatively $\tau$-regular and relatively $\tau$-elliptic representation of a reductive group over a local field.  The first of these notions was used in the statement of Theorem \ref{intro-thm2}, and we believe the second to be of interest as well.
 Let $\Pi_v$ be an irreducible admissible representation of $G(F_v)$ with space $V_{\Pi_v}$ and choose
$$
\Lambda=\sum_i \lambda_i \otimes \lambda^{\vee}_i \in \mathrm{Hom}_{G^{\sigma}(F_v)}(V_{\Pi_v},\CC) \otimes\mathrm{Hom}_{G^{\theta}(F_v)}(V_{\Pi_v^{\vee}},\CC);
$$\index{$\Lambda$}here $\CC$ is the trivial representation and $\Pi_v^{\vee}$ is the contragredient representation acting on $V_{\Pi_v^{\vee}}$.
The linear form $\Lambda$ defines a distribution (i.e. a linear map)
\begin{align}
\Theta_{\Lambda}:C_c^{\infty}(G(F_v)) &\lto \CC\\
\nonumber f_v &\longmapsto \sum_i\langle \Pi_v(f_v)\lambda_i, \lambda^{\vee}_i\rangle.
\end{align}\index{$\Theta_{\Lambda}$}A \textbf{spherical matrix coefficient} (of $\Pi_v$) is a distribution attached to $\Lambda$ in this manner (compare \cite{HakimSCHP}).  If the extension $M/F$ is trivial (so $H=G$ and $G^{\sigma}=G^{\theta}=H^{\sigma}$) then the spherical matrix coefficient $\Theta_{\Lambda}$ is representable by a locally constant function on
the relatively regular subset of $G(F_v)=H(F_v)$ by \cite[Lemma 6]{HakimSCHP}.  We denote this function by $\Theta_{\Lambda}$ as well.  The authors suspect that the same result is true when $\tau$ is nontrivial, but we will not prove this.

\begin{defn}
An admissible representation $\Pi_v$ of $G(F_v)$ is \textbf{relatively $\Lambda$-regular} if there is a function $f \in C_c^{\infty}(G(F_v))$ supported in the set of relatively $\tau$-regular elements of $G(F_v)$ admitting norms in $H(F_v)$ such that
$$
\Theta_{\Lambda}(f) \neq 0.
$$
  It is \textbf{relatively $\tau$-regular} if it is relatively $\Lambda$-regular for all nonzero
$$
 \Lambda \in  \mathrm{Hom}_{G^{\sigma}(F_v)}(V_{\Pi_v},\CC) \otimes \mathrm{Hom}_{G^{\theta}(F_v)}(V_{\Pi_v^{\vee}},\CC).
$$
\end{defn}

\begin{defn}
An admissible representation $\Pi_v$ of $G(F_v)$ is \textbf{relatively $\Lambda$-elliptic} if there is a function $f \in C_c^{\infty}(G(F_v))$ supported in the set of relatively $\tau$-elliptic semisimple elements of $G(F_v)$ admitting norms in $H(F_v)$ such that
$$
\Theta_{\Lambda}(f) \neq 0.
$$
 It is \textbf{relatively $\tau$-elliptic} if it is relatively $\Lambda$-elliptic for all nonzero
$$
 \Lambda \in  \mathrm{Hom}_{G^{\sigma}(F_v)}(V_{\Pi_v},\CC) \otimes \mathrm{Hom}_{G^{\theta}(F_v)}(V_{\Pi_v^{\vee}},\CC).
$$
\end{defn}
\noindent  If $\tau$ is trivial then we omit it from notation, writing relatively regular for relatively $\tau$-regular.  The analogous convention with regular replaced by elliptic will also be in force.  If $\tau$ is trivial the norm map is the identity map so the condition on elements admitting norms can be omitted.

\begin{rem}
In certain circumstances one can use the simple twisted relative trace formula of \cite{Hahn} to prove the existence of relatively $\tau$-regular (resp.~$\tau$-semisimple) representations.
\end{rem}

It is well known that the character of an irreducible admissible representation does not vanish on the regular semisimple set.  Thus if $\tau$ is trivial, $G:=H=H^{\sigma} \times H^{\sigma}$ and $\sigma:H \to H$ is the automorphism switching the two factors every irreducible admissible representation is relatively regular.  We do not know if the same statement is true in general, and one has to be cautious given \cite[\S 4]{RadRal}.  However, it seems likely that in the settings of interest to this paper every irreducible admissible representation arising as a local factor of a cuspidal automorphic representation that is both $G^{\sigma}$ and $G^{\theta}$-distinguished is relatively $\tau$-regular.

\section{Prestabilization of a single stable relative orbital integral} \label{sec-prestab}

In this section, following work of Langlands, Kottwitz, Shelstad, and Labesse, we define
stable relative and stable twisted relative global orbital integrals and show how they decompose into a sum of global
relative $\kappa$-orbital integrals.  The main result is Proposition \ref{prop-prestab}.
As anyone familiar with the usual (not relative) stable trace formula could guess,
the reason for introducing the relative $\kappa$-orbital integrals is that they factor into local $\kappa$-orbital integrals.  This makes
it possible to apply the local matching theory developed in \S \ref{sec-norm-maps} and \S \ref{sec-matching}.
Our treatment follows \cite{Lab}, and we refer to loc.~cit.~ for notation involving abelianized cohomology of reductive groups and quotients.
We emphasize that we do \emph{not} attempt to write the $\kappa$-orbital integrals given below in terms of stable relative orbital integrals
on other groups; this is why the section is entitled ``Prestabilization...'' instead of ``Stabilization...''  In \S \ref{sec-group-ell} below,
we show how to collect the relatively elliptic terms of the relative trace formula together.

We should note that in the biquadratic, non-twisted case, a stabilization dependent on various conjectural fundamental lemmas
was given by Flicker in \cite{FlickerStable}, together with a conjectural definition
of relative transfer factors.

\subsection{Global relative orbital integrals}

Let $\Phi \in C^{\infty}_c(H(\A_F))$, $f \in C^{\infty}_c(G(\A_F))$.  Moreover let $\gamma \in
H(F)$ (resp. $\delta \in G(F)$) be a relatively semisimple element (resp.~relatively $\tau$-semisimple element).  We define the (global) \textbf{relative orbital integral}
\begin{align} \label{goi}
RO_{\gamma}(\Phi):=RO_{\gamma}(\Phi,dt_{\gamma}):=
\iint_{H_{\gamma}(\A_F) \backslash
H^{\sigma}(\A_F)^2}\Phi(h_1^{-1}\gamma h_2)
\frac{dh_1dh_2}{dt_{\gamma}}
\end{align}\index{$RO_{\gamma}(\Phi)$}where $dh_i=\otimes_v'dh_{i,v}$ and $dt_{\gamma}=\otimes_v' dt_{\gamma,v}$ are Haar measures on
$H^{\sigma}(\A_F)$ and $H_{\gamma}(\A_F)$, respectively.  Similarly, we define the (global) \textbf{twisted relative orbital integral}
\begin{align} \label{gtoi}
TRO_{\delta}(f):=TRO_{\delta}(f,dt_{\delta}):=
\iint_{G_{\delta}(\A_F) \backslash G^{\sigma}\times G^{\theta}(\A_F)}
f(g_1^{-1} \delta g_2)\frac{dg_1dg_2}{dt_{\delta}}
\end{align}\index{$TRO_{\delta}(f)$}where $dg_i=\otimes_v'dg_{i,v}$ and $dt_{\delta}:=\otimes_v'dt_{\delta,v}$ are Haar measures on $G^{\sigma}(\A_F)$ and $G_{\delta}(\A_F)$, respectively.
If
$\Phi=\otimes'_v\Phi_v$ is factorable then
\begin{align}
RO_{\gamma}(\Phi,dt_{\gamma})=\prod_v RO_{\gamma_v}(\Phi_v,dt_{\gamma,v}).
\end{align}
If  $f=\otimes_v' f_v$ is factorable then
\begin{align}
TRO_{\delta}(f,dt_{\delta})=\prod_vTRO_{\delta_v}
(f_v,dt_{\delta,v}).
\end{align}

Let $\gamma_0 \in H(F)$ and $\delta_0 \in G(F)$ be relatively semisimple and relatively $\tau$-semisimple, respectively.  The
(global) \textbf{stable relative orbital integral} is
\begin{align} \label{gsoi}
SRO_{\gamma_0}(\Phi):=\sum_{\gamma \sim \gamma_0}
RO_{\gamma}(\Phi,dt_{\gamma})
\end{align} \index{$SRO_{\gamma_0}(\Phi)$}where the sum is over relatively semisimple $\gamma \in H(F)$ in the same stable relative class as $\gamma_0$.
The (global) \textbf{stable twisted relative orbital integral} is
\begin{align}
STRO_{\delta_0}(f):=
\sum_{\delta \sim \delta_0}TRO_{\delta}(f,dt_{\delta})
\end{align} \index{$STRO_{\delta_0}(f)$}where the sum is over relatively $\tau$-semisimple $\delta \in G(F)$ in the same stable relative $\tau$-class as $\delta_0$.
Here we assume that the
$dt_{\gamma}=\otimes_v'dt_{\gamma,v}$ (resp.~$dt_{\delta}=\otimes_v'dt_{\delta,v}$) are compatible in the sense that for each
$v$ the $dt_{\gamma,v}$ (resp.~$dt_{\delta,v}$) are compatible as in \S \ref{ssec-matching}.

In analogy with \eqref{localKgroup},  for a pair of (connected) reductive $F$-groups $I \leq H$ we set notation for the abelian group
\begin{align*}
\mathfrak{K}(I,H;F):&=H^0_{\mathrm{ab}}(\A_F/F,I \backslash H)^D.
\end{align*} \index{$\mathfrak{K}(I,H;F)$}For each place $v$ there is a localization map
\begin{align} \label{loc-map}
 \mathfrak{K}(I,H;F) &\lto \mathfrak{K}(I,H;F_v)\\
\nonumber \kappa_H &\longmapsto \kappa_{Hv}
\end{align}
defined as the dual of
$$
H^0_{\mathrm{ab}}(F_v,I \backslash H) \lto H^0_{\mathrm{ab}}(\A_F,I \backslash H) \lto H^0_{\mathrm{ab}}(\A_F/F,I \backslash H).
$$
Let $\Phi=\otimes_v'\Phi_v \in C_c^{\infty}(H(\A_F))$
and $f =\otimes_v'f_v \in C_c^{\infty}(G(\A_F))$ be factorable, and let
$\gamma \in H(F)$ (resp.~$\delta \in G(F)$) be relatively semisimple
(resp.~relatively $\tau$-semisimple).  Finally, let
$\kappa_H \in
\mathfrak{K}(H^{\sigma}_{\gamma},H^{\sigma} \times H^{\sigma};F)$ and
$\kappa\in \mathfrak{K}(G_{\delta},G^{\sigma} \times G^{\sigma};F)$.  We then set
\begin{align}
RO^{\kappa_H}_{\gamma}(\Phi):&=\prod_vRO^{\kappa_{Hv}}_{\gamma_v}(\Phi_v)\\
\nonumber TRO^{\kappa}_{\delta}(f):&=\prod_v
TRO^{\kappa_v}_{\gamma_v}(f_v),
\end{align}\index{$RO^{\kappa_H}_{\gamma}(\Phi)$}\index{$TRO^{\kappa}_{\delta}(f)$}whenever this product is well-defined (i.e. convergent).  The following proposition ensures convergence:

\begin{prop}  \label{prop-1-ae}
Let $\gamma_0 \in H(F)$ (resp. $\delta_0 \in G(F)$) be relatively regular semisimple
(resp.~relatively $\tau$-regular semisimple) and let
$\kappa_H \in \mathfrak{K}(H_{\gamma_0},H^{\sigma} \times H^{\sigma};F)$
(resp.~$\kappa \in \mathfrak{K}(G_{\delta_0},G^{\sigma} \times G^{\theta};F)$).  Moreover, let
$\Phi=\otimes_v'\Phi_v \in C_c^{\infty}(H(\A_F))$ and $f=\otimes_v'f_v \in C_c^{\infty}(G(\A_F))$.
There is a finite set $S$ of places of $F$ such that if  $v \not \in S$  then
$RO_{\gamma_{0v}}^{\kappa_{Hv}}(\Phi_v)=1$ (resp. $TRO_{\delta_{0v}}^{\kappa_{v}}(f_v)=1$).
\end{prop}
We note that this is a weak relative analogue of the results of \cite[\S 7]{KottEllSing}.
\begin{proof}
Note that
$(G_{\delta})_{F_v}=G_{\delta F_v}$ is quasi-split for almost all $v$, and hence $e(G_{\delta F_v})=1$ for
almost all $v$, and the character $\kappa_v$ is trivial
on the intersection of the $G^{\sigma} \times G^{\theta}$-orbit of $\delta_v$ and the support of $f_v$
for almost all $v$ by the definition of the (restricted direct) topology on
$H^0_{ab}(\A_F,G_{\gamma}\backslash G^{\sigma} \times {G}^{\theta})$ \cite[\S 1.4-1.8]{Lab} and the proof of \cite[Proposition 3.4]{Hahn}.  With this
in mind, the proposition follows immediately from \cite[Proposition 3.2]{Hahn} and the proof of \cite[Proposition 3.4]{Hahn}.
\end{proof}

\subsection{Some cohomology groups}
\label{ssec-ch-groups}
In this subsection we collect notation for some Galois cohomology groups that will be used in the following subsections.
Let $v$ be a place of $F$ and let $I \leq H$ be a pair of connected reductive $F_v$-groups. Set
\begin{align} \label{elocdef}
\mathfrak{E}(I,H;F_v):=
\mathrm{ker}\left[H^1_{\mathrm{ab}}(F_v,I)
\to H^1_{\mathrm{ab}}(F_v,H)\right].
\end{align}\index{$\mathfrak{E}(I,H;F_v)$}There is a natural map
\begin{align*}
H^0_{\mathrm{ab}}(F_v,I \backslash H)
 &\lto \mathfrak{E}(I,H;F_v)
\end{align*}
induced by the long exact sequence attached to the crossed $\Gal(\bar{F}_v/\bar{F}_v)$-module $I(\bar{F}_v) \to H(\bar{F}_v)$ \cite[p.~20]{Lab}.
The image of the dual map
\begin{align*}
\mathfrak{E}(I,H;F_v)^D
&\lto \mathfrak{K}(I,H;F_v)
\end{align*}
is denoted $\mathfrak{K}(I,H;F_v)_1$. \index{$\mathfrak{K}(I,H;F_v)_1$}
Similarly, if $I \leq H$ is a pair of connected reductive $F$-groups, write
\begin{align*}
\mathfrak{E}(I,H;\A_F/F):&=
\mathrm{coker}\left[ H^0_{\mathrm{ab}}(\A_F,H) \to
H^0_{\mathrm{ab}}(\A_F/F,I\backslash H)\right].
\end{align*}\index{$\mathfrak{E}(I,H;\A_F/F)$}There is a natural quotient map
\begin{align*}
H^0_{\mathrm{ab}}(\A_F/F,I \backslash H)
 &\lto \mathfrak{E}(I, H;\A_F/F).
\end{align*}
The image of the dual map
\begin{align*}
\mathfrak{E}(I, H;\A_F/F)^D
&\lto \mathfrak{K}(I,H;F)
\end{align*}
is denoted $\mathfrak{K}(I,H;F)_1$. \index{$\mathfrak{K}(I,H;F)_1$}
  The localization map \eqref{loc-map} induces a homomorphism
\begin{align} \label{loc-map-1}
\mathfrak{K}(I,H;F)_1
&\lto \prod_v \mathfrak{K}(I,H;F_v)_1
\end{align}
(compare \cite[p. 43]{Lab}).  The kernel of this map is denoted
$\mathfrak{K}(I,H;F)_0$. \index{$\mathfrak{K}(I,H;F)_0$}

\subsection{Prestabilization of a single relatively elliptic term}

For a reductive $F$-group $H$, write $\tau(H)$ for the Tamagawa number of $H$ (this
$\tau$ should not be confused with the Galois automorphism $\tau$ from above).  Finally, for a pair of reductive $F$-groups $I$ and $H$, write
\begin{align}
d(I,H):=\# \mathrm{coker}\left[H^1_{\mathrm{ab}}(\A_F/F,I) \to H^1_{\mathrm{ab}}(\A_F/F,H)\right].
\end{align} \index{$d(I,H)$}

\noindent
The main result of this section is the following adaptation of the work of Langlands, Kottwitz, and Labesse
\cite{LanglStab}, \cite{KottStCusp},
\cite{KottEllSing}, \cite{Lab} to our situation:

\begin{prop} \label{prop-prestab} Let $\Phi=\otimes_v'\Phi_v \in C_c^{\infty}(H(\A_F))$ and $f=\otimes_v'f_v \in C_c^{\infty}({G}(\A_F))$ be factorable.
If $\gamma \in H(F)$ is relative regular and relatively elliptic then
\begin{align*}
SRO_{\gamma}(\Phi)&= \frac{\tau(H^{\sigma} \times H^{\sigma})}{\tau(H_{\gamma}) d(H_{\gamma},H^{\sigma}\times H^{\sigma})}
\sum_{\kappa_H}RO^{\kappa_H}_{\gamma}(\Phi)\\
&=\frac{\tau(H^{\sigma} \times H^{\sigma})}{\tau(H_{\gamma}) d(H_{\gamma},H^{\sigma}\times H^{\sigma})}
\sum_{\kappa_H}\prod_{v} RO^{\kappa_{Hv}}_{\gamma_v}(\Phi_v)
\end{align*}
where the sum is over $\kappa_H \in \mathfrak{K}(H_{\gamma},H^{\sigma} \times H^{\sigma};F)_1$.

Similarly, if $\delta \in {G}(F)$ is relatively $\tau$-regular and relatively elliptic then
\begin{align*}
STRO_{\delta}(f)&= \frac{\tau({G}^{\sigma}\times G^{\theta})}{\tau(G_{\gamma})
  d(G_{\delta} ,G^{\sigma} \times G^{\sigma})}
\sum_{\kappa}TRO^{\kappa}_{\delta}(f)\\
&=\frac{\tau({G}^{\sigma} \times G^{\theta})}{\tau(G_{\delta}) d(G_{\delta},{G}^{\sigma}\times G^{\theta})}
\sum_{\kappa}\prod_{v} TRO^{\kappa_v}_{\gamma_v}(f_v)
\end{align*}
where the sum is over
$\kappa \in \mathfrak{K}({G}_{\gamma},{G}^{\sigma} \times {G}^{\theta};F)_1$.
\end{prop}

\begin{proof}
We begin by recalling that $H_{\gamma}$ and $G_{\delta}$ are connected, and hence the groups
\begin{align*}
\mathfrak{K}(H_{\gamma},H^{\sigma} \times H^{\sigma};F)_1&=\mathfrak{E}(H_{\gamma},H^{\sigma} \times H^{\sigma};\A_F/F)^D\\
\mathfrak{K}(G_{\delta},{G}^{\sigma} \times {G}^{\theta};F)_1&=
\mathfrak{E}(G_{\delta},{G}^{\sigma} \times {G}^{\theta};\A_F/F)^D
\end{align*}
are finite \cite[Proposition 1.8.4]{Lab}. Moreover, the $\kappa$-orbital integrals converge by Proposition \ref{prop-1-ae}.  With these observations
 in mind, the proof is a standard consequence of the Fourier transform on a finite
group. In more detail, one combines the first exact sequence on
 \cite[p.~42]{Lab} with \cite[Proposition 1.8.5 and 1.8.6]{Lab} (compare the proof of \cite[Proposition 4.2.1]{Lab}).
\end{proof}

\section{Grouping relatively elliptic terms} \label{sec-group-ell}

In this section, we group together the relatively elliptic and relatively $\tau$-elliptic portions of the relative trace
formula and twisted relative trace formula, respectively.

\subsection{Haar measures} \label{ssec-measures}

Recall the Harish-Chandra subgroups $^1H(\A_F)$ and the central subgroups  $A_H \leq H(F \otimes_{\QQ} \RR)$ of \S \ref{HC-subgroup}.  For every $\gamma \in H(\A_F)$ and $\delta \in G(\A_F)$ write
{\allowdisplaybreaks \begin{align*}
^2H^{\sigma}(\A_F) :&=H^{\sigma}(\A_F) \cap {}^1H(\A_F) \\
^2G^{\sigma}(\A_F):&=G^{\sigma}(\A_F) \cap {}^1G(\A_F) \textrm{ } \textrm{ }\\
^2G^{\theta}(\A_F):&=G^{\theta}(\A_F) \cap {}^1G(\A_F)\\
^2H_{\gamma}(\A_F):&=H_{\gamma}(\A_F) \cap {}^1H(\A_F) \times {}^1H(\A_F)\\
^2G_{\delta}(\A_F):&=G_{\delta}(\A_F) \cap {}^1G(\A_F) \times {}^1G(\A_F).
\end{align*}}\index{$^2H^{\sigma}(\A_F)$}\index{$^2G^{\sigma}(\A_F)$}\index{$^2G^{\theta}(\A_F)$}\index{$^2H_{\gamma}(\A_F)$}\index{$^2G_{\delta}(\A_F)$}Moreover, write
{\allowdisplaybreaks \begin{align*}
A_{H}^{\sigma}:&=H^{\sigma}(F \otimes_{\QQ}\RR) \cap A_H\\
A_{G}^{\sigma}:&=G^{\sigma}(F \otimes_{\QQ}\RR) \cap A_{G}\\
 A_{G}^{\theta}:&=G^{\theta}(F \otimes_{\QQ}\RR) \cap A_{G}
\end{align*}}\index{$A_H^{\sigma}$}\index{$A_G^{\sigma}$}\index{$A_G^{\theta}$}and
{\allowdisplaybreaks \begin{align*}
A:&=\{(x,x) \in A_{H}^{\sigma} \times A_{H}^{\sigma}\}\\
\widetilde{A}:&=\{(z,z) \in A_{G}^{\sigma} \times A_{G}^{\theta}:z \in A_{G}^{\sigma} \cap A_{G}^{\theta}\}.
\end{align*}}\index{$A$}\index{$\widetilde{A}$}We then have decompositions
{\allowdisplaybreaks \begin{align} \label{quot-decomp}
H_{\gamma}(\A_F) \backslash H^{\sigma}(\A_F) \times H^{\sigma}(\A_F) & \cong
\left(A \backslash A_{H}^{\sigma} \times A_{H}^{\sigma} \right) \times \left({}^2H_{\gamma}(\A_F) \backslash {}^2H^{\sigma}(\A_F) \times {}^2H^{\sigma}(\A_F) \right)\\
\nonumber G_{\delta}(\A_F) \backslash G^{\sigma}(\A_F) \times G^{\theta}(\A_F) &=(\widetilde{A} \backslash A_{G}^{\sigma} \times A_{G}^{\theta}) \times \left({}^2G_{\delta}(\A_F) \backslash {}^2 G^{\sigma}(\A_F)\times {}^2G^{\theta}(\A_F)\right).
\end{align}}

We now specify a choice of Haar measure on all of the groups appearing in \eqref{quot-decomp}, at least if $\gamma$ is relatively elliptic (resp.~$\delta$ is relatively $\tau$-elliptic).
Whenever a reductive $F$-group $H$ appears, we give $H(\A_F)$, $A_H$, and ${}^1H(\A_F)$ the measures that are
used in the definition of the Tamagawa number $\tau(H)$.  We then fix, once and for all, measures
 $dz_1=dz_2$ for $A_{H}^{\sigma}$, $d\widetilde{z}$ for $A_{G}^{\sigma}$ and $d\widetilde{z}_{\theta}$ for $A_{G}^{\theta}$, and stipulate that the isomorphisms
{\allowdisplaybreaks \begin{align*}
{}^2H^{\sigma}(\A_F) &\cong A_{H^{\sigma}}/A_H^{\sigma} \times {}^1H^{\sigma}(\A_F)\\
{}^2G^{\sigma}(\A_F) &\cong A_{G^{\sigma}}/A_{G}^{\sigma} \times {}^1G^{\sigma}(\A_F)\\
{}^2G^{\theta}(\A_F) &\cong A_{G^{\theta}}/A_{G}^{\theta} \times {}^1G^{\theta}(\A_F)
\end{align*}}are measure preserving.
Notice that if $\gamma$ is relatively elliptic (resp. $\delta$ is relatively $\tau$-elliptic) then
\begin{align*}
{}^2H_{\gamma}(\A_F)&={}^1H_{\gamma}(\A_F)\\
{}^2G_{\delta}(\A_F)&= {}^1G_{\delta}(\A_F).
\end{align*}
By our earlier convention, we have already endowed ${}^2H_{\gamma}(\A_F)$, ${}^2G_{\delta}(\A_F)$,
$A=A_{H_{\gamma}}$ and $\widetilde{A}=A_{G_{\delta}}$ with measures.   Altogether, this endows all of the groups occurring in \eqref{quot-decomp} with measures.  Thus
 ${}^2H_{\gamma}(\A_F)$   and ${}^2G_{\delta}(\A_F)$ are given the unique Haar measures such that if $\gamma$ is relatively elliptic and $\delta$ is relatively $\tau$-elliptic then
\begin{align*}
\mathrm{vol}(H_{\gamma}(F) \backslash {}^2H_{\gamma}(\A_F))&=\tau(H_{\gamma})\\
\mathrm{vol}(G_{\delta}(F) \backslash
{}^2G_{\delta}(\A_F))&=\tau(G_{\delta}).
\end{align*}

For the rest of this paper, we use these choices of measures when we form relative orbital integrals and twisted relative orbital integrals.
These measures will be compatible for the same reason that the Tamagawa measures in the usual trace formula are compatible, namely that
the Tamagawa numbers of two inner forms of the same quasi-split reductive group are equal \cite{KottTama} \cite{Cherno}.

For $\Phi \in C_c^{\infty}(H(\A_F))$ and $f \in C_c^{\infty}(G(\A_F))$ we define
\begin{align} \label{fones}
\Phi^1(x):& = \int_{A \backslash A_{H}^{\sigma} \times A_{H}^{\sigma}} \Phi(z_1^{-1}z_2x) \frac{dz_1dz_2}{du}\\
\nonumber f^1(x):&= \int_{\widetilde{A} \backslash A_{G}^{\sigma} \times A_{G}^{\theta}} f(\widetilde{z}^{-1} \widetilde{z}_{\theta}x) \frac{d\widetilde{z}d\widetilde{z}_{\theta}}{dt}.
\end{align} \index{$\Phi^1$} \index{$f^1$}

\subsection{Elliptic kernels}

Let $\Phi \in C_c^{\infty}(H(\A_F))$, $f \in C_c^{\infty}(G(\A_F))$ and consider the kernel functions
\begin{align*}
K_{\Phi}(x,y):=\sum_{\gamma} \Phi^1(x^{-1} \gamma y) :{}^1H(\A_F) \times {}^1H(\A_F) &\lto \CC\\
K_{f}(x,y):=\sum_{\delta} f^1(x^{-1} \delta y):{}^1G(\A_F) \times {}^1G(\A_F)
&\lto \CC
\end{align*}
\index{$K_{\Phi}$} \index{$K_f$}

\noindent
where the first sum is over relatively regular elliptic $\gamma \in H(F)$ and the second is over
 relatively $\tau$-regular elliptic $\delta \in G(F)$.   Define the integrals
\begin{align*}
RT_e(\Phi):&=\iint_{(H^{\sigma}(F) \backslash {}^{2}H^{\sigma}(\A_F))^2} K_{\Phi}(h_1,h_2) dh_1 dh_2\\
TRT_e(f):&=\iint_{G^{\sigma}(F) \backslash {}^{2}G^{\sigma}(\A_F) \times G^{\theta}(F) \backslash
{}^{1}G^{\theta}(\A_F)}K_{f}(g_1,g_2)dg_1 dg_2.
\end{align*} \index{$RT_e(\Phi)$} \index{$TRT_e(f)$}

\noindent
Here the $dh_i$ are both induced by the measure on $ {}^{2}H^{\sigma}(\A_F)$ fixed above
and $dg_1,dg_2$ are induced by the measures on  ${}^{2}G^{\sigma}(\A_F)$ and ${}^{2}G^{\theta}(\A_F)$ fixed above, respectively. These integrals are absolutely convergent (see the proof of \cite[Theorem 4.1]{Hahn}).

\subsection{Stable geometric expansions} \label{ssec-st-geom-exp}

 Using standard manipulations (compare \cite{Hahn}), we rewrite
\begin{align} \label{1stgeomsum}
RT_e(\Phi)&=\sum_{\gamma_0} \sum_{\gamma \sim \gamma_0} a(\gamma)RO_{\gamma}(\Phi)\\
\nonumber TRT_e(f)&=\sum_{\delta_0} \sum_{\delta \sim \delta_0}a^{\tau}(\delta)TRO_{\delta}(f)
\end{align}
where the exterior sums are over a set of representatives for the stable relatively regular elliptic classes (resp. stable relatively $\tau$-regular elliptic classes) and the interior sums are over a set of
representatives for the relative classes (resp. relative $\tau$-classes) in the stable relative
class of $\gamma_0$ (resp. stable relative $\tau$-class of $\delta_0$).
Here
\begin{align}
a(\gamma):&=\mathrm{vol}(H_{\gamma}(F)\backslash {}^{2}H_{\gamma}(\A_F))=\mathrm{vol}(H_{\gamma}(F)\backslash {}^{1}H_{\gamma}(\A_F))=\tau(H_{\gamma})\\
\nonumber a^{\tau}(\delta):&=\mathrm{vol}(G_{\delta}(F)
\backslash {}^{2}G_{\delta}(\A_F))=\mathrm{vol}(
G_{\delta}(F)
\backslash {}^{1}G_{\delta}(\A_F))=\tau(G_{\delta}),
\end{align}\index{$a(\gamma)$}\index{$a^{\tau}(\delta)$}where the volumes are taken with respect to the measures fixed in \S \ref{ssec-measures}.
Here, as in \S \ref{ssec-measures}, we are using the fact that $\gamma$ and $\delta$ are relatively elliptic and relatively $\tau$-elliptic, respectively, to conclude that ${}^2H_{\gamma}(\A_F)={}^1H_{\gamma}(\A_F)$ and ${}^2G_{\delta}(\A_F)={}^1G_{\delta}(\A_F)$.
Note that with the choice of Haar measure fixed in \S \ref{ssec-measures},
the measures occurring in the summands corresponding to $\gamma$
in the same stable relative class as a given $\gamma_0$ are compatible with respect to \eqref{src-twist} above.
Similarly, the measures occurring in the summands corresponding to
 $\delta$ in the same stable relative $\tau$-class as a given $\delta_0$ are compatible with respect to \eqref{src-twist}.
This follows from the well-known fact
that the Tamagawa numbers of two inner forms of the same quasi-split reductive group are equal \cite{KottTama} \cite{Cherno}.  With this in mind, we group stable classes in \eqref{1stgeomsum} and obtain
\begin{align} \label{RT-st-gp}
RT_e(\Phi)&=\sum_{\gamma_0} \tau(H_{\gamma_0}) SRO_{\gamma_0}(\Phi)\\
\nonumber TRT_e(f)&=\sum_{\delta_0} \tau(G_{\delta_0}) STRO_{\delta_0}(f)
\end{align}
where the first sum is over a set $\{\gamma_0\}$ of representatives for the stable relative classes in $H(F)$ that consist of relatively regular elliptic elements
and the second sum is over a set $\{\delta_0\}$ of representatives for the stable relative $\tau$-classes in $G(F)$ that consist of  relatively $\tau$-regular semisimple elements.  The measures inherent in the definition of $SRO_{\gamma_0}(\Phi)$ and $STRO_{\delta_0}(f)$ are specified as in \S \ref{ssec-measures}.

Applying Proposition \ref{prop-prestab}  we can rewrite \eqref{RT-st-gp} as
\begin{align} \label{RTe-formula}
RT_e(\Phi)&=\sum_{\gamma_0}
\frac{\tau(H^{\sigma} \times H^{\sigma})}{d(H_{\gamma_0},H^{\sigma} \times H^{\sigma})}
\sum_{\kappa_H}RO^{\kappa_H}_{\gamma_0}(\Phi)\\
\nonumber TRT_e(f)&=\sum_{\delta_0}
\frac{\tau(G^{\sigma} \times G^{\theta})}{d(G_{\delta_0},G^{\sigma} \times G^{\theta})}
\sum_{\kappa}TRO^{\kappa}_{\delta_0}(f).
\end{align}
Here the interior sum indexed by $\gamma_0$ is over $\kappa_H \in \mathfrak{K}(H_{\gamma_0},H^{\sigma} \times H^{\sigma};F)_1$
and the interior sum indexed by $\delta_0$ is over $\kappa \in 
\mathfrak{K}(G_{\delta}, G^{\sigma} \times G^{\theta};F)_1$ (compare Proposition \ref{prop-prestab}).

\subsection{The unitary case}

We now specialize our notation to our primary case of interest.  Thus assume that
$H^{\sigma}=U^{\sigma}$ \index{$U^{\sigma}$} is a unitary group with respect to a quadratic extension of
fields $M/F$ with $M$ a CM field and $F$ a totally real field.    In other words, we assume that there is a simple algebra $D$ over $F$ with center $M$ and an involution $\dagger$ of $D$ such that the fixed field of $\dagger$ acting on $M$ is $F$ and such that if $R$ is an $F$-algebra then
\begin{align} \label{unitary-gp}
U^{\sigma}(R):=\{g \in (D \otimes_F R)^{\times}: gg^{\dagger}=1\}.
\end{align}
We let $E/F$ be a quadratic extension of fields with $E$ totally real and assume moreover that we are in the biquadratic situation,
thus $U:=\mathrm{Res}_{E/F}U^{\sigma}$ and $\sigma$ is the
automorphism induced by the generator of $\Gal(E/F)$
which we will also denote by $\sigma$.  We let $\tau$
be the generator of $M/F$ and $G^{\sigma}:=\mathrm{Res}_{M/F}U^{\sigma}$, $G:=\mathrm{Res}_{M/F}U$.  The automorphism $\tau$ defines an automorphism $\tau:G \to G$ such that the subgroup of $G$ fixed by $\tau$ is $U$.  Finally, choose a (finite-dimensional) representation
$$
U^{\sigma}_{F_{\infty}} \lto \mathrm{Aut}_{\RR}(V)
$$
and let $G^{\sigma}_{F_{\infty}} \to \mathrm{Aut}_{\RR}(\mathrm{Res}_{M_{\infty}/F_{\infty}}V)$ be the representation obtained by restriction of scalars.
  We have the following proposition:

\begin{prop} \label{prop-form-compar} Suppose that $\Phi \in C_{c}^{\infty}(U(\A_F))$ and
$f \in C_c^{\infty}(G(\A_F))$ are factorable and that $\Phi_v$ matches $f_v$ for all finite places $v$ of $F$.
 Consider the following assumptions:
\begin{enumerate}
\item
One has
$$
f_{\infty}=f_1 \times f_2  \in C^{\infty}_c(G^{\sigma}(F_{\infty}) \times G^{\sigma}(F_{\infty}))= C_c^{\infty}(G(F_{\infty}))
$$
with $f_1^{-\tau}*f_2=f_{L,\mathrm{Res}_{M_{\infty}/F_{\infty}}V,\tau}$.

\item One has
$$
\Phi_{\infty}=\Phi_1 \times \Phi_2 \in C^{\infty}_c(U^{\sigma}(F_{\infty}) \times U^{\sigma}(F_{\infty}))=C_c^{\infty}(U(F_{\infty}))
$$
with $\Phi_1^{-1}*\Phi_2=c_{\infty}f_{EP,V}$ for some $c_{\infty} \in \RR_{>0}$.
\end{enumerate}
If the assumptions hold, then for an appropriate choice of $c_{\infty} \in \RR_{>0}$ one has
$$
RT_e(\Phi)=2TRT_e(f).
$$
\end{prop}
\noindent In the proposition, $f_{L,\mathrm{Res}_{M_{\infty}/F_{\infty}}V,\tau}$ is the Lefschetz function attached to the representation $\mathrm{Res}_{M_{\infty}/F_{\infty}}V$ of $G^{\sigma}_{F_{\infty}}$ and the involution $\tau$ of $G^{\sigma}_{F_v}$ (see \cite[Proposition 8.4]{BLS}). Moveover $f_{EP,V}$ is the Euler-Poincar\'e  function attached to $V$; this is simply the Lefschetz function in the case that the associated automorphism is trivial.
\begin{proof}
The functions $f_{L,\mathrm{Res}_{M_{\infty}/F_{\infty}}V,\tau}$ and $f_{EP,V}$ are stable in the sense of \cite[D\'efinition 3.8.2]{Lab} (see \cite[Th\'eor\`eme A.1.1]{ClozLab} and \cite[Th\'eor\`eme 7.1]{LabCM}).  Thus, in view of Proposition \ref{prop-E-split-match} and assumptions (2) and (1), any $\gamma_0$ (resp.~$\delta_0$) contributing a nonzero summand to $RT_e(\Phi)$ (resp.~$TRT_e(f)$) is relatively regular elliptic (resp.~relatively $\tau$-regular elliptic) at $F_{\infty}$.  Applying \cite[Proposition 1.9.6 and Lemme 1.9.7]{Lab} we conclude that for these places the set $\infty$ is $(H,H_{\gamma_0})$ (resp.~$(G,G_{\delta})$)-essential for any $\gamma_0$ (resp.~$\delta_0$) contributing a nonzero summand.
Using the fact that $f_{L,\mathrm{Res}_{M_{\infty}/F_{\infty}}V,\tau}$ and $f_{EP,V}$ are stable and Proposition \ref{prop-E-split-match} again, we conclude that
the $\kappa$-orbital integrals for $\kappa \neq 1$ in \eqref{RTe-formula} all vanish, and hence
\begin{align*}
RT_e(\Phi)&=\sum_{\gamma_0}
\frac{\tau(U^{\sigma} \times U^{\sigma})}{d(U_{\gamma_0},U^{\sigma} \times U^{\sigma})}
SRO_{\gamma_0}(\Phi)\\
TRT_e(f)&=\sum_{\delta_0}
\frac{\tau(G^{\sigma} \times G^{\theta})}{
d(G_{\delta_0},G^{\sigma} \times G^{\theta})}
STRO_{\delta_0}(f).
\end{align*}
By \cite[Corollaire A.1.2]{Lab} for all $v|\infty$ the functions $f_{v}$ and $c_v\Phi_v$ match for an appropriate constant $c_v \in \RR_{>0}$ (see also \cite[Th\'eor\`eme 7.1]{LabCM}).  We henceforth assume that $f_v$ and $\Phi_v$ match for all $v$.

Since every relatively regular semisimple $\gamma_0$ is a norm of a $\delta_0$ by Lemma \ref{matching-lem} and
 every relatively $\tau$-regular elliptic semisimple $\delta_0$ that has local norms at infinity has a global norm by Proposition \ref{matching-prop}, to complete the proof it suffices to show that
\begin{align} \label{eq123}
\frac{\tau(U^{\sigma} \times U^{\sigma})}{d(U_{\gamma_0},U^{\sigma} \times U^{\sigma})}
SRO_{\gamma_0}(\Phi)=2\frac{\tau(G^{\sigma}  \times G^{\theta})}{d(G_{\delta_0},G^{\sigma} \times G^{\theta})}
STRO_{\delta_0}(f)
\end{align}
if $\gamma_0$ is a norm of $\delta_0$.  By the definition of matching we have
$SRO_{\gamma_0}(\Phi)=STRO_{\delta_0}(f)$. Moreover
\begin{align}\label{dformula}
d(U_{\gamma_0},U^{\sigma} \times U^{\sigma})&=\#\mathrm{coker}\left[H^1_{\mathrm{ab}}(\A_F/F,U^{\sigma} \times U^{\sigma}) \to H^1_{\mathrm{ab}}(\A_F/F,U^{\sigma}\times U^{\sigma})\right]=1\\
d(G_{\delta_0},G^{\sigma} \times G^{\theta})&=\# \mathrm{coker}\left[  H^1_{\mathrm{ab}}(\A_F/F,U^{\sigma} \times G^{\theta}) \to H^1_{\mathrm{ab}}(\A_F/F,G^{\sigma} \times G^{\theta})\right] =1 \nonumber
\end{align}
by the fact that $\gamma_0$ and $\delta_0$ are relatively elliptic and relatively $\tau$-elliptic, respectively, and \cite[Corollaire 1.9.3]{Lab}.  By \cite[Corollaire 1.7.4]{Lab}, we have
\begin{align}\label{tau-formula}
\tau(U^{\sigma} \times U^{\sigma})&=\frac{\# H^1_{\mathrm{ab}}(\A_F/F, U^{\sigma} \times U^{\sigma})}{\#\ker^1_{\mathrm{ab}}(F,U^{\sigma} \times U^{\sigma})}=\frac{4}{\#\ker^1_{\mathrm{ab}}(F,U^{\sigma} \times U^{\sigma})}\\ \tau(G^{\sigma} \times G^{\theta})&=\frac{\#H^1_{\mathrm{ab}}(\A_F/F,G^{\sigma} \times G^{\theta})}{\#\ker^1_{\mathrm{ab}}(F,G^{\sigma} \times G^{\theta})}=\frac{2}{\#\ker^1_{\mathrm{ab}}(F,G^{\sigma} \times G^{\theta})} \nonumber
\end{align}
Here we are using the fact that $H^1_{\mathrm{ab}}(\A_F/F,\GL_n)=1$ and $H^1_{\mathrm{ab}}(\A_F/F,H)=2$ if $H$ is a (nonsplit) unitary group (see \cite[Lemma 1.2.1(i)]{HarLab} for the latter statement).

  Since $G^{\sigma}$ is an inner form of a general linear group, the Hasse principle is valid for it.  On the other hand, $G^{\theta}$ and $U^{\sigma}$ are unitary groups, so the Hasse principle is valid for them as well \cite[Lemma 1.2.1(i)]{HarLab}, so $\ker^1_{\mathrm{ab}}(F,U^{\sigma} \times U^{\sigma})=\ker^1_{\mathrm{ab}}(F,G^{\sigma} \times G^{\theta})=1$.  In view of \eqref{dformula} and \eqref{tau-formula}, this completes the proof of the proposition.
\end{proof}

\section{Relative trace formulae}
\label{sec-rtf}

\subsection{A simple relative trace formula} \label{ssec-sRTF}

For $f \in C_c^{\infty}(G(\A_F))$ and a cuspidal unitary automorphic representation $\Pi$ of $G(\A_F)$, Arthur has shown \cite[Lemma 4.5, Lemma 4.8]{A}
that there is a (unique) function $K_{\Pi(f^1)}(x,y) \in L_0^2(G(F) \backslash {}^1G(\A_F) \times G(F) \backslash {}^1G(\A_F))$ \index{$K_{\Pi(f^1)}(x,y)$} that is smooth in $x$ and $y$ separately with $L^2$-expansion
\begin{align} \label{expansion}
K_{\Pi(f^1)}(x,y)=\sum_{\phi \in \mathcal{B}(\Pi)} \Pi(f^1){\phi}(x)\overline{{\phi}(y)}.
\end{align}
Here $\mathcal{B}(\Pi)$ \index{$\mathcal{B}(\Pi)$} is an orthonormal basis of the $\Pi$-isotypic subspace $V_{\Pi} \leq L^2_0(G(F) \backslash {}^1G(\A_F))$ with respect to the pairing
\begin{align} \label{ortho-pair}
V_{\Pi} \times V_{\Pi} &\lto \CC\\
\nonumber (\phi_1,\phi_2) &\longmapsto \int_{G(F) \backslash {}^1G(\A_F)} \phi_1(y) \overline{{\phi}_2(y)}dy,
\end{align}
with $dy$ induced by the Tamagawa measure.
We emphasize that this expansion \eqref{expansion}, in general, is only convergent in the $L^2$ sense.
Following \cite{Hahn}, we define
\begin{align}
TRT(\Pi(f^1))&=\iint_{G^{\sigma}(F) \backslash {}^2G^{\sigma}(\A_F) \times G^{\theta}(F) \backslash {}^2G^{\theta}(\A_F)} K_{\Pi(f^1)}(x,y)dxdy.
\end{align} \index{$TRT(\Pi(f^1))$}

\noindent
The integral is absolutely convergent by \cite[\S 2, Proposition 1]{AGR}.
The following simple relative trace formula is proved in \cite{Hahn} via a modification of the
argument used to prove the usual simple trace formula:

\begin{thm} \label{thm-rel-trace-formula} Let $f=f_{v_1} \otimes f_{v_2}\otimes f^{v_1v_2} \in C_c^{\infty}(G(\A_F))$ be a factorable function such that
\begin{itemize}
\item $f_{v_1}$ is $F$-supercuspidal.
\item $f_{v_2}$ is supported on relatively $\tau$-regular elliptic elements of $G(F_{v_2})$.
\end{itemize}
Then
\begin{align}\label{rtf}
TRT_e(f):=\sum_{\delta}\tau(G_{\delta}) TRO_{\delta}(f)
&=
\sum_{\Pi} TRT(\Pi(f^1))
\end{align}
where the sum on the left is over a set of representatives for the relatively $\tau$-regular elliptic classes in $G(F)$ and the sum on the right
is over a set of representatives for the equivalence classes of cuspidal automorphic representations $\Pi$  of ${}^1G(\A_F)$.
\end{thm}
Here we say that $f_v$ is
\textbf{$F$-supercuspidal} if $f_v$ has zero integral along the unipotent radical of any proper parabolic of $G_{F_v}$ that is
\emph{defined over} $F$; i.e.~is the base change to $F_v$ of a parabolic subgroup of $G$.
As usual, we allow $\tau$ to be trivial in the theorem above.
By convention, an automorphic representation of ${}^1G(\A_F)$ is the restriction to ${}^1G(\A_F)$ of an automorphic representation of $G(\A_F)$, and we consider two such to be equivalent if they are equivalent as representations of ${}^1G(\A_F)$.

Let $K_{\infty} \leq G(F_{\infty})$ be a maximal compact subgroup.  We note that if $\Pi$ is cuspidal and $f$ is $K_{\infty}$-finite or $\Pi(f^1)$ has finite rank then we have
\begin{align} \label{if-k-finite}
TRT(\Pi(f^1))=\sum_{\phi \in \mathcal{B}(\Pi)}\mathcal{P}_{G^{\sigma}}(\Pi(f^1)\phi)\overline{\mathcal{P}_{G^{\theta}}(\phi)}
\end{align}
where the sum is over an orthonormal basis $\mathcal{B}(\Pi)$ of the $\Pi$-isotypic subspace of $L^2_0(G(F) \backslash {}^1G(\A_F))$ consisting of smooth vectors (see \eqref{period} for the definition of $\mathcal{P}_{G^{\sigma}}$).
  For any $f \in C_c^{\infty}(G(\A_F))$, if $TRT(\Pi(f^1))$ is nonzero then $\Pi$ is both $G^{\sigma}$ and $G^{\theta}$-distinguished.

\subsection{Comparison}

Upon combining Proposition \ref{prop-form-compar} and Theorem \ref{thm-rel-trace-formula}, we obtain the following proposition:

\begin{prop} \label{prop-spect-compar}  Suppose that $\Phi=\otimes_v'\Phi_v \in C_c^{\infty}(U(\A_F))$
and $f=\otimes_v' f_v\in C_c^{\infty}(G(\A_F))$ are factorable, and satisfy the following conditions:
\begin{itemize}
\item $\Phi_v$ matches $f_v$ for all $v$.
\item $\Phi_{\infty}$ and $f_{\infty}$ satisfy conditions (1-2) of Proposition \ref{prop-form-compar}.
\item There is a finite place $v_1$ such that $\Phi_{v_1}$ is $F$-supercuspidal.
\item There is a place $v_2$ of $F$ such that $f_{v_2}$ is $F$-supercuspidal.
\item There is a place $v_3$ of $F$ such that  $\Phi_{v_3}$ is supported on relatively regular elliptic semisimple elements.
\item There is a place $v_4$ of $F$ such that $f_{v_4}$ is supported on relatively $\tau$-regular elliptic semisimple elements.
\end{itemize}
Under the above assumptions, we have
$$
\sum_{\pi} RT(\pi(\Phi^1))=2\sum_{\Pi} TRT(\Pi(f^1)),
$$
where the sums are over equivalence classes of cuspidal automorphic representations $\pi$ of $U(\A_F)$ and $\Pi$ of ${}^1G(\A_F)$, respectively.  \qed
\end{prop}
\noindent
Note that we do not require that the places $v_i$ be distinct.  Here $RT(\pi(\Phi^1))$ is defined to be $TRT(\Pi(f^1))$ in the ``$\tau=1$'' case.

\section{Application}

\label{ssec-apps}

 For this entire section we will place ourselves in the following special case of the construction exposed in the previous sections.
 Let $E/F$ be a quadratic extension of totally real fields and let $M/F$ be a CM extension.
Let $U^{\sigma}$ be a unitary group over $F$ as in \eqref{unitary-gp} and $U:=\mathrm{Res}_{E/F}U^{\sigma}$.  We let $\sigma$ be the automorphism of $U$ induced by the generator of $\Gal(E/F)$, which we will also denote by $\sigma$.  The groups $G^{\sigma}=\mathrm{Res}_{M/F}U^{\sigma}$ and $G=\mathrm{Res}_{M/F}U$ are isomorphic to inner forms of $\Res_{M/F}\GL_n$ and $\Res_{ME/F}\GL_n$, respectively, for some $n$.

When we refer to the base change map below, we will mean the (partially defined)
functorial lifting with respect to the map of $L$-groups
$$
b:{}^LU \lto {}^LG
$$
induced by the natural inclusion $U \to G$ (see \cite{HarLab}).

\begin{defn} \label{defn-wbc} Let $\pi$ be a cuspidal automorphic representation of $U^{\sigma}(\A_E)=U(\A_F)$.  We say that a cuspidal automorphic representation $\Pi$ of $G^{\sigma}(\A_E)=G(\A_F)$ is a \textbf{weak base change}
of $\pi$ if $\Pi_w$ is the base change of $\pi_w$ for all places $w$ of $E$ satisfying the following:
\begin{itemize}
\item $w$ is infinite,
\item $ME/E$ is split at $w$, or
\item $\pi_w$ and $U^{\sigma}_{E_w}$ are unramified.
\end{itemize}
\end{defn}
By strong multiplicity one for $G$ \cite{Badu}, if a weak base change of $\pi$ exists then it is unique.  Suppose that $\pi$ and $\pi'$ are cuspidal and both admit weak base changes to $G(\A_F)$.  If $\pi$ and $\pi'$ are moreover \textbf{nearly equivalent}, i.e. $\pi_v \cong \pi_{v}'$ for almost all places $v$, then their weak base changes are obviously equal.

\subsection{Statement of theorem} \label{ssec-st-thms}

To state the main theorem of this section, it is convenient to introduce a definition.  Let $K_{U\infty} \leq
U(F_{\infty})$ be a maximal compact subgroup and let $V$ be a representation of $U(F_{\infty})$.
Let $\mathfrak{u}:=\mathrm{Lie}(U_{F_{\infty}}) \otimes_{\RR}\CC$ be the complexification of the real Lie group
 $U(F_{\infty})$.

\begin{defn} \label{cohom-defn} An automorphic representation $\pi$ of $U(\A_F)$ has \textbf{nonzero cohomology with
coefficients in $V$} if $H^*(\mathfrak{u},K_{U\infty};\pi_{\infty} \otimes V)$ is not identically zero.
\end{defn}

We have the following theorem:

\begin{thm} \label{main-thm} Let $\pi$ be a cuspidal automorphic representation of $U^{\sigma}(\A_E)=U(\A_F)$ that admits a weak base change $\Pi$ to $G(\A_F)$.
Suppose that $\pi$ satisfies the following conditions:
\begin{enumerate}
\item There is a finite-dimensional representation $V$ of $U_{F_{\infty}}$ such that $\pi$ has
nonzero cohomology with coefficients in $V$.
\item There is a finite place $v_1$ of $F$ totally split in $ME/F$ such that $\pi_{v_1}$ is supercuspidal.
\item There is a finite place $v_2 \neq v_1$ of $F$ totally split in $ME/F$ such that $\pi_{v_2}$ is in the discrete series.
\item For all places $v$ of $F$ such that $ME/F$ is ramified and $M/F$, $E/F$ are both nonsplit at $v$ the weak base change $\Pi$ of $\pi$ to $G(F)$ has the property that $\Pi_v$ is relatively $\tau$-regular.
\end{enumerate}

If the automorphic representation $\Pi$ is both $G^{\sigma}$
and $G^{\theta}$-distinguished then there is a cuspidal automorphic representation $\pi'$ of $U^{\sigma}(\A_E)=U(\A_F)$ that is $U^{\sigma}$-distinguished and
nearly equivalent to $\pi$.  Moreover, we can take $\pi'$ to have nonzero cohomology with coefficients in $V$.
\end{thm}

  Theorem \ref{main-thm} will be proven later in this section.  Combining it with the work of Jacquet, Lapid and their collaborators, and Flicker and his collaborators, we obtain the following corollary:

\begin{cor} \label{main-cor} Assume that $U^{\sigma}$ is quasi-split or more generally that $G^{\theta}$ and $G^{\sigma}$ are quasi-split. Let $\pi$ be a cuspidal automorphic
representation of $U(\A_F)$ satisfying conditions (1)-(4) of Theorem
\ref{main-thm}.
The representation $\pi$ admits a weak base change $\Pi$ to $\GL_n(\A_{ME})$. If the partial Asai $L$-function $L^S(s,\Pi;r)$ has a pole at $s=1$ then some cuspidal automorphic representation
$\pi'$ of $U(\A_F)$ nearly equivalent to $\pi$ is $U^{\sigma}$-distinguished.
 Moreover, we can take $\pi'$ to have nonzero cohomology
with coefficients in $V$.
\end{cor}

We require the following lemma for the proof of Corollary \ref{main-cor} and also below in the proof of Theorem \ref{main-thm}:

\begin{lem} \label{lem-2nd-stab} Suppose that $\Pi$ is a cuspidal automorphic representation of $G(\A_F)$.
 If $\Pi$ is $G^{\theta}$-distinguished, then $\Pi \cong \Pi^{\theta\vee}$.  If $\Pi$ is $G^{\sigma}$-distinguished,
 then $\Pi^{\sigma} \cong \Pi^{\vee}$.
\end{lem}

\begin{proof} Both of these are well-known.  For the first, see \cite[\S 3]{JacquetKlII}.  For the second, assume that $\Pi$ is $G^{\sigma}$-distinguished.  For places $v$ of $F$ split in $E/F$, it is trivial to check that $\Pi^{\sigma}_v \cong \Pi^{\vee}_v$ using an analogue of the proof of Proposition \ref{prop-first-stab} given below.  If $v$ is inert (and unramified) in $E/F$, the fact that $\Pi^{\sigma}_v \cong \Pi^{\vee}_v$ is \cite[Corollary 2]{Prasad} (see also \cite{FlickerDist}).  Thus $\Pi^{\sigma} \cong \Pi^{\vee}$ by strong multiplicity one.
\end{proof}

We also require the following proposition in the proof of Theorem \ref{main-thm}:

\begin{prop} \label{prop-first-stab} Let $\Xi$ be the set of places of $F$ that split in $E/F$,
let $\pi$ be a cuspidal automorphic representation of $U(\A_F)$.  If $\pi$ is $U^{\sigma}$-distinguished, then
$\pi_{\Xi}^{\vee} \cong (\pi_{\Xi})^{\sigma}$.
\end{prop}

\begin{proof}
We may assume, without loss of generality, that the $\phi_0$ in the space of $\pi$ with nonzero period over $U^{\sigma}(\A_F)$ is
factorable.  Write $\phi_0=\otimes_v \phi_{0,v}$, where the tensor
product is over all places of $F$.  For  $v \in \Xi$, choose an isomorphism $U_{F_v} \cong U^{\sigma}_{F_v} \times U^{\sigma}_{F_v}$ intertwining $\sigma$ with $(x,y) \mapsto (y,x)$.  Using this isomorphism, we can and do decompose
\begin{align} \label{loc-isom5}
\pi_v \cong \pi_{1,v}' \otimes \pi_{2,v}'
\end{align}
for some admissible representations $\pi_{i,v}'$ of $U^{\sigma}(F_v)$.
Thus
$$
\pi_v^{\sigma} \cong \pi_{2,v}' \otimes \pi_{1,v}'.
$$

For
a fixed $v \in \Xi$, let $V_{\pi_{1,v}'}$, $V_{\pi_{2,v}'}$, and
$V_{\pi}$ be the spaces of the representations $\pi_{1,v}',\pi_{2,v}'$, and $\pi$, respectively. Consider the bilinear pairing
\begin{align*}
V_{\pi_{1,v}'} \otimes V_{\pi_{2,v}'} &\hookrightarrow V_{\pi} \lto
\CC.
\end{align*}
Here the first map is $(\phi_1,\phi_2) \mapsto (\otimes_{v' \neq v}
\phi_{0,v'})\otimes (\phi_1 \otimes \phi_2)$ and the second is $\phi \mapsto \mathcal{P}_{U^{\sigma}}(\phi)$.
This bilinear pairing is $U^{\sigma}(F_v)$-equivariant, and is nonzero by
hypothesis.    By the irreducibility of $\pi_v$, we conclude that
$$
\pi_{2,v} \cong \pi_{1,v}^{\vee}
$$
and thus
$$
\pi_v^{\sigma} \cong \pi_v^{\vee}.
$$
\end{proof}

We  give the proof of Corollary \ref{main-cor} now:

\subsection{Proof of Corollary \ref{main-cor}} \label{ssec-main-cor}
We assume the notation and hypotheses of Corollary \ref{main-cor}. Let $\Pi$ be the weak base change of $\pi$; it exists by the proof of \cite[Theorem 3.1.4]{HarLab}.  Since $\Pi$ is a weak base change we have  $\Pi_v\cong \Pi^{\tau}_v$ for almost all places $v$.
By strong multiplicity one we conclude that
$\Pi
\cong \Pi^{\tau}$.  View $\Pi$ as an automorphic representation of $\Res_{ME/M}\GL_n(\A_M)$.
Let $r:{}^{L}\mathrm{Res}_{ME/M} \GL_n \to \GL_{n^2}(\CC)$ be the Asai
representation (see \cite[\S 6]{Ramakrishnan} or \cite{FlickerTwTen} for the
definition of this representation).  By assumption the partial Asai $L$-function
$$
L^S(s,\Pi;r)
$$
has a pole at $s=1$.  Here $S$ is a finite set of places of $M$ containing the infinite places and all finite places
where $ME/M$ or $\Pi_v$ is ramified.  We note that we are using the fact that $L^S(s,\Pi;r)$ admits
a meromorphic continuation to a closed right half-plane containing $s=1$, a fact established by
Flicker and Flicker-Zinoviev \cite{FlickerTwTen} \cite[Theorem]{FlickerZin} using the
Langlands-Shahidi  method \cite{Shahidi} and an adaptation of the Rankin-Selberg method of Jacquet,
Piatetski-Shapiro, and Shalika \cite{JPSS}. As a biproduct of this method,
Flicker and Flicker-Zinoviev also establish in \cite[Theorem]{FlickerZin} that if $L^S(s,\Pi;r)$ has
a pole at $s=1$, then $\Pi$ is distinguished by $\GL_{n/M}$.  This implies that $\Pi^{\sigma} \cong \Pi^{\vee}$ by Lemma \ref{lem-2nd-stab}.

Writing $L$ for the subfield of $ME$ fixed by $\theta:=\sigma \circ \tau$,
note that $ME/L$ splits at all infinite places.  Since $\Pi^{\sigma} \cong \Pi^{\vee}$ and $\Pi^{\tau} \cong \Pi$ we have $\Pi^{\theta \vee} \cong \Pi$. Thus, applying results of Jacquet (\cite[Theorem 3 and 4]{JacquetKlII} and \cite{Jacquetqs})
we conclude that $\Pi$ is distinguished by ``the'' quasi-split unitary group in $n$-variables with respect to $ME/L$.

The corollary now follows from Theorem \ref{main-thm}. \qed

\subsection{Preparations for the proof of Theorem \ref{main-thm}}
\label{ssec-preps}
We isolate three steps in
the proof of Theorem \ref{main-thm} in the following lemmas.  If the reader so
desires, (s)he can skip this subsection and refer back to it as needed during the
following subsection.

For the purpose of stating a lemma we develop some notation.  Assume that we are in the biquadratic case of \S \ref{ssec-basic-notat} and let $v$ be a place of $F$ split in $E/F$.
Suppose $\Pi_v$  is an irreducible unitary (admissible) representation of
$G(F_v)$.  We agree to realize it on a Hilbert space $V_{\Pi_v}$.  Write
$$
\mathcal{H}_{\Pi_{v}}:=\mathrm{Im}(C_c^{\infty}(G(F_v)) \to \mathrm{End}_{\mathrm{HS}}(V_{\Pi_{v}})),
$$
where the map sends $f$ to $\Pi_{v}(f)$.  Here $\mathrm{End}_{\mathrm{HS}}(V_{\Pi_{v}})$ is the space of Hilbert-Schmidt operators on $V_{\Pi_v}$. For any $x,y \in G(F_v)$ we have linear left multiplication by $x$ and right multiplication by $y$ maps
$$
L_x,R_y:\mathrm{End}_{\mathrm{HS}}(V_{\Pi_{v}}) \lto \mathrm{End}_{\mathrm{HS}}(V_{\Pi_{v}})
$$
given by $L_x(T):= \Pi_v(x^{-1}) \circ T$ and $R_y(T):=T \circ \Pi_v(y^{-1})$.

\begin{lem} \label{lem-lforms} The space of linear forms
$$
\ell:\mathcal{H}_{\Pi_{v}} \lto \CC
$$
satisfying $L_x(\ell)=R_y(\ell)=\ell$ for $(x,y) \in G^{\sigma} \times G^{\theta}(F_v)$ is at most one-dimensional.
\end{lem}
A linear form as in the lemma is known as a \textbf{$(G^{\sigma}(F_v),G^{\theta}(F_v))$-invariant linear form}.
\begin{proof}
The natural isomorphism
$$
\End_{\mathrm{HS}}(V_{\Pi_{v}})^{\vee} \cong V_{\Pi_{v}} \boxtimes V_{\Pi_{v}}^{\vee}
$$
is $G(F_v) \times G(F_v)$-equivariant.  On the right we are taking the completed (Hilbert space) tensor product.
A linear form $\ell$ as in the lemma is sent to
\begin{align} \label{fixedsspace}
 (V_{\Pi_{v}})^{G^{\sigma}(F_v) } \boxtimes (V_{\Pi_{v}}^{\vee})^{G^{\theta}(F_v)}
\end{align}
under this isomorphism.  Since $\Pi_{v} \cong \Pi_{1v} \otimes \Pi_{2v}$ for some
admissible representations $\Pi_{iv}$ of $G^{\sigma}(F_v)$, it follows that \eqref{fixedsspace} is at most one dimensional.
 This proves that the space of linear forms satisfying the hypotheses of the lemma is at most one dimensional.
\end{proof}

We now record a lemma on supercusp forms.  Let $v$ be a finite
place of $F$ that splits completely in $ME/F$.  Thus we have isomorphisms
\begin{align} \label{uisom3}
U_{F_v} &\cong (U^{\sigma})^2_{F_v}\\
\nonumber G_{F_v} &\cong U_{F_v}^2 \cong (U^{\sigma})^4_{F_v}
\end{align}
intertwining $\sigma$ with $(x,y) \mapsto (y,x)$ (resp. $(x,y,z,w) \mapsto
(z,w,x,y)$) and $\tau$ with $(x,y,z,w) \mapsto (y,x,w,z)$.
Let $\pi_v$ be a unitary
admissible representation of $U(F_v)$ satisfying $\pi_v^{\vee} \cong
\pi_v^{\sigma}$; thus we can and do decompose
$$
\pi_v \cong \pi_v' \otimes \pi'^{\vee}_v
$$
for some admissible representation $\pi_v'$ of $U^{\sigma}(F_v)$ using
\eqref{uisom3}. Let $\Pi_v$ be the base change of $\pi_v$ to
$G(F_v)$. We can and do factor
$$
\Pi_v=\pi_v \otimes \pi_v^{\vee} \cong \pi_{v}' \otimes \pi'^{\vee}_v
\otimes \pi'^{\vee}_v \otimes \pi'_v
$$
with respect to the second line of \eqref{uisom3}.  We also write $\Pi_v'=\pi_v'
\otimes \pi'^{\vee}_v$.

Let $H$ be a reductive $F$-group.  As above, we say that a function $f \in C^{\infty}_c(H(F_v))$ is
$F$-supercuspidal if the integral of $f$ along the $F_v$-points of the unipotent radical of any $F$-rational proper parabolic subgroup of $H$ vanishes.

\begin{lem} \label{lem-supercusp} If $\pi_v'$ is supercuspidal, then
there are matching functions
$$
f=f_1 \times f_2 \in
C_c^{\infty}(G^{\sigma} \times G^{\sigma}(F_v)) \cong C_c^{\infty}(G(F_v))
$$
and $\Phi \in
C_c^{\infty}(U(F_v))$, both $F$-supercuspidal, such that
$$
\mathrm{tr}\left( \Pi_v'(f_1^{-\tau}*f_2) \circ \tau \right) \neq 0.
$$
\end{lem}

\begin{proof}
Let $\Phi_{1}$ be a truncated diagonal matrix coefficient of $\pi'_v$ \cite[\S 1.9]{HarLab}.  We simply
let
\begin{align*}
f_1&=\Phi_1^{-1} \times \mathrm{ch}_{K_{Uv}'} \\
f_2&=\mathrm{meas}(K_{Uv}')^{-2}\mathrm{ch}_{K_{Uv}'} \times \mathrm{ch}_{K_{Uv}'}\\
\Phi&=\Phi_1^{-1} \times \mathrm{ch}_{K_{Uv}'}
\end{align*}
for a sufficiently small compact open subgroup $K_{Uv}' \leq U^{\sigma}(F_v)$ such that $\Phi_1 \in C_c^{\infty}(U^{\sigma}(F_v)//K_{Uv}')$.  The functions $f_1 \times f_2$ and $\Phi$ match by Proposition \ref{prop-E-split-match}.
\end{proof}

We require an analogous lemma in the discrete series case:

\begin{lem} \label{lem-sqint}
If $\pi_v'$ is in the discrete series,
then
there are matching functions
$$
f=f_1 \times f_2 \in
C_c^{\infty}(G^{\sigma} \times G^{\sigma}(F_v))\cong C_c^{\infty}(G(F_v))
$$
and
$\Phi \in
C_c^{\infty}(U(F_v))$, supported on the relatively $\tau$-regular elliptic
subset of $G(F_v)$  and the relatively regular elliptic subset of
$U(F_v)$, respectively, such that
$$
\mathrm{tr}\left( \Pi_v'(f_1^{-\tau}*f_2 \circ \tau)\right) \neq 0.
$$

\end{lem}

\begin{proof}
Let $G^{re}(F_v) \subset G(F_v)$ (resp.
$U^{re}(F_v)\subset U(F_v)$) denote the subset of relatively
$\tau$-regular elliptic elements (resp.~regular elliptic elements).  By
definition, $G^{re}(F_v)$ is the preimage of the set
$$
\{(x,y,y^{-1},x^{-1})\in (U^{\sigma})^4(F_v) \cong G(F_v):xy \textrm{ is
elliptic regular}\}\subset S(F_v)
$$
under the map $g \mapsto g g^{-\theta}$.   Similarly, $U^{re}(F_v)$ is the preimage
of the set
$$
\{(x,x^{-1}) \in (U^{\sigma})^2(F_v) \cong U(F_v) : x \textrm{ is elliptic regular}\}
\subset Q(F_v)
$$
under the map $g \mapsto g g^{-\sigma}$.  It follows that both
$G^{re}(F_v)\subset G(F_v)$ and
$U^{re}(F_v)\subset U(F_v)$ are open and intersect arbitrarily small
neighborhoods of the identity.

Let $\Theta_{\pi_v'}$ be the character of $\pi_v'$.  Using a well-known result of Harish-Chandra, we
view $\Theta_{\pi_v'}$ as a locally constant function on the elliptic regular set of $U^{\sigma}(F_v)$
and as a locally integrable function on all of $U^{\sigma}(F_v)$.  Since
$\pi_v'$ is in the discrete series, there is a regular elliptic $\gamma_0 \in U^{\sigma}(F_v)$ such that
$\Theta_{\pi_v'}(\gamma_0) \neq 0$ \cite[Proposition 5.5]{RogGLn}.
Note $(\gamma_0,1,1,1) \in U^{\sigma}(F_v)^4 \cong G(F_v)$ is relatively $\tau$-elliptic regular.  By the openness
statement above, we can  and do choose a sufficiently small compact open subgroup $K_{Uv}' \leq U^{\sigma}(F_v)$ and a function $\Phi_1 \in
C_c^{\infty}(U^{\sigma}(F_v))$ supported in the elliptic regular set of $U^{\sigma}(F_v)$ such that
$$
f_1 \times f_2=(\Phi_{1}^{-1} \times \mathrm{ch}_{K_{Uv}'}) \times \mathrm{ch}_{K_{Uv}'} \times \mathrm{ch}_{K_{Uv}'}
$$
is supported in $G^{re}(F_v)$ and $\mathrm{tr}(\pi_{v}'(\Phi_1)) \neq 0$.  Here we are using the fact that $\Theta_{\pi_v'}$ is a locally constant
function on the elliptic regular set of $U^{\sigma}(F_v)$.  By Proposition \ref{prop-E-split-match} the functions $f_1 \times f_2$ and $\Phi:=\Phi_1^{-1} \times \mathrm{ch}_{K_{Uv}'}$ match, and hence $f_1 \times f_2$ and
 $\Phi:=\Phi_1^{-1} \times \mathrm{ch}_{K_{Uv}'}$ satisfy the requirements of the lemma.
\end{proof}

\subsection{Proof of Theorem \ref{main-thm}: Separating Hecke characters}
\label{ssec-part1}
In this subsection we use our trace formula identity Proposition \ref{prop-spect-compar} together with an adaptation of the argument of
\cite[\S 3]{JacquetLai} (see also \cite[\S 13]{Hakim})
to reduce the proof of Theorem \ref{main-thm} to a nonvanishing statement that will be proved in the following subsection.
We assume the hypotheses of Theorem \ref{main-thm}.

 By assumption, $\pi$ has nonzero cohomology with coefficients in $V$.  In view of
Proposition \ref{prop-first-stab} this implies that there is a representation $V_1$ of $U^{\sigma}_{F_{\infty}}$ such that $V  \cong V_1 \boxtimes V_1^{\vee}$.
Let $K_{\infty} \leq G(F_{\infty})$ (resp.~$K_{U\infty} \leq U(F_{\infty})$) be a maximal compact subgroup.  We claim that we can choose $f_{\infty} \in G(F_{\infty})$ and $\Phi_{\infty} \in C_c^{\infty}(U(F_{\infty}))$ that match, are right $K_{\infty}$ and right $K_{U\infty}$-finite, respectively, and satisfy the hypotheses (1-2) of Proposition \ref{prop-form-compar} for the representations $V_1$.  To see this we recall that Lefschetz functions are finite under the left and right action of the relevant maximal compact subgroup \cite[Proposition 8.4]{BLS}.  
Thus it follows from the Dixmier-Malliavin lemma \cite{DM} that we can choose right-$K^{\sigma}_{\infty}:=K_{\infty} \cap G^{\sigma}(F_{\infty})$-finite $f_1,f_2\in C_c^{\infty}(G^{\sigma}(F_{\infty}))$ such that $f_1^{-\tau}*f_2=f_{L,\mathrm{Res}_{M_{\infty}/F_{\infty}}V_1,\tau}$.  A similar argument with $G(F_{\infty})$ replaced by $U(F_{\infty})$ together with Proposition \ref{prop-form-compar} implies our claim.
We henceforth assume $\Phi_{\infty}$ and $f_{\infty}$ satisfying the conclusion of our claim.

Now let
\begin{align*}
\Phi&=\Phi_{\infty} \otimes \Phi^{\infty} \in C_c^{\infty}(U(\A_F))\\
f&=f_{\infty} \otimes f^{\infty} \in C_c^{\infty}(G(\A_F))
\end{align*}
be factorable functions satisfying the hypotheses of Proposition \ref{prop-spect-compar}.  Such functions exist
by Lemma \ref{lem-M-split-match}, Proposition \ref{prop-E-split-match}, Theorem \ref{thm-weak-match}, Lemma \ref{lem-supercusp} and Lemma \ref{lem-sqint}.
Thus by Proposition \ref{prop-spect-compar} we have
\begin{align} \label{sp-id}
\sum_{\pi'} RT(\pi'(\Phi^1))
=2\sum_{\Pi'} TRT(\Pi'(f^1)).
\end{align}
 Here the sum on the left is over equivalence classes of cuspidal automorphic representations of ${}^1U(\A_F)=U(\A_F)$ and the sum on the right is over equivalence classes of cuspidal automorphic representations of ${}^1G(\A_F)$.  Let $S$ be a finite set of places containing all infinite places and all places where $\pi$ is ramified.  Choose compact open subgroups $K_U^{S} \leq U(\A_F^{S})$, $K^S \leq G(\A_F^S)$ such that $K^S_{Uv}$ and $K_v^S$ are hyperspecial for all $v \not \in S$ and $\pi^S$ (resp. $\Pi^S$) contains the unit representation of $K^S_U$ (resp. $K^S$).  We assume that $f^S \in C_c^{\infty}(G(\A_{F}^S)//K^S)$ and $\Phi^S \in C_c^{\infty}(U(\A_F^S)//K^S)$; then \eqref{sp-id} implies the following identity:
\begin{align} \label{sp-id2}
\sum_{\pi'} \mathrm{tr}(\pi'(\Phi^S))RT(\pi'(\Phi^1_S\mathrm{ch}_{K^S_U}))=
2\sum_{\Pi'} \mathrm{tr}(\Pi'(f^S))TRT(\Pi'(f^1_S\mathrm{ch}_{K^S}))
\end{align}
(compare \cite[\S 3(2)]{JacquetLai}).  The reason \eqref{sp-id} simplifies to \eqref{sp-id2} is simply that if $\pi'_v$ (resp. $\Pi'_v$) is unramified then it contains a unique fixed vector under the hyperspecial subgroup $K_{Uv}$ (resp. $K_v$).

We have the following lemma:
\begin{lem} \label{lem-cohomol} Any $\Pi'$ contributing a nonzero summand to the right of \eqref{sp-id2}
 has nonzero cohomology with coefficients in $\mathrm{Res}_{M_{\infty}/F_{\infty}}V$.  Any $\pi'$ contributing a nonzero summand to the left of
 \eqref{sp-id2} has nonzero cohomology with coefficients in $V$.
\end{lem}
\begin{proof} Choose an isomorphism $G(F_{\infty}) \cong G^{\sigma}(F_{\infty}) \times G^{\sigma}(F_{\infty})$ equivariant with respect to $\tau$ and intertwining $\sigma$ with $(x,y) \mapsto (y,x)$.
 Write $\Pi_{\infty}'=\Pi_1 \otimes \Pi_2$ for some representations $\Pi_i$ of $G^{\sigma}(F_{\infty})$ using this isomorphism.
For a fixed $f_{S}^{\infty} \in C_c^{\infty}(G(F_S^{\infty}))$, consider the linear forms
\begin{align*}
\ell_i:\mathcal{H}_{\Pi_{\infty}} &\lto \CC\\
(\Pi_1(f_1),\Pi_2(f_2)) &\longmapsto \mathrm{tr}\left( \Pi_{1}(f_1^{-\tau}*f_2 \circ \tau )\right)\\
(\Pi_1(f_1),\Pi_2(f_2)) & \longmapsto \sum_{\phi \in \mathcal{B}(\Pi')^{K_S}} \mathcal{P}_{G^{\sigma}}(\Pi_1 \times \Pi_2((f_1 \times f_2)^1)\Pi'(f_S^{\infty})\phi) \overline{\mathcal{P}_{G^{\theta}}(\phi)}.
\end{align*}
They are both $(G^{\sigma}(F_v),G^{\theta}(F_{\infty}))$-invariant and hence equal up to a constant multiple (possibly zero) by Lemma \ref{lem-lforms}.  Thus the first assertion of the lemma follows from the defining property of Lefschetz functions \cite[Proposition 8.4]{BLS} and Lemma \ref{lem-2nd-stab}.  The proof of the second assertion is similar; one uses Proposition \ref{prop-first-stab} instead of Lemma \ref{lem-2nd-stab}.
\end{proof}

By the lemma, the collection of $\Pi'$ on the right of \eqref{sp-id2} is finite in a sense independent of $f^S \in C_c^{\infty}(G(\A_F^S)//K^S)$ for fixed $f_S$ by our assumption on
$f_{\infty}$. Indeed, the sum can be thought of as being over automorphic
representations contributing to the cohomology of a locally symmetric space depending only on $f_S$ with coefficients in a fixed local system depending only on
$f_{\infty}$ by Lemma \ref{lem-cohomol}.
Using the supply of matching $f^S$ and $\Phi^S$ provided by Corollary \ref{cor-H-split-matching} and Corollary \ref{fl-e-spl}
we separate strings of Hecke eigenvalues ``outside $S$'' in \eqref{sp-id2} to arrive at the following refined identity:
\begin{align} \label{sp-id5}
\sum_{\pi'\atop{\pi'^S \cong \pi^S}} \mathrm{tr}(\pi'(\Phi^S))RT(\pi'(\Phi^1_S\mathrm{ch}_{K^S_U}))= 2 \mathrm{tr}(\Pi(f^S))TRT(\Pi(f_S^1\mathrm{ch}_{K^S})).
\end{align}
Here we are using Lemma \ref{lem-2nd-stab} and strong multiplicity one for $G(\A_F)$ to isolate the contribution of $\Pi$ on the right hand side.  We are also using the fact that when $ME/F$, $U$, and $G$ are unramified at a finite place $v$ then the base change map from irreducible admissible unramified representations of $U(F_v)$ to irreducible admissible unramified representations of $G(F_v)$ is injective \cite[Corollary 4.2]{Minguez}.

Let $S(ME)$ be the set of finite places of $F$ such that $ME/F$ is ramified and both $E/F$ and $M/F$ are nonsplit.  We let $S_0=\infty \cup \{v_1,v_2\}$ and $S_0(ME)=S(ME) \cup S_0$. Enlarging $S$ if necessary, we assume that $S_0(ME)  \subseteq S$.
To complete the proof of the theorem, we show that the left side of \eqref{sp-id5} is nonzero for some $f \in C_c^{\infty}(G(\A_{F}))$ satisfying the various conditions we have placed earlier.
In view of Theorem  \ref{thm-weak-match}, it suffices to show that upon enlarging $S$ if necessary we can choose $f_{S}^{\infty}$ so that
$$
TRT(\Pi(f_S^1\mathrm{ch}_{K^S})) \neq 0
$$
where $f_{S(ME)}$ is supported on relatively $\tau$-regular semisimple elements admitting norms, $f_{v_1}$ is $F$-supercuspidal and matches an $F$-supercuspidal $\Phi_{v_1}$ and $f_{v_2}$ is supported on relatively $\tau$-regular elliptic semisimple elements and matches a function $\Phi_{v_2}$ supported on relatively  regular elliptic semisimple elements.
This is done in the following subsection.

\subsection{Proof of Theorem \ref{main-thm}: Nonvanishing}

We assume all of the notation and conventions of the previous section and the hypotheses of Theorem \ref{main-thm}.  In particular, $\Pi$ is both $G^{\sigma}$ and $G^{\theta}$-distinguished.  We prove the following proposition:
\begin{prop} With notation as in \S \ref{ssec-part1}, upon possibly enlarging $S$ we can choose a function $f_S^{\infty} \in C_c^{\infty}(G(F_{S}^{\infty}))$ so that
$$
TRT(\Pi(f_S^1\mathrm{ch}_{K^S})) \neq 0
$$
where $K^S \leq G(\A_F^S)$ is a hyperspecial subgroup, $f_{S(ME)}$ is supported on relatively $\tau$-regular semisimple elements admitting norms, $f_{v_1}$ is $F$-supercuspidal and matches an $F$-supercuspidal $\Phi_{v_1}$, and  $f_{v_2}$  is supported on
relatively $\tau$-regular elliptic semisimple elements and matches a function $\Phi_{v_2} \in C_c^{\infty}(U(F_{v_2}))$ supported on relatively regular elliptic semisimple elements.
\end{prop}

\noindent This proposition completes the proof of Theorem \ref{main-thm} as noted at the end of the previous subsection.

\begin{proof}
 We claim that we can choose a pure tensor $\phi_0 \in V_{\Pi}$ such that the restriction of the two forms
\begin{align*}
\mathcal{P}_{G^{\sigma}}(\cdot):V_{\Pi} \lto \CC \quad \textrm{ and } \quad
\mathcal{P}_{G^{\theta}}(\cdot):V_{\Pi} \lto \CC
\end{align*}
to $\phi^{S}_0 \otimes V_{\Pi,S}$ are nonzero.
Indeed, by assumption, there are vectors $\phi_1,\phi_2 \in V_{\Pi}$ such that $\mathcal{P}_{G^{\sigma}}(\phi_1) \neq 0$ and $\mathcal{P}_{G^{\theta}}(\phi_2) \neq 0$.  If either $\mathcal{P}_{G^{\sigma}}(\phi_2) \neq 0$ or $\mathcal{P}_{G^{\theta}}(\phi_1) \neq 0$ then we are done, otherwise $\mathcal{P}_{G^{\sigma}}(\phi_1+\phi_2)
\mathcal{P}_{G^{\theta}}(\phi_1+\phi_2) \neq 0$.  Thus there is a vector $\phi_0 \in V_{\Pi}$ such that $\mathcal{P}_{G^{\sigma}}(\phi_0)\mathcal{P}_{G^{\theta}}(\phi_0) \neq 0$.  We may assume that $\phi_0$ is a pure tensor, which implies the claim.

Enlarging $S$ if necessary, we choose a function
$f_{S}^{S_0(ME)} \in C_c^{\infty}(G(F_S^{S_0(ME)}))$ such that $\Pi(f_S^{S_0(ME)})$
is the orthogonal projection onto $\phi_S^{S_0(ME)}$; this is possible by the Jacobson density theorem and the matching statements Lemma \ref{lem-M-split-match} and Proposition \ref{prop-E-split-match}.

With this choice, for an orthonormal basis $\mathcal{B}(\Pi)$ of the $\Pi$-isotypic subspace of $L_0^2(G(F) \backslash {}^1G(\A_F))$ we have
\begin{align} \label{inftyv1-stuff}
&\sum_{\phi \in \mathcal{B}(\Pi)^{{K}^S}} \mathcal{P}_{G^{\sigma}}(\Pi'(f^1)\phi) \overline{\mathcal{P}_{G^{\theta}}(\phi)}=
\mathrm{tr}(\Pi(f^S)) \sum_{a_i}\mathcal{P}_{G^{\sigma}}(\phi^{S_0(ME)}_0 \otimes \Pi'(f_{S_0(ME)}^1)a_{i})
\mathcal{P}_{G^{\theta}}(\overline{\phi^{S_0(ME)}_0 \otimes a_{i}})
\end{align}
where the sum is over an orthonormal basis
$\mathcal{B}(\Pi)$ of $V_{\Pi_{S_0(ME)}}$ with respect to the Hermitian pairing
\begin{align*}
(\textrm{ },\textrm{ } ):V_{\Pi_{S_0(ME)}} \times V_{\Pi_{S_0(ME)}} &\lto \CC
\end{align*}
given by
\begin{align*}
(\psi_1,\psi_2) &\longmapsto \int_{G(F) \backslash {}^1G(\A_F)}
\phi_0(g^{S_0(ME)})\otimes \psi_1(g_{S_0(ME)} ) \overline{\phi_0(g^{S_0(ME)})
\otimes \psi_2(g_{S_0(ME)})} dg d\bar{g},
\end{align*}
where $dg$ is the Tamagawa measure.  Note that only finitely many of the $a_i$ terms will have a nonzero contribution to
\eqref{inftyv1-stuff}.

For the moment let $f_{S_0}$ be chosen so that $\Pi_{S_0}(f_{S_0})$ is the projection to the space spanned by $\phi_{0S_0}$.  Consider the linear functional
\begin{align}
\Theta:C_c^{\infty}(G(F_{S(ME)})) &\lto \CC\\
f_{S(ME)} &\longmapsto \sum_{a_i}\mathcal{P}_{G^{\sigma}}(\phi^{S_0}_0 \otimes \Pi(f_{S_0(ME)}^1)a_{i})
\mathcal{P}_{G^{\theta}}(\overline{\phi^{S_0}_0 \otimes a_{i}}). \nonumber
\end{align}
For each $v \in S(ME)$ this defines a spherical matrix coefficient of $\Pi_v$ in the sense of \S \ref{sec-sph}.  It is nonzero because we can choose $f_{S(ME)} \in C_c^{\infty}(G(F_{S(ME)}))$ such that $\Pi(f_{S(ME)})$ is the projection onto the space spanned by $\phi_{0S(ME)}$.  By assumption, $\Pi_v$ is relatively $\tau$-regular for $v \in S(ME)$, so we can and do choose $f_{S(ME)} \in C_c^{\infty}(G(F_{S(ME)}))$ that is supported in the set of $\tau$-regular semisimple elements in $G(F_{S(ME)})$ that admit norms
and such that $\Theta(f_{S(ME)}) \neq 0$.  This $f_{S(ME)}$ admits a matching function $\Phi_{S(ME)} \in C_c^{\infty}(H(F_{S(ME)}))$ by Theorem \ref{thm-weak-match} above.

We are left with choosing $f_{S_0}$.
Fix an isomorphism
\begin{align} \label{useful-isom}
G(F_{S_0}) \tilde{\lto} G^{\sigma}(F_{S_0}) \times G^{\sigma}(F_{S_0})
\end{align}
intertwining $\sigma$ with $(x,y) \mapsto (y,x)$ and $\tau$
with $(x,y) \mapsto (x^{\tau},y^{\tau})$.  By Lemma \ref{lem-2nd-stab} we can and do factor
$$
\Pi_{0S_0} \cong \Pi_{1S_0} \otimes \Pi_{1S_0}^{\vee}
$$
with respect to \eqref{useful-isom} for some irreducible admissible representation $\Pi_{1S_0}$
of $G^{\sigma}(F_{S_0})$.  The map
\begin{align*}
\ell_1:\mathcal{H}_{\Pi_{0S_0}} &\lto \CC\\
\Pi_{1v}(f_1) \otimes \Pi_{1v}^{\vee}(f_2) &\longmapsto \prod_{v \in S_0}\mathrm{tr}\left(\Pi_{1v}(f_{1v}*f_{2v}^{-\tau} )\right)
\end{align*}
is $(G^{\sigma}(F_{S_0}),G^{\theta}(F_{S_0}))$-invariant in the sense of \S \ref{ssec-preps}, and clearly not identically zero.

Notice the linear functional
\begin{align*}
\ell_{2}: \mathcal{H}_{\Pi_{S_0}} &\lto \CC\\
\Pi(f_{S_0}) &\longmapsto \sum_{a_i}\mathcal{P}_{G^{\sigma}}(\phi^{S_0}_0 \otimes \Pi(f_{S_0(EF)}^1)a_{i})
\mathcal{P}_{G^{\theta}}(\overline{\phi^{S_0}_0 \otimes a_{i}})
\end{align*}
is also $(G^{\sigma}(F_{S_0}),G^{\theta}(F_{S_0}))$-invariant.  By Lemma \ref{lem-lforms} we have
that
$$
\ell_{2}=c\ell_1
$$
for some constant $c \in \CC$.  By our choice of $f_{S(ME)}$ above the linear functional $\ell_2$ is not identically zero and hence we have that $c \neq 0$.

Choosing a factorable function
$$
f_{S_0}^{\infty}=f_{1} \times f_{2} \in C_c^{\infty}(G^{\sigma} \times G^{\sigma}(F_{S_0}^{\infty})),
$$
we have
\begin{align} \label{nonzero}
\ell_2(f_{S_0})=
c\left(\mathrm{tr}(\Pi_{1\infty}(f_{L,\mathrm{Res}_{M_{\infty}/F_{\infty}}V_1,\tau} \circ \tau )) \right)
\left(\prod_{v \in S_0 \setminus \infty}\mathrm{tr}\left(\Pi_{1v}(f_{1v}*f_{2v}^{-\tau} \circ \tau)\right)\right)
\end{align}
with $V_1$ as in \S \ref{ssec-part1}.
 We now show that we can choose a
test function $f=f_{1} \times f_2$ satisfying the conditions at $\infty$, $v_1$, and $v_2$ stipulated in the statement
of the proposition  such that \eqref{nonzero} is
nonzero; this will complete the proof of the proposition.

We work place by place.  The factor corresponding to the infinite places is nonzero by \cite[Lemme 4.2]{LabCM}.  Lemma \ref{lem-supercusp} takes care of $v=v_1$, and Lemma \ref{lem-sqint}
takes care of $v=v_2$.
\end{proof}

\section*{Acknowledgments}
The first named author appreciates the time Erez Lapid spent pointing out serious mistakes in
an earlier version of this work and for thoughtful comments as the new version emerged.
He would like to thank Chris Skinner for support throughout this project, and Herv\'e Jacquet, David Whitehouse, and Sophie Morel for useful conversations.  He also thanks  Heekyoung Hahn for help with editing.  The second author thanks Dinakar Ramakrishnan, Elena Mantovan, Kimball Martin and David Whitehouse for helpful conversations and advice.  Both authors
thank the referees for thorough readings of the paper and their suggestions and corrections, some of which were incorporated.


\printindex

\end{document}